\documentclass[french,11pt]{smfartVSautodual}
\usepackage[latin1]{inputenc}
\usepackage[T1]{fontenc}
\usepackage{lmodern}
\usepackage{smfenum,smfthm,amssymb,mathrsfs,euscript,color}
\usepackage{stmaryrd,mathabx,upgreek,mathdots,mathtools,amscd}
\input{xypic}
 
\numberwithin{equation}{section} 
\theoremstyle{plain}

\def\CC{\mathbb{C}}

\def\QQ{\mathbb{Q}}

\def\ZZ{\mathbb{Z}} 
\def\FF{\mathbb{F}} 

\def\A{{ A}}
\def\B{{ B}}
\def\C{{ C}}
\def\D{{ D}}
\def\E{{ E}}
\def\F{{ F}}
\def\G{{ G}}
\def\H{{ H}}

\def\J{{\rm J}}
\def\K{{ K}}
\def\L{{ L}}

\def\N{{\rm N}}

\def\SS{{ S}}
\def\T{{ T}}
\def\U{\boldsymbol{\rm U}}
\def\V{{ V}}
\def\W{{ W}}

\def\Cc{\EuScript{C}}

\def\Ee{\EuScript{E}}

\def\Ll{\mathscr{L}}

\def\Oo{\EuScript{O}}
\def\Pp{\EuScript{P}}

\def\Ww{\EuScript{W}}

\def\Om{\boldsymbol{\Omega}}

\def\a{\alpha} 
\def\b{\beta}
\def\cc{c}
\def\d{\delta}
\def\e{\varepsilon}

\def\g{\gamma}
\def\h{\varphi}
\def\k{\kappa}
\def\l{\lambda}
\def\n{\eta}

\def\p{\mathfrak{p}}
\def\s{\sigma}
\def\t{\theta}
\def\v{\upsilon}
\def\w{\varpi}

\def\ie{c'est-à-dire }

\def\>{\geqslant}
\def\<{\leqslant}

\def\Hom{{\rm Hom}}
\def\End{{\rm End}}

\def\Mat{\boldsymbol{{\sf M}}}
\def\GL{{\rm GL}}
\def\Gal{{\rm Gal}}
\def\Ker{{\rm Ker}}

\def\tr{{\rm tr}}

\def\Ind{{\rm Ind}}
\def\ind{{\rm ind}}
\def\mult#1{{#1}^{\times}}

\def\kk{\boldsymbol{k}}

\def\ee{\boldsymbol{l}}

\def\tt{\boldsymbol{t}}

\def\bk{\boldsymbol{\k}}
\def\bl{\boldsymbol{\l}}
\def\bt{\boldsymbol{\rho}}
\def\bm{\boldsymbol{\mu}}
\def\bx{\boldsymbol{\xi}}

\def\BJ{\boldsymbol{\J}}
\def\BU{\boldsymbol{\rm U}}
\def\TT{\boldsymbol{\Theta}}
\def\0{\boldsymbol{0}}

\def\aa{\mathfrak{a}}

\def\bb{\mathfrak{b}}

\def\qq{\boldsymbol{\omega}}
\def\pp{\mathfrak{p}}

\def\tdt{\times\dots\times}

\def\GG{\EuScript{J}}

\def\Ad{{\rm Ad}}
\def\dd{\kappa}

\def\ss{\sigma}

\def\FC{\CC}
\def\ep{{\sf e}}
\def\EP{{\sf w}}

\def\bo{{\omega}}
\def\iso{\EuScript{P}}
\def\BJI{\BJ}

\def\Nrd{{\rm Nrd}}

\def\AA{\boldsymbol{\sf A}}

\def\SD{\boldsymbol{\Delta}}

\long\def\MSC#1\EndMSC{\def\arg{#1}\ifx\arg\empty\relax\else
  {\par\narrower\noindent%
     2010 Mathematics Subject Classification: #1\par}\fi}

\long\def\KEY#1\EndKEY{\def\arg{#1}\ifx\arg\empty\relax\else
	{\par\narrower\noindent Keywords and Phrases: #1\par}\fi}

\title[Représentations cuspidales distinguées de $\GL_r(\D)$]
{Représentations cuspidales de $\GL_r(\D)$
distinguées par une involution intérieure}

\author{Vincent Sécherre}
\address{Laboratoire de Mathémati\-ques de Versailles\\
UVSQ\\
CNRS\\
Université Paris-Saclay\\
78035, Versailles, France\\
Institut Universitaire de France}

\email{vincent.secherre@uvsq.fr}

\begin{abstract}
Soit un entier $n\>1$,
soit $F$ un corps localement compact non archimédien de~carac\-téristique 
résiduelle $p\neq2$ et soit $G$ une forme intérieure de $\GL_{2n}(F)$. 
C'est un groupe de~la for\-me $\GL_r(D)$ pour un entier $r\>1$ et une
$F$-algèbre à division $D$ de degré réduit $d$ tel 
que~$rd=2n$.~Soit 
$K$~une exten\-sion qua\-dra\-ti\-que de $F$ dans
l'algèbre des matrices de taille $r$ à coefficients 
dans~$D$,~et soit $H$ son centralisateur dans $G$.
Nous étudions les représentations~cuspi\-dales autoduales 
comple\-xes~de~$G$ 
et leur distinction par $H$
du point de vue de la théorie des types.
Si~$\pi$ est une telle~re\-pré\-sentation et
si $\phi$ est son paramètre de Langlands,
nous calculons la valeur~en~$1/2$ du~facteur~ep\-silon de
la restric\-tion de $\phi$ au groupe de Weil-Deligne de $K$,
notée $\ep_K(\phi)$.
Lorsque $F$ est~de caractéristique nulle, 
nous en déduisons que $\pi$ est distinguée par~$H$
si et seulement si $\phi$ est de parité sym\-plec\-tique et 
$\ep_K(\phi)=(-1)^r$.
Ceci prouve dans ce cas 
une conjecture de Prasad et Takloo-Big\-hash.
\end{abstract}

\begin{altabstract}
Let $n$ be a positive integer,
$F$ be a non-Archimedean locally compact field of residue characteristic 
$p\neq2$ and $G$ be an inner form of $\GL_{2n}(F)$. 
This is a group of the form $\GL_r(D)$ for a positive integer $r$ and division
$F$-algebra $D$ of reduced degree $d$ such that $rd=2n$.
Let $K$ be~a quadratic exten\-sion of $F$ 
in the algebra of matrices of size $r$ with coefficients in~$D$,
and $H$ be its centralizer in $G$.
We study selfdual cuspidal complex 
representations of $G$ and their distinction by $H$ 
from the point of view of type theory. 
Given such a representation $\pi$, 
we compute the value at $1/2$ of the epsilon factor of the restriction of the
Langlands parameter $\phi$ of $\pi$ to the Weil-Deligne group of 
$K$,~denoted $\ep_K(\phi)$.
When $F$ has characteristic $0$,
we deduce that $\pi$ is $H$-distinguished 
if and only if~$\phi$~is symplectic and $\ep_K(\phi)=(-1)^r$.
This proves in this case a conjecture by Prasad--Tak\-loo-Big\-hash.
\end{altabstract}

\begin{document} 

\maketitle

{\footnotesize
\MSC 
22E50 
\EndMSC
\KEY 
Cuspidal representation, Distinguished representation,
Endo-class,
Root number,
Symplectic parameter, Type theory 
\EndKEY
}

\bigskip

\hfill{\textit{\`A la mémoire de Colin Bushnell}}

\thispagestyle{empty}

\section{Introduction}

\subsection{}
\label{intro1}

Soit un entier $n\>1$,
soit $F$ un corps localement compact non archimédien de carac\-téris\-ti\-que 
résiduelle $p$
et soit $A$ une $F$-algèbre centrale simple de degré
réduit $2n$.
C'est une algèbre de la forme $\Mat_r(D)$ pour un entier $r\>1$ et une
$F$-algèbre à division $D$ de degré réduit $d\>1$ tels que
$rd=2n$.
Notons $G=A^\times=\GL_r(D)$ le groupe des éléments inversibles de $A$.
Un tel groupe est une forme intérieure de $\GL_{2n}(F)$, 
et la théorie des représentations (lisses, complexes)~de~$G$~est liée à celle de 
$\GL_{2n}(F)$ par 
la correspondance de Jacquet-Langlands \cite{JL,Rog,DKV,Badu}, 
qui est une bijection entre clas\-ses~d'isomorphisme de
représentations essentiellement de carré intégrable~de $G$ et $\GL_{2n}(F)$, 
caractérisée par une identité de caractères sur les classes de
conjugaison~el\-lip\-ti\-ques~régulières.  

Soit $K$ une extension quadratique de $F$ incluse dans $A$,
soit $H$ son centralisateur dans $G$
et soit $\mu$ un carac\-tère de $K^\times$.
On note $\Nrd_{H}$ la norme réduite de $H$ dans $K^\times$.
Une représentation irréductible $\pi$ de $G$
sur un espace vectoriel (complexe) $V$ est dite 
$\mu$-distinguée si $V$ admet une forme linéaire non nulle $\Ll$ telle que :
\begin{equation*}
\Ll(\pi(h)v)=\mu(\Nrd_{H}(h)) \cdot v,
\quad
h\in H,
\quad
v\in V,
\end{equation*} 
\ie si l'espace $\Hom_H(\pi,\mu\circ\Nrd_{H})$ est non nul.  
Dans~\cite{PTB}
Prasad et Takloo-Bighash~ont formulé la conjecture suivante.
On verra $\mu$ indifféremment com\-me un caractère de $K^\times$ ou du 
groupe de Weil~de $K$ \textit{via} l'homomorphisme de réciprocité de la
théorie du corps de classes local. 

\begin{conj}
\label{CONJPTB}
Soit $\pi$ une représentation irréductible de $G$ dont le transfert
à $\GL_{2n}(F)$ soit générique.
Supposons que la restriction de $\mu^{n}$ à $F^\times$
soit égale au ca\-ractère central de $\pi$.~No\-tons
$\phi$ le paramètre de Langlands de $\pi$
et $\phi_K$ sa~res\-triction au groupe de Weil-Deligne~de $K$.
\begin{enumerate}
\item 
Si la représentation $\pi$ est $\mu$-distinguée, alors~:
\begin{enumerate}
\item 
le paramètre de Langlands $\phi$ est à valeurs dans le groupe
${\rm GSp}_{2n}(\CC)$ des similitudes symplectiques 
et son facteur de similitude est égal à la restriction de $\mu$ à $\F^\times$, 
\item
la valeur en $1/2$ du facteur epsilon de
$\phi_K\cdot\mu^{-1}$ est égale à $(-1)^r\mu(-1)^{n}$.
\end{enumerate}
\item
Si la représentation $\pi$ est essentiellement de carré intégrable,
alors elle est $\mu$-distinguée~si et seulement si les con\-ditions
{\rm (1.a)} et {\rm (1.b)} sont satisfaites. 
\end{enumerate}
\end{conj}

Expliquons ce que signifie la condition de généricité portant sur $\pi$.
D'après \cite{ZelAENS,Tadic,BHLS},
il~y~a des entiers $r_1,\dots,r_k\>1$ 
et des représentations essentiellement de carré intégrable 
$\d_1,\dots,\d_k$~de $\GL_{r_1}(D),\dots,\GL_{r_k}(D)$ 
tels que $\pi$ soit isomorphe à l'unique quotient irréductible de
l'induite~pa\-rabolique normalisée de $\d_1\otimes\dots\otimes\d_k$
prise par rapport au sous-groupe parabolique
triangulai\-re~supérieur par blocs.
Chaque $\d_i$ a un transfert de Jacquet-Langlands noté $\s_i$, 
qui est une repré\-sen\-ta\-tion~essentiellement de carré intégrable 
de $\GL_{r_id}(F)$. 
L'induite parabolique normalisée de $\s_1\otimes\dots\otimes\s_k$
admet un unique quotient irréductible,
qui est le transfert de $\pi$,
et ce transfert est générique si cette induite est irréductible.

\subsection{}

Cette conjecture est inspirée de résultats dus à Tunnel \cite{Tunnel} et à Saito 
\cite{Saito} dans le cas où $n$ et $d$ sont égaux à $1$,
et où $F$ est soit de caractéristique nulle, soit de caractéristique $p$ impaire.
Dans le cas où $\mu$ est trivial,
Guo \cite{GuoPJM97} prouve que, si
$F$ est de ca\-ractéristique nulle et si $d\<2$,
l'es\-pa\-ce 
$\Hom_H(\pi,\CC)$ est de dimension au plus $1$
pour toute représentation irréductible $\pi$ de $G$
et que, si cette dimension est non nulle,
\ie si $\pi$ est $H$-distinguée,
alors $\pi$ est autoduale. 
Dans~\cite{PTB}, Prasad et Takloo-Bighash prouvent~leur conjecture
lorsque $n=2$.
Leur preuve repose
en partie sur des résultats de Gan-Takeda \cite{GanTakedaANNALS}, 
qui supposent que $F$ est de caracté\-ris\-ti\-que nulle.
Il y a eu récemment plusieurs résultats en direction de la preuve de la
conjecture \ref{CONJPTB}. 

Quand $\mu$ est trivial et $F$ de caractéristique nulle,
Feigon-Martin-Whitehouse \cite{FMW} prouvent la con\-jec\-ture \ref{CONJPTB}(1)
pour les représentations cuspidales de $\GL_{2n}(F)$.

Chommaux \cite{Chommaux} prouve la~con\-jecture \ref{CONJPTB}(2)
lorsque $\pi$ est la représentation de Steinberg de $G$ tordue
par un caractère et $F$ est de ca\-ractéristique différente de $2$.
Puis Chommaux et Matringe \cite{ChommauxMatringe}
prouvent la même conjecture 
dans le cas où $\pi$ est une~re\-pré\-sentation cuspidale
de niveau $0$ de $\GL_{2n}(F)$ et où $p\neq2$.

Broussous et Matringe \cite{BroussousMatringe}
étendent le théorème~de multiplicité $1$ de Guo
au cas où l'entier~$d$ est quel\-conque et $F$ est de caractéristique
différente de~$2$.~Ils
étendent aussi
au cas d'une forme intérieure quelconque
le théorème d'autodualité de Guo~; 
leur argument repose sur des résultats 
(\cite{GuoPJM97,JR}) supposant que $F$ est de ca\-ractéristique nulle,
mais on trouve dans 
\cite{ChommauxMatringe}
un ar\-gu\-ment valable dès que $F$ est de caractéristique
différente de $2$.

Dans le cas où $\mu$ est trivial et $F$ de caractéristique différente de $2$, 
Suzuki \cite{Suzuki} étend la~con\-jecture \ref{CONJPTB}(1) à toutes les
représentations irréductibles, 
sans hypothèse de généricité,~et
ramène la preuve de cette conjecture 
au cas des représentations essentiellement de carré intégrable.

Enfin, dans le cas où $\mu$ est trivial et $F$ de ca\-ractéristique nulle,
Xue \cite{Xue} prouve d'une part~la con\-jec\-ture \ref{CONJPTB}(1),
d'autre part la con\-jecture \ref{CONJPTB}(2) dans les cas suivants~:
(i) pour les représenta\-tions~cus\-pidales
dont le transfert de Jacquet-Langlands à $\GL_{2n}(F)$ est cuspidal,
(ii) pour toutes les~re\-pré\-senta\-tions cuspidales si $d\<2$~;
et Suzuki et Xue \cite{SuzukiXue} ont annoncé que
la preuve de la con\-jectu\-re~\ref{CONJPTB}(2) 
se ramène au cas des représentations cuspidales.

\subsection{}
\label{chap13}

Dorénavant,
nous supposerons que le caractère $\mu$ de la conjecture
\ref{CONJPTB} est trivial.
Dans~cet ar\-ticle, 
nous prouvons le théorème suivant (voir le théorème \ref{PTB}). 

\begin{theo}
\label{theo12intro}
Supposons que $p\neq2$,
et que la condition (1.a) de la conjecture \ref{CONJPTB}
soit~vé\-ri\-fiée par toute re\-présentation cuspidale $H$-distinguée de $G$. 
Alors la conjecture \ref{CONJPTB}(2) est vraie
pour~toute~re\-présentation cuspidale de $G$.
\end{theo}

Dans le cas où $F$ est de caractéristique nulle et de caractéristique 
résiduelle $p\neq2$,
on en~dé\-duit~grâce à \cite{Xue}
le résultat suivant (voir le corollaire \ref{PTBcoro}). 

\begin{coro}
\label{biper}
Supposons que $F$ soit de caractéristique nulle et que $p\neq2$. 
Alors la conjecture \ref{CONJPTB}(2) est vraie
pour~toute~re\-présentation cuspidale de $G$.
\end{coro}

Par conséquent,
compte tenu des résultats présentés dans le paragraphe précédent,
le corollai\-re~\ref{biper} complète la preuve de la conjecture \ref{CONJPTB} 
dans le cas où $F$ est de caractéristique nulle et de caractéristique
résiduelle $p\neq2$ (et où le caractère $\mu$ est trivial).

Notre approche, complètement différente de celle de \cite{Xue},
repose sur la description des~repré\-sen\-ta\-tions cuspidales des formes 
intérieures de $\GL_{2n}(F)$ par la théorie des types.
On renvoie au paragraphe \ref{intro15intro} pour une discussion de
l'hypothèse ``$p\neq2$'' dans ce théorème et son corollaire.

\subsection{}

Bushnell et Kutzko ont montré dans \cite{BK} que,
quels que soient l'entier $n$ et le corps loca\-le\-ment compact non archimédien 
$F$,
toute représentation cuspidale de $\GL_n(F)$ 
s'obtient par~in\-duction compacte
d'une représentation d'un sous-groupe ouvert et compact modulo le
centre de $\GL_n(F)$.
Ce résultat a ensuite été étendu aux représentations 
cuspidales de n'importe quelle forme inté\-rieure de $\GL_n(F)$
(voir \cite{BroussousGL1,VS1,VS2,SeStJIMJ}). 
Plus précisément,
notant temporairement $G$ une~for\-me~intérieure de $\GL_n(F)$
(et non pas de $\GL_{2n}(F)$ comme dans les paragraphes
précédents),~il~y a
une famille de paires $(\BJ,\bl)$ formées d'un sous-groupe
$\BJ$ ouvert compact modulo le centre de $\G$ et d'une
représentation $\bl$ de $\BJ$ telle que~:
\begin{itemize}
\item 
pour toute paire $(\BJ,\bl)$, 
l'induite compacte de $\bl$ à $\G$ soit irréductible et cuspidale~;
\item
toute représentation (irréductible) cuspidale de $\G$ s'obtienne ainsi
pour une paire $(\BJ,\bl)$~uni\-que à $\G$-conjugaison près. 
\end{itemize}
Ces paires $(\BJ,\bl)$ sont appelées des
\textit{types simples maximaux étendus} de $G$,
ce que nous abrégerons en \textit{types} dans cet article. 
\'Etant donné une représentation cuspidale $\pi$ de $\G$ et un type
$(\BJ,\bl)$ lui correspondant,
ainsi qu'un sous-groupe fermé $H$ de $G$, 
une simple application de la formule de Mackey montre qu'on a un
isomorphisme d'espaces vectoriels~:
\begin{equation}
\label{girafeintro}
\Hom_H(\pi,\CC) \simeq\prod\limits_{g}\Hom_{\BJ^g\cap H}(\bl^g,\CC)
\end{equation}
où $g$ décrit un système de représentants de $(\BJ,H)$-doubles classes de 
$G$.
Ainsi, pour étudier la distinction de $\pi$ par $H$,
il suffit d'étudier la distinction de $\bl^g$ par $\BJ^g\cap H$ pour chaque 
$g$. 

\subsection{}
\label{intro15intro}

Cette approche a été utilisée par Hakim et Murnaghan
(voir par exemple \cite{HMIMRN,HMCMAT02,HakimRT})~pour
étudier la distinction par divers sous-groupes
d'une certaine classe de représentations cuspidales de $\GL_n(F)$, 
dites \textit{essentiellement modérées}.
Il s'agit des représentations cuspidales dont le para\-mè\-tre de Langlands
contient un caractère du sous-groupe d'inertie sauvage de $F$.
La construction de ces re\-présentations par induction compacte est plus 
simple que dans le cas non~es\-sen\-tiellement modéré,
et peut être décrite au moyen des paires admissibles
(\cite{Howe,BHETLC1}) de Howe. 

Si l'on s'intéresse aux représentations cuspidales générales de $\GL_n(F)$ 
et non pas seulement à celles qui sont essentiellement modérées,
les paires admissibles ne suffisent pas~:
il faut consi\-dé\-rer direc\-tement les types $(\BJ,\bl)$ mentionnés au 
paragraphe précédent.
Dans cette situation,~\cite{VSANT19} fournit une analyse com\-plète
de la distinction
par $\GL_n(F_0)$ des représentations irréductibles~cus\-pi\-dales de
$\GL_n(F)$
pour une extension quadratique $F/F_0$ de corps localement compacts
non~ar\-ch\-imédiens, \textit{lorsque $p\neq2$}.
Les mé\-tho\-des développées dans \cite{VSANT19} 
ont été adaptées par Jiandi Zou
(\cite{ZouU,ZouO}) 
aux involutions uni\-taires et orthogonales de $\GL_n(F)$,
toujours lorsque $p\neq2$. 

Une restriction inhérente à la méthode générale employée pour aborder
tous ces résultats
est que la ca\-rac\-téristique résiduelle $p$ doit être supposée impaire,
notamment parce qu'on utilise~le
fait que le premier ensemble de cohomologie de $\ZZ/2\ZZ$ à valeurs
dans~un pro-$p$-groupe est trivial.
(Voir aussi le paragraphe \ref{voiraussi110}
pour une utilisation de l'hypothèse $p\neq2$
dans le cas de ce travail.)

\begin{center}
$\bullet$
\end{center}

\textit{Dans cette introduction,
nous~sup\-po\-se\-rons dorénavant que $p\neq2$.}

\subsection{}
\label{intro8}

Revenons maintenant à la situation introduite au paragraphe \ref{intro1},
en ayant supposé que~$p$ est impair et que $\mu$ est trivial.
En particulier, $G$ est une forme intérieure de
$\GL_{2n}(F)$.~Rappe\-lons~que,
selon \cite{GuoPJM97,BroussousMatringe},
toute représentation cus\-pidale $H$-distinguée de $G$ est
autoduale,~c'est-à-dire isomorphe à sa contragrédiente,
de sorte que nous ne nous intéresserons par la suite qu'aux
représentations cus\-pidales de $G$ qui sont autoduales.

Notre preuve du théorème \ref{theo12intro} repose sur
un argument de compta\-ge~:
sous l'hypothèse que la condition (1.a) de la conjecture \ref{CONJPTB}
soit~vé\-ri\-fiée par toute re\-présentation cuspidale $H$-distinguée de $G$, 
nous prouvons qu'il existe pour $G$ autant de représentations cus\-pidales
$H$-distinguées~que de repré\-sen\-tations cuspidales autoduales
satisfaisant aux conditions (1.a) et (1.b) de la conjec\-tu\-re~\ref{CONJPTB}.
Bien sûr, 
les représentations cus\-pidales au\-to\-duales de $G$ sont en nombre
infini.\footnote{Par exem\-ple,~si $P_0$
est une extension modérément ramifiée de degré $n$ de $F$
et $P$ une extension~qua\-dratique~de $P_0$, les
caractères de $P^\times$ qui sont 
admissibles au sens du paragraphe \ref{intro15intro}
et triviaux sur $\N_{P/P_0}(P^\times)$
para\-mè\-trent une infinité de représentations cus\-pidales
au\-to\-duales de $\GL_{2n}(F)$~: voir \cite{HMIMRN}.}~Pour
se ra\-me\-ner à des ensembles finis,
on peut borner le niveau~des~re\-pré\-sentations~;
mais il~est~beau\-coup~plus éclairant d'introduire la notion d'endo-classe.

\subsection{}
\label{endoc}

Une endo-classe (sur $F$)
est un invariant associé, par la théorie des types,
à toute représentation essentiellement de carré intégrable d'une forme
intérieure d'un groupe linéaire général~sur $F$
(\cite{BHLTL1,BSS}).
La définition générale de cet invariant requiert une machinerie
considérable.
Ce\-pen\-dant,
il a une interprétation arithmétique simple \textit{via} la correspondance de
Langlands locale (\cite{BHLTL4,SeStLinked,Dotto})~:
deux représentations essentiellement de carré intégrable
de formes intérieures de groupes linéaires généraux sur $F$
ont la même endo-classe si et seulement si leurs paramètres~de Langlands,
une fois restreints au sous-groupe d'inertie sauvage $\Pp_F$
du groupe de Weil $\Ww_F$ de $F$ 
relati\-vement à une clôture séparable $\overline{F}$ de $F$,
ont un facteur en commun.
La
correspondance de Langlands locale induit alors une bijection entre classes
de $\Ww_F$-conjugaison de re\-présentations ir\-réductibles de $\Pp_F$
et $F$-endo-classes.
L'endo-classe est un invariant plus fin que le niveau~nor\-ma\-li\-sé,
et il n'y a qu'un nombre fini d'endo-classes de niveau normalisé
fixé. 

\subsection{}
\label{p19intro}

Reprenons la suite du paragraphe \ref{intro8},
et fixons une endo-classe $\TT$,
à laquelle on peut~pen\-ser pour le moment comme à une classe 
de $\Ww_F$-conjugaison de représentations irréductibles de~$\Pp_F$.
Supposons dans cette introduction que $\TT$ soit non nulle,
\ie qu'elle ne corresponde pas au caractère trivial de $\Pp_F$.
(Le cas de l'endo-classe nul\-le
est un peu moins direct mais se
traite es\-sen\-tiellement de la même façon.)
Notons $\AA(G,\TT)$ l'ensemble des classes
d'isomorphisme
de représentations cuspidales autoduales de $G$ d'endo-classe
$\TT$, 
et notons $\AA^{+}(G,\TT)$ le sous-en\-sem\-ble de celles qui sont
distinguées par $H$. 
Notons également~:
\begin{itemize}
\item
$\AA^{{\rm sp}}(G,\TT)$ le sous-ensemble de $\AA(G,\TT)$ formé des 
classes de représentations dont le~para\-mè\-tre de Langlands est 
symplectique (voir le paragraphe \ref{banquets}),
\item
$\AA^{{\rm ptb}}(G,\TT)$ le sous-ensemble de $\AA^{{\rm sp}}(G,\TT)$
formé des classes de représentations satisfaisant aux conditions
(1.a) et (1.b) de la conjecture \ref{CONJPTB}. 
\end{itemize}
L'ensemble $\AA(G,\TT)$ est fini,
et nous prouvons en nous inspirant de \cite{BHS}
qu'il est formé d'autant de représentations de parité symplectique~que de 
re\-présentations de parité orthogonale
(voir les lemmes
\ref{lemmedeparitegeneralise} et
\ref{lemmedeparitegeneraliseinnerform}).
Le cardinal de $\AA^{{\rm sp}}(G,\TT)$ est donc la moitié de celui de
$\AA(G,\TT)$.~Par~hy\-po\-thèse,
nous supposons dans le théorème \ref{theo12intro} que l'on a l'inclusion~:
\begin{equation}
\label{incl1intro}
\AA^{+}(G,\TT) \subseteq \AA^{{\rm sp}}(G,\TT).
\end{equation}
Une simple application d'une formule de Bushnell-Henniart \cite{BHCJM01}
dé\-cri\-vant le~com\-portement~du facteur epsilon d'une représentation 
cuspidale de niveau non nul par torsion par
un caractère~mo\-dérément ramifié
(voir la proposition \ref{FormulePourEKPi})
montre que,
pour une repré\-sen\-tation $\pi$ dans $\AA(G,\TT)$, 
la valeur du facteur epsilon apparaissant dans
la conjecture~\ref{CONJPTB} prend la forme~:
\begin{equation}
c_\pi(-1) \cdot \EP_K(\TT)
\end{equation}
où $c_\pi$ est le caractère central de $\pi$
et $\EP_K(\TT)$ est un signe ne dépendant
ni du choix de $\pi$ ni de la forme intérieure con\-sidérée
mais~\textit{uni\-quement de~$\TT$}.
Ceci a pour conséquence que
$\AA^{{\rm ptb}}(G,\TT)$~est soit égal à $\AA^{{\rm sp}}(G,\TT)$, 
soit vide.
(Il y a une exception quand $\TT$ est nulle~: voir la remarque~\ref{Parlement}.)
Notre stratégie consiste à montrer~:
\begin{itemize}
\item 
d'une part que l'ensemble $\AA^{+}(G,\TT)$ est vide si et seulement si
$\AA^{{\rm ptb}}(G,\TT)$ l'est, \ie si et seu\-lement si
$\EP_K(\TT)\neq(-1)^r$,
\item 
d'autre part que,
s'il est non vide,
son cardinal est au moins moitié de celui de $\AA(G,\TT)$. 
\end{itemize}
(On observera que le caractère central $c_\pi$ est trivial pour les 
représentations $\pi$ dans $\AA^{{\rm sp}}(G,\TT)$.) 
Il nous faut maintenant comprendre à quelles conditions $\AA^{+}(G,\TT)$ et 
$\AA^{{\rm ptb}}(G,\TT)$ sont vides. 

\subsection{}
\label{voiraussi110}

Pour déterminer à quelles conditions l'ensemble $\AA^{{\rm ptb}}(G,\TT)$
est vide, \ie pour~cal\-cu\-ler le signe $\EP_K(\TT)$,
il nous faut introduire deux invariants supplémentaires. 
Soit $\pi$ une~re\-pré\-sentation cuspidale autoduale de $G$
d'endo-classe $\TT$,
et soit $\g$ un facteur irréductible de la res\-triction à
$\Pp_F$ du paramètre de Langlands de $\pi$,
qui n'est pas trivial puisque $\TT$ a été supposée non nulle. 
(D'après le paragraphe \ref{endoc}, 
la classe~de $\Ww_F$-conjugaison de $\g$ et $\TT$
se déterminent l'une l'autre.)
Le sta\-bi\-li\-sa\-teur de $\g$ dans $\Ww_F$ est un~sous-groupe
de $\Ww_F$
égal à $\Ww_T$~pour une unique extension modérément ramifiée
$T$ de~$F$ dans $\overline{F}$.
On pose~:
\begin{equation}
\deg(\TT) = \dim(\g) \cdot [T:F],
\end{equation}
qu'on appelle le \textit{degré} de $\TT$. 
De même,
le sta\-bi\-li\-sa\-teur dans $\Ww_F$
de la somme directe de $\g$~et~de sa contragrédiente $\g^\vee$
est égal à $\Ww_{T_0}$ pour une unique extension $T_0$ de $F$ contenue dans
$T$.
La représentation $\pi$ étant auto\-dua\-le, et comme $p\neq2$, 
les représentations $\g^\vee$ et $\g$ sont conjuguées mais pas isomorphes, 
ce qui entraîne que l'extension $T/T_0$ est quadratique.
Celle-ci ne dépend, à $F$-isomorphisme près, que de $\TT$
(voir le lemme \ref{connemara}).
(Si $\TT$ est nulle,
\ie si $\g$ est le caractère trivial de $\Pp_F$,
on a $T=T_0=F$.) 

Fixons maintenant un élément $\dd\in\K$ engendrant~$K$ sur $F$,
et tel que $\a=\dd^2\in\F^\times$.
L'auto\-morphisme de conjugaison $\Ad(\dd)$, noté $\tau$,
est une in\-vo\-lu\-tion sur $G$,
dont le sous-groupe $G^\tau$~des points fixes est égal à $H$.
Nous prouvons le résultat suivant
(voir le théorème \ref{IsabelledeFrejusNS}). 

\begin{theo}
\label{valeureKTTintro}
Soit $\TT$ l'endo-classe d'une représentation cuspidale autoduale
de niveau non nul de $G$. 
Alors~: 
\begin{equation}
\EP_K(\TT) = \bo_{T/T_0}(\a)^{2n/[T:F]} 
\end{equation}
où $\bo_{T/T_0}$ est le caractère de $T_0^\times$ de noyau 
$\N_{T/T_0}(T^\times)$. 
\end{theo}

Le cas des représentations de niveau $0$ nécessite un traitement à part~:
voir le paragraphe \ref{Banaan}. 

\subsection{}

Pour aller plus loin,
on ne peut plus se contenter de l'interprétation d'une
endo-classe
comme classe de $\Ww_F$-conjugaison de représentations irréductibles
de $\Pp_F$,
et il faut introduire la notion de caractère simple,
qui est au coeur de la théorie des types. 

Les caractères simples sont des caractères bien particuliers de 
pro-$p$-sous-groupes ouverts~bien particuliers de $G$,
construits dans \cite{BK,BroussousGL1,Grabitz,VS1}.
L'ensemble des caractères simples~des~formes in\-té\-rieures des groupes
linéaires 
généraux sur $F$ possède des propriétés remarquables d'entrela\-cement et de
transfert d'un groupe à l'autre,
subsumées dans la définition d'une relation~d'équi\-valence sur 
cet ensemble appelée \textit{endo-équivalence}
(voir le paragraphe \ref{simplechar}).
Une endo-classe~est alors une classe d'équi\-valence de caractères simples 
pour cette relation.
Une représentation~cus\-pi\-dale $\pi$ de $G$ contient,
à conjugaison près par $G$,
un unique caractère simple~;
sa classe d'endo-équivalence est l'endo-classe de $\pi$. 

\subsection{}

Dans l'étude des représentations cuspidales de $\GL_n(F)$ distinguées par
$\GL_n(F_0)$ effec\-tuée dans \cite{VSANT19},
un résultat fondamental est l'existence,
dans toute représentation cuspidale~\textit{$\s$-autoduale} de $\GL_n(F)$,
\ie toute représentation cuspidale $\pi$ dont la contragrédiente
$\pi^\vee$ est~iso\-morphe à la conjuguée $\pi^\s$ par l'automorphisme
non trivial $\s\in\Gal(F/F_0)$,
d'un caractère sim\-ple $\s$-au\-todual (\cite[Theorem 4.1]{AKMSS}),
\ie un caractère simple $\t$ tel que $\t^\s=\t^{-1}$. 
Dans l'étu\-de des représentations cuspidales autoduales de $\GL_{2n}(F)$,
il suit des travaux de Blondel \cite{BlondelAENS04} qu'un bon
analogue consiste à choisir
pour $\s$ un élément de $\GL_{2n}(F)$ tel que $\s^2=1$,
et dont le~polynôme ca\-ractéristique est égal à $(X^2-1)^{n}$~:
toute représentation cuspidale autoduale de
$\GL_{2n}(F)$~con\-tient alors un caractère simple $\s$-autodual
(voir le corollaire \ref{MAIN1}).

Soit maintenant $\pi$ une représentation cuspidale autoduale
d'une forme intérieure quelconque $G$~de $\GL_{2n}(F)$,
et soit $\tau$ comme au paragraphe \ref{voiraussi110}.
\textit{Contrairement au cas $\s$-autodual considéré plus haut},
$\pi$ ne contient pas toujours un caractère simple $\t$ qui soit 
$\tau$-autodual, \ie tel que $\t^\tau=\t^{-1}$.
Le théorème suivant donne une condition nécessaire~et suf\-fisante 
d'existence d'un caractère simple $\tau$-autodual dans
une représentation cuspidale autoduale de $G$. 
Il généralise le corollaire \ref{MAIN1}
(voir le théorème \ref{THMEXISTENCECARACSIMPLE}).

\begin{theo}
\label{THMEXISTENCECARACSIMPLEintro}
Soit $\pi$ une représentation cuspidale autoduale 
de niveau non nul de $G$. Notons $\TT$ son endo-classe.
Alors $\pi$ contient un caractère simple
$\tau$-autodual si et seulement~si~: 
\begin{equation}
\bo_{T/T_0}(\a)^{2n/[T:F]} = (-1)^{r}.
\end{equation} 
\end{theo}

Ici encore, 
le cas des représentations de niveau $0$ nécessite un traitement à part~:
voir le~corol\-lai\-re \ref{baffcor}.

Le lien avec le facteur epsilon du paragraphe \ref{p19intro} est transparent~:
d'après le théorème \ref{valeureKTTintro},~une 
représentation cuspidale autoduale de niveau non nul de $G$ et 
d'endo-classe $\TT$
contient un~ca\-rac\-tère simple $\tau$-autodual si et seulement~si
$\EP_K(\TT)=(-1)^r$. 

Il ressort aussi des paragraphes \ref{tetrobot1} à \ref{degueulis}
menant à la preuve du théorème \ref{THMEXISTENCECARACSIMPLEintro}
que 
le problè\-me de l'existence d'un caractère simple $\tau$-autodual dans
une représentation cuspidale de niveau non nul et d'endo-classe donnée
peut être interprété comme un problème de plongement~:
pour qu'un tel caractère existe,
il faut et suffit qu'il y ait un plongement de $T$ dans $A$ tel que
la conjugaison par $\dd$ induise sur $T$ l'automorphisme non trivial
de $T/T_0$ (voir la proposition \ref{gandalf0coro}). 

Il ressort enfin de cette analyse que,
pour qu'une représentation cuspidale autoduale
de niveau non nul de $\AA^{{\rm sp}}(G,\TT)$
appartienne à $\AA^{{\rm ptb}}(G,\TT)$,
il faut et suffit qu'elle contienne un caractère simple $\tau$-autodual.  

\subsection{}

Cette analyse de $\AA^{{\rm ptb}}(G,\TT)$
en termes de caractères simples $\tau$-autoduaux
n'est pas un sim\-ple artifice~:
c'est ce qui va nous permettre de déterminer s'il peut y avoir 
distinction par~$H$.
Faire le lien entre caractères simples $\tau$-autoduaux et distinction 
est l'objet de la section \ref{SEC5}.
Le premier résultat~dans cette direction est le suivant
(voir la proposition \ref{grammage} et le paragraphe \ref{enterrement83}). 

\begin{prop}
\label{grammageintro}
Toute représentation cuspidale autoduale $H$-distinguée
de $G$ 
contient un caractère simple $\tau$-autodual. 
\end{prop}

En d'autres termes,
il découle de la proposition \ref{grammageintro}, 
des théorèmes \ref{THMEXISTENCECARACSIMPLEintro}
et \ref{valeureKTTintro}
et enfin de \eqref{incl1intro},  
qu'on a l'inclusion~:
\begin{equation}
\label{incl2intro}
\AA^{+}(G,\TT) \subseteq \AA^{{\rm ptb}}(G,\TT).
\end{equation} 
On en déduit immédiatement que,
si $\AA^{{\rm ptb}}(G,\TT)$ est vide,
alors $\AA^{+}(G,\TT)$ l'est aussi,
ce qui~prou\-ve la conjecture dans ce cas. 

Il reste alors à prouver l'égalité entre 
$\AA^{+}(G,\TT)$ et $\AA^{{\rm ptb}}(G,\TT)$ 
dans le cas où ce dernier n'est pas vide. 
Dans ce cas, partant de~:
\begin{equation}
\label{incl3intro}
\AA^{+}(G,\TT) \subseteq \AA^{{\rm ptb}}(G,\TT) = \AA^{{\rm sp}}(G,\TT), 
\end{equation}
et sachant que le cardinal de $\AA^{{\rm sp}}(G,\TT)$ est égal à
la moitié de celui de $\AA(G,\TT)$
(voir le paragra\-phe \ref{p19intro}), 
notre stratégie consiste à construire dans $\AA^{+}(G,\TT)$ des
représentations $H$-distinguées de $G$ en quantité au moins
égale à la moitié du cardinal de $\AA(G,\TT)$.
C'est ce que nous faisons dans les paragraphes \ref{viteA} à 
\ref{malinlegarconcor}~:
voir la proposition \ref{OlivierdeNoyen}
(et la proposition \ref{OlivierdeNoyen0} en niveau $0$). 
Pour des raisons de cardinal,
et compte tenu de \eqref{incl3intro},
on obtient ainsi l'égalité cherchée. 

\subsection{}

Lorsque $r$ est pair, de la forme $2k$ pour un entier $k\>1$, 
tous les résultats des sections \ref{SEC4} et \ref{SEC5} 
s'adaptent au cas où~$H$
est remplacé par le sous-groupe de Levi
$L=\GL_k(D)\times\GL_k(D)$ de $G$. 
Si l'on choisit un élément $\s\in\G$ tel que $\s^2=1$
et dont le polynôme caractéristique réduit est égal à $(X^2-1)^{n}$,
l'auto\-morphisme de conjugaison $\Ad(\s)$,
simplement noté $\s$,
est une in\-vo\-lu\-tion de $G$
dont le sous-groupe des points fixes est conjugué à $L$.
Le théorème \ref{THMEXISTENCECARACSIMPLEintro} devient
alors (voir le corollaire \ref{tolkienelrond} et la remarque 
\ref{balmasquecor0})~: 

\begin{theo}
\label{tolkienelrondintro}
Toute représentation cuspidale autoduale de $G$
contient un~ca\-rac\-tère simple $\s$-autodual.
\end{theo}

Les arguments
de la preuve du théorème \ref{theo12intro} 
sont eux aussi valables pour le groupe~$L$.
Ainsi,
si $\pi$ est une représentation cuspidale de $\GL_{2k}(D)$, 
et si l'on considère les assertions~:
\begin{enumerate}
\item 
la représentation $\pi$ est distinguée par $\GL_k(D)\times\GL_k(D)$, 
\item
le paramètre de Langlands de $\pi$ est autodual symplectique, 
\end{enumerate}
on a le résultat suivant (voir le théorème \ref{PTBsplitlevi}). 

\begin{theo}
\label{theo12introsplit}
Si (1) implique (2) pour~toute~re\-présentation cuspidale de $G$, 
alors les asser\-tions (1) et (2) sont équivalentes
pour~toute~re\-présentation cuspidale de $G$.
\end{theo}

Dans le cas où $G$ est déployé et $F$ est de caractéristique nulle,
il est connu que les asser\-tions (1) et (2) sont équivalentes~pour
toute~re\-présentation cuspidale~:
voir \cite{JNQ} et le théorème \ref{THMJNQ}.

\subsection{}

Voici deux conséquences que nous tirons du théorème \ref{theo12intro} 
et des résultats de la section \ref{SEC5}
dans le cas où le corps $F$ est de caractéristique nulle. 

D'abord,
la proposition \ref{MAIN1prop}
affirme qu'une représen\-tation cuspidale autoduale de niveau~non nul
$\pi$ de $G$
conte\-nant un caractère simple $\tau$-au\-todual~contient
automatiquement un type $\tau$-au\-todual,
\ie un type $(\BJ,\bl)$ tel que $\bl^\tau$ soit isomorphe à
$\bl^\vee$.
(On observera que ce résultat est faux en niveau $0$~:
voir le paragraphe \ref{findumonde}.)
Parmi les types $\tau$-auto\-duaux contenus dans $\pi$,
qui sont en nombre fini à $H$-conjugaison près,
il est possible de mettre en évidence 
un type~par\-ti\-cu\-lier,~unique à $H$-conjugaison près,
dit \textit{générique}.
(On renvoie au paragraphe~\ref{generic}~pour~une~justi\-fication~de la 
terminologie.)
On a le résultat suivant 
(voir le théorème \ref{MAIN2},
que l'on comparera à \cite[Corollary 6.6]{AKMSS}).

\begin{theo}
La représentation $\pi$ est $H$-distinguée si et seulement si son type 
générique est distinguée. 
\end{theo}

Enfin,
soit $\pi$ une représentation cuspidale autoduale de $\GL_{2n}(F)$
de niveau non nul.
Il existe une unique représentation cuspidale autoduale de $\GL_{2n}(F)$
inertiellement équivalente mais non isomorphe à $\pi$~;
notons-là $\pi^*$.
Soit $\TT$ l'endo-classe de $\pi$,
soit $T/T_0$ l'extension quadratique qui lui est associée et posons
$m=2n/\deg(\TT)$.
Lorsque $T/T_0$ est ramifiée et $m=1$,
les représenta\-tions $\pi$ et $\pi^*$ ont la même parité,
et~\cite[6.8]{BHS} montre comment déterminer
cette parité en termes~de théorie des types.
Dans les autres cas,
\ie si $T/T_0$ est non ramifiée, 
ou si $T/T_0$ est~ra\-mi\-fiée et $m$ est pair,
les représentations~$\pi$ et $\pi^*$ ont des parités différentes,
et il s'agit de~dé\-ter\-miner laquelle des deux est
de parité sym\-plec\-tique en termes de types. 
Les propositions \ref{critsymplec} et \ref{critsymplecnr}
donnent une réponse à cette question. 

\subsection{}

Pour finir,
dans la section \ref{tarteauxpommes},
nous expliquons brièvement comment  les méthodes développées dans les
sections \ref{SEC4} et \ref{SEC5} de l'article
peuvent s'appliquer au cas d'une
involution galoisienne sur une forme intérieure de $\GL_n(F)$.

\section*{Remerciements}

Je remercie chaleureusement Nadir Matringe pour de nombreuses discussions
stimulantes~à pro\-pos de ce travail.
Je le remercie en particulier de m'avoir expliqué les arguments de 
globalisation utilisés dans \cite{BroussousMatringe} et 
\cite{ChommauxMatringe},
ainsi que \cite{NadirJNT14}.

Je remercie également l'Institut Universitaire de France
pour les excellentes conditions de~tra\-vail qu'il m'a fournies
durant la réalisation de ce travail.

Je remercie enfin les rapporteurs anonymes pour leur relecture méticuleuse
et leurs excellentes suggestions, 
et pour m'avoir permis de corriger plusieurs erreurs dans certains arguments.

\section{Notations}
\label{Notation}

\subsection{}

On fixe un corps localement compact non archimédien $\F$,
de caractéristique résiduelle $p$ impaire.
On fixe une fois pour toute un caractère~:
\begin{equation}
\label{AddCharPsi}
\psi : \F \to \FC^\times
\end{equation}
trivial sur $\p_\F$ mais pas sur $\Oo_\F$,
à valeurs complexes. 

Si $\K$ est une extension finie de $F$,
ou plus généralement une $\F$-algèbre à division de dimension finie,
on note $\Oo_\K$ son anneau d'entiers,
$\p_\K$ l'idéal maximal de $\Oo_\K$ et $\kk_\K$ son corps résiduel,
qui est un corps fini de cardinal noté $q_K$.

Si $n$ est un entier strictement positif,
on note $\Mat_n(K)$ l'algèbre des~ma\-trices carrés de
taille $n$ à coefficients dans $K$
et $\GL_n(K)$ le groupe de ses éléments inversibles.

Si $K$ est commutatif,
on note $\N_{K/F}$ et $\tr_{K/F}$ la norme et la trace de $K$ sur $F$,
et on note $e_{K/F}$ et $f_{K/F}$ l'indice de ramification et le degré
résiduel de $K$ sur $F$.
On note aussi $\bm_K$ le sous-groupe de $\K^\times$ formé
de ses racines~de l'unité d'ordre premier à $p$.

Si $K$ est quadratique sur $F$,
on note $\bo_{K/F}$ le caractère de $\F^\times$ de noyau 
$\N_{K/F}(\K^\times)$.

\subsection{}

Par \textit{représentation} d'un groupe localement profini $\G$,
on entendra toujours une représentation lisse sur un $\FC$-espace
vectoriel.
On appellera \textit{caractère} de $G$
un homomorphisme de groupes de $\G$
dans $\FC^\times$ de noyau ouvert. 

\'Etant donné une représentation $\pi$ d'un sous-groupe fermé
$H$ de $\G$, on note $\pi^\vee$ sa représen\-ta\-tion contragrédiente. 
Si $\chi$ est un caractère de $H$, on note $\pi\chi$ la représentation
$x\mapsto\chi(x)\pi(x)$ de $H$.
Si $g\in\G$, on pose $H^g=g^{-1}\H g$ et on note 
$\pi^g$ la représentation $x\mapsto\pi(gxg^{-1})$ de $H^g$. 

Si $\s$ est une involution continue de $G$,
on note $\pi^\s$ la représentation $\pi\circ\s$ de $\s(\H)$.~Si $\mu$
est un caractère de $\H\cap\G^\s$, on dit que $\pi$ est 
$\mu$-\textit{distinguée} si l'espace 
$\Hom_{\H\cap\G^\s}(\pi,\mu)$ est non nul.~Si $\mu$
est le caractère trivial, on dit simplement que $\pi$ est
$\H\cap\G^\s$-\textit{distinguée}, ou juste \textit{distinguée}. 

\section{Préliminaires sur les types simples}
\label{PrelimST}

Nous introduisons dans cette section le langage élémentaire 
de la théorie des types simples, et nous
rappelons les principaux résultats 
dont nous aurons besoin concernant les strates, carac\-tè\-res~et types simples. 
Ces résultats sont issus
de \cite{BK,BHLTL1,BHEffectiveLC} pour les groupes linéaires~gé\-né\-raux sur $F$ 
et de \cite{VS1,VS2,VS3,SeStJIMJ,BSS} pour leurs formes intérieures.
On trouvera également quelques ré\-sultats originaux~:
il s'agit principalement de résultats qui étaient déjà connus pour les 
groupes 
linéaires~gé\-né\-raux mais pas pour leurs formes intérieures
(à l'exception des deux derniers~pa\-ra\-graphes
\ref{flipflap6} et \ref{invTTmoinsproto}).

\subsection{} 
\label{genitrix41}

Dans toute cette section,
on fixe un entier $n\>1$ et 
une $\F$-algèbre centrale simple $\A$ de~de\-gré réduit $n$.
{(On no\-te\-ra la différence avec le paragraphe \ref{intro1}
de l'introduction,
où le degré réduit de $\A$ est supposé être égal à $2n$.
Nous reviendrons à une $\F$-algèbre centrale simple de degré~ré\-duit $2n$
à partir~de la section \ref{SEC3}.)}
Fixons~un
$A$-mo\-du\-le à gauche~sim\-ple~$\V$, 
et notons $\D$ la $\F$-algèbre opposée à $\End_\A(\V)$.
C'est une $\F$-algèbre à division centrale,
$\V$ est un~$\D$-espace~vectoriel à~droi\-te et
$\A$ s'iden\-ti\-fie naturellement à la $\F$-al\-gè\-bre $\End_{\D}(\V)$.
Notons $r$ la di\-men\-sion de $\V$ sur $\D$ et $d$ le degré réduit de $D$
sur $F$.
On a donc $n=rd$ et,
si l'on fixe une base de $\V$ sur $\D$,
on obtient un isomorphisme de $D$-espaces vectoriels entre $V$ et $D^r$
et un isomorphisme de $\F$-algèbres entre $\A$ et
$\Mat_{r}(\D)$.

Soit $[\aa,\b]$ une strate simple dans $\A$.
Rappelons que $\aa$ est un $\Oo_F$-ordre héréditaire dans $\A$
et que $\b$ est un élément de $\A$ satisfaisant à certaines conditions,
parmi lesquelles~:
\begin{enumerate}
\item
la $\F$-algèbre $\E=\F[\b]$ est un corps, 
\item
son groupe multiplicatif $\E^\times$ normalise $\aa$.
\end{enumerate}
L'inclusion de $\E$ dans $\A$ fait de $\V$ un $E\otimes_{\F}\D$-module à droite.
Le cen\-trali\-sa\-teur de $\E$ dans $\A$, no\-té $\B$,
est une $\E$-algèbre centrale simple
s'identifiant natu\-rel\-lement à $\End_{\E\otimes\D}(V)$,
et l'intersec\-tion $\bb=\aa\cap\B$ est un $\Oo_E$-ordre héréditaire de $\B$.

Notons $\pp_\aa$ le radical de Jacobson de $\aa$.
Alors $\U^{i}(\aa)=1+\pp_\aa^{i}$ est un pro-$p$-sous-groupe
ouvert compact de $\G=\A^\times$ contenu dans $\aa^\times$,
pour tout entier $i\>1$.
La \textit{période} de $\aa$ est l'unique entier $e\>1$ tel 
que $\aa\pp_F$ soit égal à $\pp_{\aa}^{ed}$.

On associe à $[\aa,\b]$ des pro-$p$-sous-groupes ouverts compacts
$\H^1(\aa,\b)\subseteq\BJ^1(\aa,\b)\subseteq\U^1(\aa)$~et
un ensemble fini non vide $\Cc(\aa,\b)$
de caractères de $\H^1(\aa,\b)$ appelés
\textit{caractères simples},~dépendant
du ca\-rac\-tè\-re $\psi$ fixé à la section \ref{Notation}.
On pose $\BJ^0(\aa,\b)=\bb^\times\BJ^1(\aa,\b)$ et
on note $\BJI(\aa,\b)$ le sous-groupe de $G$ engendré par
$\BJ^1(\aa,\b)$ et le normalisateur de $\bb$ dans $\B^\times$.
On a~:
\begin{equation}
\label{TimJamieson}
\BJ^0(\aa,\b)\cap\B^\times=\bb^\times,
\quad
\BJ^1(\aa,\b)\cap\B^\times=\U^1(\bb),
\quad
\BJ^0(\aa,\b)/\BJ^1(\aa,\b) \simeq \bb^\times/\U^1(\bb).
\end{equation}
Pour la proposition qui suit, nous renvoyons à
\cite[2.1]{BHEffectiveLC}, \cite[Propo\-si\-tion 5.1.1]{BK},
\cite[3.3.2]{VS1}, \cite[Pro\-po\-sition~2.1]{MSt}
{et à la preuve de \cite[Proposition 2.6]{Dotto}.}

\begin{prop}
\label{patel}
Soit $[\aa,\b]$ une strate simple dans $\A$,
et soit $\t\in\Cc(\aa,\b)$.
\begin{enumerate}
\item
Le normalisateur de $\t$ dans $\G$ est égal à $\BJI(\aa,\b)$.
\item
Le groupe $\BJI(\aa,\b)$ a les propriétés suivantes~:
\begin{enumerate}
\item 
C'est un sous-groupe ouvert et
compact modulo le centre de $\G$. 
\item
Il possède un unique sous-groupe compact maximal,
égal à $\BJ^0(\aa,\b)$.
\item
Il possède un unique pro-$p$-sous-groupe compact distingué maximal, 
égal à $\BJ^1(\aa,\b)$.
\end{enumerate}
\item
\label{bourrin6}
L'ensemble d'entrelacement de $\t$ dans $\G$ est égal à 
$\BJ^1(\aa,\b)\B^\times\BJ^1(\aa,\b)$.
\item
\label{bourrin4}
Il existe une représentation irréductible $\n$ de $\BJ^1(\aa,\b)$,
unique à isomorphisme près,
dont~la res\-triction à $\H^1(\aa,\b)$ contienne $\t$,
et une telle représentation se prolonge à $\BJI(\aa,\b)$. 
\end{enumerate}
\end{prop}

\begin{rema}
\label{nullss}
Ceci inclut le cas particulier où $\b=0$.
Dans ce cas, $[\aa,0]$ est une strate~sim\-ple dans $\A$
quel que soit l'ordre héréditaire 
$\aa$.
Une telle strate simple est dite \textit{nulle}.
Dans cette situation, on a $E=F$ et $\H^1(\aa,0)=\U^1(\aa)$,
et $\Cc(\aa,0)$ est réduit au caractère trivial de $\U^1(\aa)$.~Le
groupe $\BJI(\aa,0)$ est le normalisateur de $\aa$ dans $\G$,
et on a $\BJ^0(\aa,0)=\aa^\times$ et $\BJ^1(\aa,0)=\U^1(\aa)$.
Ainsi la représentation $\n$ est le caractère trivial de $\U^1(\aa)$.
\end{rema}

Si $\t$ est un caractère simple dans $G$,
\ie s'il existe une strate simple $[\aa,\b]$ dans $A$~telle que
$\t\in\Cc(\aa,\b)$, 
on note
$\BJI_{\t}^{\phantom{'}}$ son normalisateur, 
$\BJ^0_{\t}$ l'unique sous-groupe compact maximal~de~$\BJI_{\t}^{\phantom{'}}$
et~$\BJ^1_{\t}$ son unique pro-$p$-sous-groupe distingué ma\-xi\-mal. 
On note aussi
$\H^1_{\t}
$ le pro-$p$-sous-groupe de $\BJ^1_\t$ sur lequel
$\t$ est défini.
La représentation irréductible $\n$ de $\BJ^1_\t$
caractérisée par~la~pro\-posi\-tion~\ref{patel}\eqref{bourrin4}
est appelée la \textit{représentation de~Hei\-sen\-berg} associée à $\t$.

\subsection{}
\label{saucisse}

Fixons un $\E\otimes_{\F}\D$-mo\-du\-le à droite~sim\-ple~$\SS$
et posons $\W=\Hom_{E\otimes\D}(\SS,\V)$.
C'est un~$\B$-modu\-le à gauche simple. 
Notons $\C$ la $E$-algèbre à division opposée à $\End_B(W)$.
Notons $m$ la di\-mension de $W$ sur $C$
et $c$ le degré réduit de $C$ sur $E$.
D'après \cite[Pro\-po\-si\-tion 1]{ZinkJNT99}, on a~:
\begin{equation}
\label{egalizink}
mc=\frac n {[E:F]},
\quad
c= \frac d {(d, [E:F])},
\quad
r = m\frac {[E:F]} {(d, [E:F])}.
\end{equation}
Le choix d'un~isomor\-phisme de $\E\otimes_{\F}\D$-modules 
$\V\simeq\SS^m$ induit un isomorphisme de $\F$-algèbres~: 
\begin{equation}
\label{PHIWEDEC}
\phi:\A\to\Mat_{m}(\End_\D(\SS))
\end{equation}  
dont la restriction à $B$ induit un isomorphisme de $\E$-algèbres
$\B\simeq\Mat_{m}(\C)$.
On note $\iso(\aa,\b)$~l'en\-semble des isomorphismes
\eqref{PHIWEDEC}~ainsi obtenus tels que l'ordre 
$\phi(\aa\cap\B)$ soit standard,
\ie formé de matrices à coefficients dans $\Oo_{\C}$ 
dont la réduction modulo $\p_C$ soit triangulaire supérieu\-re~par blocs. 
Un choix de $\phi\in\iso(\aa,\b)$ induit un~iso\-morphisme
de groupes~:
\begin{equation}
\label{JJ1UU1GLnonmax}
\BJ^0(\aa,\b)/\BJ^1(\aa,\b) \simeq \bb^\times/\U^1(\bb) \simeq
\GL_{m_1}(\ee)\tdt \GL_{m_s}(\ee)
\end{equation}
(le premier isomorphisme provenant de \eqref{TimJamieson})
où $\ee$ est le corps résiduel de $\C$,
les $m_i$~sont des~en\-tiers de somme $m$
et l'entier $s$ est la période de $\bb$.

\begin{rema}
\label{meplat}
Le cardinal de $\ee$ et l'entier $s$ sont entièrement déterminés par
n'importe~quel caractère simple $\t\in\Cc(\aa,\b)$.
En effet, d'après \eqref{JJ1UU1GLnonmax} et la proposition \ref{patel}, 
le centre de $\BJ^0_{\t}/\BJ^1_{\t}$ est un groupe abélien fini isomorphe
à $(\ee^{\times})^s$. 
On observe que l'ordre $\bb$ est maximal si et seulement si $s=1$,
c'est-à-dire si et seulement si le centre de $\BJ^0_{\t}/\BJ^1_{\t}$
est cyclique. 
\end{rema}

\subsection{}
\label{flipflap3}

Par la suite, nous ne considérerons que des caractères
simples~\textit{maximaux}, au sens de la~dé\-fini\-tion suivante
(voir la remarque \ref{chafouin1} pour une justification).

\begin{defi}
Une strate simple $[\aa,\b]$ dans $A$ est dite \textit{maximale}
si $\bb=\aa\cap\B$ est un ordre maximal dans $\B$. 
Si c'est le cas, les caractères simples dans $\Cc(\aa,\b)$
sont dits \textit{maximaux}.
\end{defi}

La remarque \ref{meplat} montre que, 
pour un caractère simple,
la propriété d'être maximal est indé\-pendante de la strate
simple~choi\-sie pour le définir.
En d'autres termes,
si $[\aa,\b]$ est maximale,~si
$\t\in\Cc(\aa,\b)$
et si $[\aa',\b']$~est une~stra\-te sim\-ple 
telle que $\t\in\Cc(\aa',\b')$,
alors $[\aa',\b']$~est maximale.

Une strate simple maximale a les propriétés suivantes,
qui précisent la proposition \ref{patel}.

\begin{prop} 
\label{patelmax}
Si $[\aa,\b]$ est une strate simple maximale dans $\A$,
alors $\aa$ est l'unique ordre de $\A$ normalisé par $E^\times$
tel que $\aa\cap\B=\bb$,
et le groupe $\BJ(\aa,\b)$ normalise $\aa$.
\end{prop}

\begin{proof}
La première assertion se déduit de \cite[Lemme 1.6]{VS2}. 
Rappelons ensuite que $\BJ(\aa,\b)$ est engendré par
$\BJ^1(\aa,\b)$, qui normalise $\aa$ car il est inclus dans
$\U^1(\aa)$, et le~nor\-ma\-lisateur de $\bb$ dans $B^\times$.
La seconde assertion suit alors de la propriété d'unicité de $\aa$.
\end{proof}

La proposition suivante montre que des strates simples
définissant un même caractère simple maximal ont de nombreux
invariants en commun. 

\begin{prop}
\label{rosimond}
Soit $[\aa,\b]$ et $[\aa',\b']$ des strates simples maximales dans $\A$
telles que~l'in\-ter\-section 
$\Cc(\aa,\b)\cap\Cc(\aa',\b')$ soit non vide.
Alors~:
\begin{equation*}
\Cc(\aa',\b')=\Cc(\aa,\b), \quad
\aa'=\aa, \quad
e_{F[\b']/F}=e_{F[\b]/F}, \quad
f_{F[\b']/F}=f_{F[\b]/F}.
\end{equation*}
\end{prop}

\begin{proof}
Posons $E=F[\b]$ et $E'=F[\b']$.
Soit $\t\in\Cc(\aa,\b)\cap\Cc(\aa',\b')$.
D'après la~pro\-position \ref{patel}, son normalisateur $\BJI_{\t}$
est égal à $\BJI(\aa,\b)=\BJI(\aa',\b')$.
Ce groupe contient $E^{\times}$~et~$E'^{\times}$~et,
les deux strates étant maximales, 
il~nor\-ma\-li\-se $\aa$ et $\aa'$. 
Intersectant la relation 
$\BJ^0(\aa,\b)\subseteq\aa'^{\times}$ avec
$\B^{\times}$, on trouve que $\bb^\times\subseteq\aa'^{\times}\cap\B^\times$.
Comme $\aa'$ est normalisé par $E^{\times}$,
l'intersection $\aa'\cap\B$ est un ordre héréditaire de $\B$.
Comme~$\bb$ est maximal, on en déduit donc que $\aa'\cap\B=\bb$.
La~pro\-po\-sition \ref{patelmax} entraîne que $\aa'=\aa$,
les égalités
$e_{E'/F}=e_{E/F}$ et $f_{E'/F}=f_{E/F}$
suivent de \cite[Lemma 4.14]{BSS} et
l'éga\-li\-té
$\Cc(\aa',\b')=\Cc(\aa,\b)$
suit de \cite[Theorem 4.16]{BSS}.
\end{proof}

\begin{rema}
\begin{enumerate}
\item 
Lorsque $A$ est déployée sur $F$,
la proposition est vraie sans hypothèse de maximalité
(voir \cite[2.1.1]{BHEffectiveLC} et l'argument au début de
\cite[Section 6]{SSPLMS01}).
\item 
Dans le cas non déployé,
il n'est pas difficile de trouver des strates simples 
$[\aa,\b]$ et $[\aa',\b']$ telles que 
$\Cc(\aa,\b)=\Cc(\aa',\b')$ et $\aa\neq\aa'$.
Par exemple,
soit $D$ une algèbre de quaternions sur~$F$,~et soit $E$ une extension
quadratique non ramifiée de $F$ incluse dans $D$.
Posons $A=\Mat_2(D)$ et~: 
\begin{equation*}
\aa = \begin{pmatrix}\Oo_D&\Oo_D\\\p_D&\Oo_D\end{pmatrix}
= \left\{\begin{pmatrix}a&b\\c&d\end{pmatrix}\in\Mat_2(\Oo_D)\
\bigg\vert \ c\in\p_D\right\}.
\end{equation*}
qui est un ordre minimal de $A$. 
Choi\-sis\-sons un $\b\in\E$ minimal
sur $F$ au sens de \cite[1.4.14]{BK}.~Alors
$[\aa,\b]$ est une strate simple dans $\A$ et $\B=\Mat_2(E)$.
Tout élément $y\in\B^\times$ normalisant $\bb=\aa\cap\B$
normalise aussi chaque caractère simple de $\Cc(\aa,\b)$.
On a donc $\Cc(\aa,\b)=\Cc(\aa^y,\b)$.
Pourtant, si~:
\begin{equation*}
y = \begin{pmatrix}0&1\\ \w_F&0\end{pmatrix}\in\B^\times,
\quad
\text{$\w_F$ uniformisante de $F$,}
\end{equation*}
on a $\aa^y\neq\aa$.
\end{enumerate}
\end{rema}

Soit $\t$ un caractère simple maximal de $G$,
et soit $[\aa,\b]$ une strate simple telle que $\t\in\Cc(\aa,\b)$.
L'ensemble~$\iso(\aa,\b)$ défini au paragraphe \ref{saucisse}
est formé~des isomorphismes \eqref{PHIWEDEC} tels que
l'image~de $\aa\cap\B$ soit égale à $\GL_{m}(\Oo_C)$.
En vertu du théorème de Skolem-Noe\-ther, deux isomorphismes
de $\iso(\aa,\b)$ sont
conjugués sous $\C^\times\GL_{m}(\Oo_C)$.
Un choix de $\phi\in\iso(\aa,\b)$~in\-duit un~iso\-morphisme~:
\begin{equation}
\label{JJ1UU1GLmax}
\BJ^0(\aa,\b)/\BJ^1(\aa,\b) \simeq \GL_{m}(\ee)
\end{equation}
de groupes,
où l'on rappelle que $\ee$ est le corps résiduel de $\C$. 

\subsection{} 
\label{pifoulechien}
\label{par44}

Une famille de~pai\-res $(\BJ,\bl)$ appelées
\textit{types simples maximaux étendus} de $\G$, 
constituées d'un sous-groupe $\BJ$~ou\-vert compact modulo le centre de $\G$
et d'une représentation irréductible~$\bl$ de~$\BJ$,
a été construite dans \cite{BK} et \cite{VS1,VS2,VS3}.
Pour qu'une paire $(\BJ,\bl)$ soit un type simple maximal étendu de $G$, 
il faut et suffit qu'il existe un caractère simple maximal~$\t$ de $G$ tel 
que~:
\begin{enumerate}
\item
$\F^\times\BJ^{0}_{\t}\subseteq\BJ\subseteq\BJI_{\t}^{\phantom{'}}$ et
le normalisateur de $\bl$ dans $\G$ est égal à $\BJ$,
\item
la restriction de $\bl$ à $\H^1_{\t}$ est un multiple de $\t$,
\item
étant donné une représentation $\bk$ de $\BJ$ prolongeant la représentation 
de Heisenberg de $\t$,
il existe une représentation irréductible $\bt$ de $\BJ$ triviale sur 
$\BJ^1_\t$ telle que~:
\begin{enumerate}
\item 
$\bl$ soit isomorphe à $\bk\otimes\bt$,
\item 
pour un
(ou, de façon équivalente, pour tout) 
choix d'une strate simple $[\aa,\b]$ telle~que
$\t\in\Cc(\aa,\b)$ et d'un isomorphisme $\phi\in\iso(\aa,\b)$,
la~res\-tric\-tion de $\bt$ à $\BJ^0$~est l'inflation 
\textit{via} \eqref{JJ1UU1GLmax}
d'une~re\-pré\-sentation cuspidale de $\GL_{m}(\ee)$.
\end{enumerate}
\end{enumerate} 
Dans toute la suite, pour abréger, nous écrirons \textit{type}
plutôt que \textit{type simple maximal étendu}.

\begin{rema}
Un tel groupe $\BJ$ a un~unique sous-groupe compact maximal $\BJ^0$
et un~unique pro-$p$-sous-grou\-pe com\-pact dis\-tingué maximal
$\BJ^1$ et,
pour tout caractère simple maximal $\t$ de~$G$ com\-me en (1),
on a $\BJ^0=\BJ^0_{\t}$ et $\BJ^1=\BJ^1_{\t}$.
\end{rema}

\begin{rema}
\label{pifoulechien00}
Un \textit{type de niveau $0$} est un type $(\BJ,\bl)$ de $\G$
attaché à une strate simple~ma\-xi\-male nulle $[\aa,0]$ au sens de la
remarque \ref{nullss}.
L'ordre $\aa$ est maximal
et $\bl$ est une~re\-pré\-sentation irréductible de $\BJ(\aa,0)$
triviale sur $\BJ^1(\aa,0)=\U^1(\aa)$ dont la restriction à
$\BJ^0(\aa,0)=\aa^\times$~induit~une représentation cuspidale de 
$\aa^\times/\U^1(\aa)\simeq\GL_{r}(\kk_D)$.
Une représentation cuspidale $\pi$ de~$G$ contient un type de niveau $0$
si et seulement si elle est de niveau $0$.
On a alors $[\E:\F]=1$.
\end{rema}

Soit $(\BJ,\bl)$ un type de $G$, 
et soit $\t$ un caractère simple maximal~comme ci-dessus. 
Si $\bk$ est une représentation de $\BJ$ prolongeant la représenta\-tion de
Heisenberg $\n$ associée à~$\t$, l'application~:
\begin{equation}
\label{fannyv}
\boldsymbol{\xi}\mapsto\bk\otimes\boldsymbol{\xi}
\end{equation} 
induit une bijection entre classes d'isomorphisme de représentations
irréductibles $\boldsymbol{\xi}$ de $\BJ$ triviales sur~$\BJ^1$ 
et classes d'isomorphisme de représentations irréductibles de $\BJ$
dont la restriction au sous-groupe $\BJ^1$ contient $\n$
(voir par exemple \cite[Lemme 2.6]{MSt}).
Par conséquent, 
la représentation $\bt$~de $\BJ$ triviale sur 
$\BJ^1$ telle que $\bl$ soit isomorphe à $\bk\otimes\bt$ est unique.

\begin{rema}
\label{Halo}
\label{malinlechien}
Fixons un isomorphisme $\phi\in\iso(\aa,\b)$ comme en (3). 
Soit~$\w$ une uniformi\-sante de $C$,
soit $\rho$ la représentation cuspidale de $\GL_{m}(\ee)$ 
dont la restriction de $\bt$ à $\BJ^0$ est~l'infla\-tion~et 
soit $b$ l'indice de $\BJ$ dans~$\BJI_{\t}$.
Le groupe $\BJ_\t$ est engendré par $\BJ^0$ et $\w$ et 
l'orbite de $\rho$ sous l'action de $\Gal(\ee/\kk_\E)$ a pour cardinal $b$, 
\ie que $\BJ$ est engendré par $\BJ^0$ et $\w^b$.
\end{rema}

\begin{rema}
Le lecteur au fait de la théorie des types aura remarqué que nous n'avons pas 
utilisé ici la notion de beta-extension (voir \cite{VS2} et \cite[4.1]{VS3}).
Une beta-extension est une~re\-pré\-sentation~de $\BJ^0$ prolongeant $\n$ et
entrelacée par $\BJ^1\B^\times\BJ^1$.
Elle est donc normalisée par $\BJ_\t$ et se prolonge à~$\BJ_\t$.
Inversement,
étant donné une représentation $\bk_\b$ de $\BJ_\t$ prolongeant une 
beta-extension, 
toute~re\-pré\-sentation $\bk$ de $\BJ_\t$ prolongeant $\n$ s'obtient
en tordant $\bk_\b$ par un caractère $\bx$ de $\BJ_\t$ trivial sur $\BJ^1$.
La restriction de $\bx$ à $\BJ^0$ s'identifie \textit{via} \eqref{JJ1UU1GLmax}
à un caractère $\xi$ de $\GL_m(\ee)$.
Ce groupe n'étant pas isomorphe~à $\GL_2(\FF_2)$ puisque $p\neq2$,
et $\bx$ étant invariant par conjugaison par l'uniformisante $\w$ de $\C$,
le caractère $\xi$ se factorise par le déterminant et est
invariant sous l'action de $\Gal(\ee/\kk_\E)$. 
Aussi $\bx$, 
vu comme un caractère de 
$(\BJ_\t\cap\B^\times)/(\BJ^1\cap\B^\times)$,
se factorise-t-il par $\Nrd_{B/E}$.
Compte tenu de \cite[Théorème 2.28]{VS2},
la restriction de $\bk$ à $\BJ^0$
est une beta-extension. 
\end{rema}

Profitons-en pour énoncer le lemme suivant,
qui sera très utile par la suite. 

\begin{lemm}
\label{bkpprimaire}
Soit $\t$ un caractère simple maximal de $G$,
soit $\BJI_\t$ son normalisateur dans $G$ et
soit $\n$ sa représentation de Heisenberg.  
Il existe des prolongements de $\n$ à $\BJ_\t$
dont le~dé\-ter\-mi\-nant soit d'ordre une puissance de $p$.
\end{lemm}

\begin{proof}
Si $\bk$ est un prolongement de $\BJI_{\t}$ à $\n$,
son~dé\-ter\-minant est un caractère de~$\BJI_{\t}$
noté $\boldsymbol{\d}$. 
Quitte à tordre $\bk$ par un caractère de $\BJI_\t$ trivial sur $\BJ_\t^0$,
on peut supposer que l'ordre de $\boldsymbol{\d}$ est fini. 
Il se décompose donc de façon unique comme le produit d'un ca\-rac\-tère
d'ordre une puissance de $p$ et d'un caractère d'ordre premier à $p$. 
Notons $\boldsymbol{\omega}$ ce dernier~:
c'est un caractère~de $\BJI_{\t}$ trivial sur $\BJ_\t^1$.
Comme la dimension de $\bk$ est une puissance de $p$, 
il existe un entier $a\>1$ tel que $a\cdot\dim\bk$ soit con\-gru à $1$ modulo
l'ordre de $\boldsymbol{\omega}$. 
On vérifie alors que la représentation $\bk\boldsymbol{\omega}^{-a}$ a la propriété 
voulue. 
\end{proof}

\subsection{}
\label{regimbard}

Donnons la classification des représentations irréductibles
cuspidales de $\G$~en~termes de~ty\-pes 
(voir \cite[6.2, 8.4]{BK}, \cite[Théorème 5.23]{SeStJIMJ} et
\cite[Corollary 7.3]{SeStIMRN}).

\begin{theo}
\label{THEOCLASSIFTYPES}
\begin{enumerate}
\item
\'Etant donné une représentation cuspidale $\pi$ de $\G$,
il existe un type $(\BJ,\bl)$ tel que $\bl$
apparaisse comme sous-représentation de la restriction de
$\pi$ à $\BJ$,
et un tel type est unique à $\G$-conjugaison près.
\item
L'induction compacte induit une bijection entre classes de
$\G$-conjugaison de types 
et~clas\-ses d'isomorphisme de représentations cuspidales de $\G$.
\end{enumerate}
\end{theo}

{Plus précisément,
fixons un caractère simple maximal $\t$ de $G$
et une représentation $\bk$ de $\BJ_\t$ prolongeant 
sa représentation de Heisenberg.  
L'application~:
\begin{equation}
\label{bijthetacusp}
(\BJ,\bt) \mapsto \ind^\G_{\BJ} \left(\bk|_{\BJ} \otimes\bt\right)
\end{equation}
induit une bijection entre~:
\begin{enumerate}
\item 
classes de $\BJ_\t$-conjugaison de paires $(\BJ,\bt)$ telles que~:
\begin{enumerate}
\item
$\F^\times\BJ^{0}\subseteq\BJ\subseteq\BJI_{\t}$
et $\bt$ est une classe d'isomorphisme de
représentations irréductibles de $\BJ$ triviale sur $\BJ^1$
dont le normalisateur dans $\BJI_{\t}$ est égal à $\BJ$,
\item
la restriction de $\bt$ à $\BJ^0$ est l'inflation d'une représentation 
cuspidale de $\GL_m(\ee)$,
\end{enumerate}
\item
classes d'isomorphisme de
représentations irréductibles cuspidales de $G$ contenant $\t$.
\end{enumerate}
La bijection \eqref{bijthetacusp} dépend du choix de la représentation $\bk$,
mais l'ensemble (1) n'en dépend pas.

\begin{rema}
\label{chafouin1}
Un caractère simple contenu dans une représentation cuspidale de $G$ est
toujours maximal (\cite[Corollaire 5.20]{SeStJIMJ}),
ce qui justifie qu'on ne s'intéresse ici qu'aux~ca\-ractè\-res sim\-ples
maximaux.
Inversement, 
étant donné un caractère simple maximal $\t$ de $\G$, 
il existe~une représentation cuspidale de $G$ contenant $\t$.
\end{rema}

\'Etant donné une représentation cuspidale $\pi$ de $\G$,
nous lui associons deux invariants importants,
qui seront utiles à partir de la section \ref{SEC4}. 
D'une part, si $(\BJ,\bl)$ est un type contenu dans~la représentation $\pi$,
le \textit{degré~para\-métrique} de $\pi$ est l'entier~:
\begin{equation}
\label{pardeg}
\delta(\pi) = mb \cdot [E:F]
\end{equation}
où $m$ est défini par \eqref{egalizink} et $b$ par la remarque \ref{Halo}
(voir \cite[Definition 2.7]{BHJL3}).

D'autre part,
le transfert de Jacquet-Langlands de $\pi$ à $\GL_{n}(F)$ est
une re\-pré\-sentation es\-sen\-tiellement de carré intégrable,
notée $\pi'$.
D'après \cite[Theorem 9.3]{ZelAENS},
il y~a un unique diviseur~$s=s(\pi)$ de $n$ et 
une unique~représen\-ta\-tion cuspidale $\pi'_{0}$ de
$\GL_{n/s}(F)$~telle
que $\pi'$ soit isomorphe~à l'uni\-que quotient irréductible
$L(\pi'_{0},s)$ de l'induite parabolique~:
\begin{equation*}
\label{portedoree}
\Ind^{\GL_{n}(F)}_{P}
\left(\pi'_{0}\nu^{(1-s)/2}\otimes\dots\otimes\pi'_{0}\nu^{(s-1)/2}\right),
\end{equation*}
l'induction étant normalisée et prise par rapport au sous-groupe
parabolique standard triangulaire supérieur par blocs $P$,
et $\nu$ désignant le caractère non ramifié ``valeur absolue du détermi\-nant''
de $\GL_{n/s}(F)$.
Ces deux entiers sont liés par la relation importante~:
\begin{equation}
\label{scoprimem}
\d(\pi) s(\pi) = n
\end{equation} 
(voir par exemple \cite[Proposition 12.2]{MSjl}). 

\begin{rema}
\label{marecage}
\begin{enumerate}
\item 
Cette relation est une conséquence directe de l'identité 
$\d(\pi)=\d(\pi'_0)$ don\-née par 2.8 Corollary 1
et l'identité (2.8.1) de \cite{BHJL3}.
Il s'ensuit que $bs(\pi)=c$,
\ie que,
compte tenu de la~re\-marque~\ref{malinlechien} dont
on reprend les notations, 
$s(\pi)$ est l'ordre du stabilisateur de $\rho$ dans 
$\Gal(\ee/\kk_E)$. 
\item
De cette première remarque on déduit que $s(\pi)$ est premier à $r$
(selon \cite[Corollaire~3.9]{MSjl}). 
Posant $\d=\d(\pi)$ et $s=s(\pi)$, on obtient les formules~:
\begin{equation}
\label{rsprime}
r = \frac {\d} {(d,\d)},
\quad
s = \frac {d} {(d,\d)}.
\end{equation}
\item
Inversement, 
si $k$ est un diviseur de $n$ et $\rho$ est une représentation cuspidale 
de $\GL_{n/k}(F)$,
le transfert de Jacquet-Langlands de $L(\rho,k)$ à $\G$ est une représentation
essentiellement de carré intégrable ${\s}$.
{D'après \cite{DKV,Tadic} (voir aussi \cite[(2.2)]{BHLS},}
il y a un unique diviseur $t$ de $r$ et une unique représentation cuspidale 
$\mu$ de $\GL_{r/t}(D)$ tel que celle-ci soit isomorphe à l'unique
quotient~ir\-ré\-duc\-tible $L(\mu,t)$ de l'induite parabolique~:
\begin{equation*}
\label{portedoreens}
\Ind^{\G}_Q
\left(\mu\nu^{s(\mu)(1-t)/2}\otimes\dots\otimes\mu\nu^{s(\mu)(t-1)/2}\right),
\end{equation*}
l'induction étant normalisée et prise par rapport au sous-groupe
parabolique standard triangu\-laire supérieur par blocs $Q$,
et $\nu$ désignant ici le caractère non ramifié
``valeur absolue de la~nor\-me~réduite'' de $\GL_{r/t}(D)$.
Le transfert de Jacquet-Langlands de $\mu$ à $\GL_{n/t}(F)$
étant de la~for\-me~$L(\mu'_0,s(\mu))$ pour une représentation cuspidale
$\mu'_0$ de $\GL_{n/ts(\mu)}(F)$,
on en déduit que $\mu'_0$ et $\rho$ sont isomorphes,
donc que $\d(\mu)=n/k$.
Appliquant \eqref{rsprime} à la représentation $\mu$, on trouve~:
\begin{equation*}
\frac r t = \frac {n/k} {(d,n/k)} = \frac {r} {(r,k)},
\end{equation*}
ce dont on déduit que $t=(r,k)$.
Par conséquent,
pour que $\s$ soit cuspidale, 
il faut et suffit que $k$ soit premier à $r$.
\end{enumerate}
\end{rema}

\subsection{}
\label{simplechar}

Introduisons maintenant la notion d'endo-classe de caractères simples
\cite{BHLTL1,BHIMRN,BSS},
qui sera au centre de tout notre travail. 

Soit $[\aa,\b]$ une strate simple dans $\A$,
et soit $[\aa',\b']$ une strate simple dans une autre $F$-algèbre
centrale simple $\A'$.
Supposons qu'il existe un isomorphisme de
$\F$-algèbres $\phi:\F[\b]\to\F[\b']$ tel que $\phi (\b)=\b'$.
Il y a alors une bijection canonique~:
\begin{equation*}
\Cc(\aa,\b) \to \Cc(\aa',\b')
\end{equation*}
appelée \emph{transfert}
(\cite[3.6]{BK} et \cite[3.3.3]{VS1}).

Les applications de transfert jouissent d'une propriété de transitivité~:
si $[\aa,\b]$ et $[\aa',\b']$ sont comme ci-dessus,
si $[\aa'',\b'']$ est une strate simple dans une $F$-algèbre centrale simple 
$\A''$ et s'il~y a un~iso\-mor\-phisme de $\F$-algèbres
de $\F[\b']$ sur $\F[\b'']$ envoyant $\b'$ sur $\b''$,
le transfert de $\Cc(\aa,\b)$~à $\Cc(\aa'',\b'')$ est la composée du
transfert de $\Cc(\aa,\b)$ à $\Cc(\aa',\b')$ et du 
transfert de $\Cc(\aa',\b')$ à $\Cc(\aa'',\b'')$.

\begin{rema}
\label{transfertconj}
Un cas particulier simple mais important de transfert
est celui où $\A'=\A$ et
$[\aa',\b']$ est~conju\-guée à $[\aa,\b]$ par un $g\in\G$.
Dans ce cas, le transfert de $\Cc(\aa,\b)$ à $\Cc(\aa',\b')$
est donné par la conjugaison par $g$.
\end{rema}

Dans la suite du paragraphe,
nous nous concentrons sur les propriétés des caractères simples
\textit{maximaux}
vis-à-vis du transfert.

\begin{defi}
Soient $A_1$ et $\A_2$ des $F$-algèbres centra\-les simples,
et $\t_1$ et $\t_2$ des caractères simples maximaux dans $\A_1^\times$ et
$\A_2^\times$ respectivement. 
Les caractères simples $\t_1$ et $\t_2$
sont dits \emph{endo-équi\-valents} s'il
existe des strates simples $[\aa_1,\b_1]$ et $[\aa_2,\b_2]$ dans $\A_1$
et $\A_2$ telles que~:
\begin{enumerate}
\item
$\t_1\in\Cc(\aa_1,\b_1)$ et $\t_2\in\Cc(\aa_2,\b_2)$,
\item
les extensions $\F[\b_1]$ et $\F[\b_2]$ ont le même degré sur $F$,
\item 
il y a une $F$-algèbre centrale~sim\-ple $\A'$ et 
des strates simples maximales $[\aa',\b'_1]$ et $[\aa',\b'_2]$ dans~$\A'$
tel\-les que $\t_1$ et $\t_2$ 
se transfèrent en des~ca\-ractères simples $\t'_1\in\Cc(\aa',\b'_1)$ et 
$\t'_2\in\Cc(\aa',\b'_2)$ qui~s'en\-tre\-lacent dans $\A'^\times$, 
\ie qu'il existe un $g\in \A'^\times$ tel que~:
\begin{equation*}
\t'_2(x)=\t'_1(gxg^{-1}),
\quad
x\in\H^1(\aa',\b'_2)\cap g^{-1}\H^1(\aa',\b'_1)g.
\end{equation*}
\end{enumerate}
\end{defi}

En particulier, 
deux caractères simples maximaux transferts l'un de l'autre sont 
endo-équi\-va\-lents.

La notion d'endo-équivalence a été introduite dans \cite{BHLTL1}
(voir \cite{BSS} pour les formes intérieures) d'une
fa\-çon légèrement différente
reposant sur la notion de caractère simple potentiel,
que~nous n'utiliserons pas dans cet article. 
Dans le cas des groupes linéaires généraux déployés,
\cite{BHIMRN}~ex\-plique comment faire l'économie des
caractères simples potentiels.  
Dans le présent article,
nous nous contentons de le faire pour les caractères simples 
maximaux des formes intérieures.
Ainsi,
la proposition suivante est une adaptation de \cite[Corollary 8.3]{BSS}. 

\begin{prop}
L'endo-équivalence est une relation d'équivalence sur l'ensemble~:
\begin{equation*}
\Cc_{{\rm max}}(F) = \bigcup\limits_{[\aa,\b]} \Cc(\aa,\b),
\end{equation*}
l'union étant prise sur l'ensemble des strates simples maximales
des $F$-algèbres centrales simples.
\end{prop}

\begin{proof}
Prouvons qu'il s'agit d'une relation transitive. 
Soit $\t_1$, $\t_2$ et $\t_3$ des caractères simples maximaux tels que 
$\t_1$, $\t_2$ soient endo-équivalents,
ainsi que $\t_2$, $\t_3$.
Pour $\t_1$, $\t_2$, il y a~des strates simples maximales
$[\aa_1,\b_1]$, $[\aa_2,\b_2]$ et $[\aa',\b'_1]$, $[\aa',\b'_2]$
comme ci-dessus,
avec des transferts $\t'_1$, $\t'_2$ qui s'entrelacent dans 
$\A'^\times$.
{D'après \cite[Lemma 4.7]{BSS}, on a~:}
\begin{equation}
\label{forestier}
{e_{\F[\b_1]/\F} = e_{\F[\b_2]/\F},
\quad
f_{\F[\b_1]/\F} = f_{\F[\b_2]/\F}.}
\end{equation}
Pour $\t_2$, $\t_3$, il y a~des strates simples maximales
$[\aa^*_2,\b^*_2]$ et $[\aa^*_3,\b^*_3]$ dans $A_2$ et $A_3$
telles que~: 
\begin{enumerate}
\item
$\t_2\in\Cc(\aa^*_2,\b^*_2)$ et $\t_3\in\Cc(\aa^*_3,\b^*_3)$,
\item
les extensions $\F[\b^*_2]$ et $\F[\b^*_3]$ ont le même degré sur $F$,
\item 
il y a une $F$-algèbre centrale~sim\-ple $\A''$ et 
des strates simples maximales $[\aa'',\b''_2]$ et $[\aa'',\b''_3]$ dans~$\A''$
tel\-les que $\t_2$ et $\t_3$ 
se transfèrent en des~ca\-ractères simples $\t''_2\in\Cc(\aa'',\b''_2)$ et 
$\t''_3\in\Cc(\aa'',\b''_3)$ qui~s'en\-tre\-lacent dans $\A''^\times$,
\end{enumerate}
et de façon analogue à \eqref{forestier} on a les égalités
$e_{\F[\b^*_2]/\F} = e_{\F[\b^*_3]/\F}$ et
$f_{\F[\b^*_2]/\F} = f_{\F[\b^*_3]/\F}$. 
Comme le caractère simple 
$\t_2$ est dans l'intersection
$\Cc(\aa_2^{\phantom{*}},\b_2^{\phantom{*}})\cap\Cc(\aa_2^*,\b_2^*)$, 
la proposition \ref{rosimond} implique~:
\begin{equation*}
\Cc(\aa_2^*,\b_2^*)=\Cc(\aa_2^{\phantom{*}},\b_2^{\phantom{*}}),
\quad
\aa_2^*=\aa_2^{\phantom{*}},
\quad
e_{\F[\b^*_2]/F}=e_{\F[\b_2^{\phantom{*}}]/F},
\quad
f_{\F[\b^*_2]/F}=f_{\F[\b_2^{\phantom{*}}]/F}.
\end{equation*} 
On en déduit que les extensions $\F[\b_1^{\phantom{*}}]$,
$\F[\b_2^{\phantom{*}}]$, $\F[\b^*_2]$ et $\F[\b^*_3]$
ont 
le même indice~de ra\-mi\-fication (noté $e$),
le même degré résiduel (noté $f$) et 
le même degré (noté $g$) sur $\F$.

{Soit maintenant une $F$-algèbre centrale sim\-ple déployée $\A^{\circ}$,
dont le degré réduit $n^{\circ}$ soit~divi\-sible par $g$,
et fixons un $\Oo_\F$-ordre principal $\aa^{\circ}$ de $\A^{\circ}$
de période $e$
(\ie que la $\kk_\F$-algèbre $\aa^{\circ}/\pp_{\aa^{\circ}}$ est isomorphe
à la somme directe de $e$ copies de $\Mat_{n^{\circ}/e}(\kk_\D)$).
Selon \cite[Proposition~1.1]{BHLTL1}, 
pour chaque $i\in\{1,2,3\}$, 
il y a des plongements de
$\F$-algèbres $\iota_i:\F[\b_i]\to\A^{\circ}$
tels que $\iota_i(\F[\b_i])^\times$ normalise $\aa^{\circ}$.
On en déduit des~stra\-tes simples maximales
$[\aa^{\circ},\b^{\circ}_1]$, $[\aa^{\circ},\b^{\circ}_2]$, $[\aa^{\circ},\b^{\circ}_3]$ 
de $\A^{\circ}$ tel\-les que $\t_1$, $\t_2$, $\t_3$ 
se~trans\-fèrent en 
$\t_1^{\circ}\in\Cc(\aa^{\circ},\b^{\circ}_1)$,
$\t_2^{\circ}\in\Cc(\aa^{\circ},\b^{\circ}_2)$,
$\t_3^{\circ}\in\Cc(\aa^{\circ},\b^{\circ}_3)$.}

Comme $\t_1'$ et $\t_2'$ s'entrelacent dans $\A'^\times$,
il suit de \cite[Theorem 1.11]{BSS} 
que $\t_1^{\circ}$ et $\t_2^{\circ}$ s'entrelacent dans $\A^{\circ\times}$,
et de \cite[Theo\-rem 3.5.11]{BK}
qu'ils sont conjugués sous $\A^{\circ\times}$. 
Faisant de même avec $\t_2''$ et $\t_3''$,
on trouve que 
$\t_2^{\circ}$, $\t_3^{\circ}$ sont conjugués sous $\A^{\circ\times}$.
Par conséquent, $\t_1^{\circ}$ et $\t_3^{\circ}$ sont conjugués,
donc s'en\-tre\-lacent dans $\A^{\circ\times}$.
\end{proof}

Une classe d'équivalence pour cette relation 
sur $\Cc_{{\rm max}}(F)$
est appelée une $F$-\emph{endo-classe}.
On note $\Ee(F)$ l'ensemble des $F$-endo-classes. 

\'Etant donné un caractère simple maximal $\t\in\Cc(\aa,\b)$,
le degré de $\F[\b]/\F$, son indice de~rami\-fication et son degré résiduel
dépendent uniquement de son endo-classe $\TT$.
Ces entiers sont~ap\-pe\-lés respectivement le degré,
l'indice de ramification et le degré résiduel de $\TT$.
Cette extension $\F[\b]$ n'est pas uniquement déterminée,
mais sa sous-extension modérément ramifiée maximale l'est,
à $F$-isomorphisme près (\cite[2.2,~2.4]{BHEffectiveLC}). 

\begin{rema}
L'endo-classe nulle est l'endo-classe d'un caractère simple
maximal trivial associé à une strate simple maximale nulle
(voir la remarque \ref{nullss}). 
\end{rema}

\subsection{}
\label{toposcristallin}

Lorsqu'on s'intéresse aux caractères simples maximaux d'un même 
groupe,
la relation~d'en\-do-équivalence prend une forme simplifiée. 

\begin{prop}
\label{entrenous}
{Soit $A$ une $F$-algèbre centrale simple,
et soit $\t_1$, $\t_2$ des caractères simples maximaux de $G=\A^\times$.
Les assertions suivantes sont équivalentes~: 
\begin{enumerate}
\item
Les caractères $\t_1$, $\t_2$ sont endo-équivalents,
\item
{Les caractères $\t_1$, $\t_2$ sont entrelacés par un élément de $G$,}
\item
Les caractères $\t_1$, $\t_2$ sont con\-jugués sous $G$,
\item
Les caractères $\t_1$, $\t_2$ sont contenus dans une même
représentation cuspidale de $G$. 
\end{enumerate}}
\end{prop}

\begin{proof}
{Le fait que (1) implique (3) est donné par \cite[Proposition 2.5]{Dotto}.
Bien~sûr,~(3) implique (2),
elle-même impliquant (1) par définition.
Ensuite,~la~re\-marque \ref{chafouin1} assure
qu'il~y~a une représentation cuspidale de~$G$ contenant $\t_1$.
Une telle~re\-pré\-sentation contient tous les conju\-gués de $\t_1$
sous $\G$, 
ce qui entraîne que (3) implique~(4).
Enfin,~si $\t_1$, $\t_2$ sont contenus dans une même 
représentation cuspidale $\pi$ de $G$, on a par~ré\-ci\-pro\-cité~de Frobenius
des morphismes~sur\-jec\-tifs~:
\begin{equation*}
\ind^\G_{\H^1_{\t_i}}(\t_i)\to\pi, 
\quad
i=1,2,
\end{equation*}
le membre de gauche désignant l'induite compacte de $\t_i$ à $\G$.
Le sous-groupe~$\H^1_{\t_i}$ étant ouvert et compact, 
cette induite compacte est un objet projectif de la catégorie
des~re\-pré\-sentations (lisses, complexes) de $\G$.
On en déduit un morphisme non nul~:
\begin{equation*}
\ind^\G_{\H^1_{\t_1}}(\t_1) \to \ind^\G_{\H^1_{\t_2}}(\t_2).
\end{equation*}
Appliquant à nouveau la ré\-ci\-pro\-cité~de Frobenius,
puis la formule de Mackey, 
on en déduit que $\t_1$ et $\t_2$ s'entrelacent, \ie que (4) implique (2).}
\end{proof}

\begin{defi}
L'\textit{endo-classe} d'une représentation~cus\-pi\-da\-le $\pi$ de $\G$ 
est l'endo-classe de n'importe quel caractère simple maximal contenu dans 
$\pi$. 
\end{defi}

\begin{rema}
Une représentation cuspidale de $\G$ est d'endo-classe nulle
si et seulement si elle est de niveau $0$,
\ie si et seulement si elle admet un vecteur non nul invariant
par $\BU^1(\aa)$, où $\aa$ est un ordre maximal quelconque de $A$.
\end{rema}

{Compte tenu de la remarque \ref{chafouin1},
on en déduit le corollaire suivant.}

\begin{coro}
\label{entrenouscor}
{Soit $\pi$ une représentation~cus\-pi\-da\-le de $\G$ d'endo-classe 
$\TT$. 
Pour qu'un~ca\-rac\-tère simple de $\G$ soit contenu dans $\pi$,
il faut et suffit qu'il soit maximal et d'endo-classe $\TT$.}
\end{coro}

\subsection{}

Nous déduisons du paragraphe \ref{toposcristallin}
le résultat complémentaire suivant. 

\begin{prop}
\label{VanceWestonCARASS}
{Soit $(\BJ,\bl)$ un type de $G$.
Le caractère simple maximal $\t$ vérifiant les~con\-di\-tions 
{\rm (1)} et {\rm (2)} du paragraphe \ref{pifoulechien} est unique.
On l'appelle le caractère simple {\rm attaché à} $(\BJ,\bl)$.}
\end{prop}

\begin{proof}
Soit $\t$, $\t'$ des caractères simples maximaux de $G$ 
vérifiant les conditions~re\-qui\-ses. 
Soit $[\aa,\b]$, $[\aa',\b']$ des strates simples de $\A$ telles
que $\t\in\Cc(\aa,\b)$ et $\t'\in\Cc(\aa',\b')$.~Restrei\-gnant~$\bl$~au 
pro-$p$-sous-groupe $H^1_{\t}\cap H^1_{\t'}$,
on a l'égalité $\t=\t'$ sur ce sous-grou\-pe. 
D'après~la~pro\-position \ref{entrenous},
et du fait que $\t$, $\t'$ sont contenus dans la représentation cuspidale
$\ind^G_{\BJ}(\bl)$, 
il y~a un $g\in\G$ tel que $\t'=\t^g$.
Il entrelace $\t$, on a donc $g\in\BJ^1\B^\times\BJ^1$.
Comme~$\BJ^1$~nor\-ma\-lise~$\t$, 
on peut supposer que $g\in\B^\times$.
De l'égalité $\BJI_{\t'}^{\phantom{g}}=\BJI_{\t}^{g}$,
on tire que $\BJ^{0}=\BJ^{0g}$
par unicité~du sous-groupe compact~ma\-ximal.
Intersectant avec $\B^\times$, on trouve $\bb^\times=\bb^{\times g}$,
\ie~que~$g$~nor\-malise $\bb$, donc $\t$.
\end{proof}

\begin{rema}
\label{chafouin}
\begin{enumerate}
\item 
Un même caractère simple maximal $\t$
peut être attaché à des types non~iso\-mor\-phes. 
\item 
Si $\pi$ est une représentation~cus\-pi\-da\-le de $\G$,
deux types contenus dans $\pi$ ont le même~ca\-rac\-tère simple attaché $\t$
si et seulement s'ils sont conjugués sous $\BJI_{\t}$. 
\item
Si $\A$ est~dé\-ployée sur $F$,
et si $(\BJ,\bl)$ est un type de ca\-rac\-tère simple attaché $\t$, 
alors~$\BJ=\BJI_{\t}$.
Par conséquent,
une représentation~cus\-pi\-da\-le de $\GL_n(F)$ 
contenant $\t$ contient un unique
type~au\-quel $\t$ est attaché.
\end{enumerate}
\end{rema}

\subsection{}
\label{flipflap6}

Si $\t$ est un caractère simple maximal de $G$,
et si $[\aa,\b]$ est une strate simple de $A$ telle que
$\t\in\Cc(\aa,\b)$,
alors $[\aa,-\b]$ est une strate simple dans $A$
et $\t^{-1}\in\Cc(\aa,-\b)$.

Deux caractères simples maximaux
sont endo-équivalents si et seulement si leurs~inver\-ses
sont endo-équivalents.
Par conséquent, 
étant donné une endo-classe $\TT\in\Ee(F)$ et un caractère simple
maximal $\t$ d'endo-classe $\TT$,
l'endo-classe de $\t^{-1}$ ne dépend pas du choix de $\t$.
On la note $\TT^\vee$.

\begin{defi}
Une endo-classe $\TT\in\Ee(F)$ est dite \textit{autoduale} si $\TT^\vee=\TT$.
\end{defi}

{En d'autres termes,
pour qu'une endo-classe $\TT$ soit autoduale,
il faut et suffit qu'au moins un,
ou de façon équivalente n'importe quel,
caractère simple maximal d'endo-classe $\TT$ soit conjugué à son inverse.}

\begin{lemm}
\label{lemmedetransfertautodual}
Soient $A_1$ et $A_2$ des $F$-algèbres centrale simples,
soient $[\aa_1,\b_1]$ et $[\aa_2,\b_2]$ des strates simples maximales
respectivement dans $A_1$ et $A_2$,
soient $\t_1\in\Cc(\aa_1,\b_1)$, $\t_2\in\Cc(\aa_2,\b_2)$~des 
caractè\-res simples 
et soient $u_1$ et $u_2$ des éléments inversibles de $A_1$ et $A_2$
respectivement.~Sup\-po\-sons que $\t_2$ soit le transfert de $\t_1$.
Pour chaque $i\in\{1,2\}$,
notons $\s_i$ l'automorphisme intérieur~de conjugaison par $u_i$,
supposons que $\aa_i$ soit stable par $\s_i$ et que
$\s_i(\b_i)=-\b_i$.  
Alors $\t_1\circ\s_1=\t_1^{-1}$ si et seulement si $\t_2\circ\s_2=\t_2^{-1}$.
\end{lemm}

\begin{proof}
Dans un premier temps, 
supposons simplement que $[\aa_1,\b_1]$ et $[\aa_2,\b_2]$ soient des strates 
simples maximales dans $A_1$ et $A_2$ et que $\t_2$ soit le transfert de 
$\t_1$,
sans supposer que $\aa_i$ soit stable par $\s_i$ ni que $\s_i(\b_i)=-\b_i$.  
Désignons par $\boldsymbol{{\sf t}}$ 
le transfert de $\Cc(\aa_1,\b_1)$ à $\Cc(\aa_2,\b_2)$.
On en déduit
\begin{itemize}
\item 
d'une part des strates simples maximales $[\s_1(\aa_1),\s_1(\b_1)]$ et 
$[\s_2(\aa_2),\s_2(\b_2)]$
et~un transfert $\boldsymbol{{\sf t}}'$ 
de $\Cc(\s_1(\aa_1),\s_1(\b_1))$~à 
$\Cc(\s_2(\aa_2),\s_2(\b_2))$,
\item
et d'autre part des strates simples maximales $[\aa_1,-\b_1]$ et $[\aa_2,-\b_2]$ et 
un transfert $\boldsymbol{{\sf t}}''$ de $\Cc(\aa_1,-\b_1)$ à 
$\Cc(\aa_2,-\b_2)$.
\end{itemize} 
D'après la remarque~\ref{transfertconj},
le transfert de $\Cc(\aa_i,\b_i)$~à $\Cc(\s_i(\aa_i),\s_i(\b_i))$ 
est donné par la conju\-gai\-son par $u_i$. 
Les propriétés~de compo\-si\-tion des transferts font donc que 
$\boldsymbol{{\sf t}}'\circ\s_1$ est égal à $\s_2\circ\boldsymbol{{\sf t}}$.

{Supposons maintenant que $\aa_i$ soit stable par $\s_i$ et que
$\s_i(\b_i)=-\b_i$.}
Alors \cite[Theo\-rem 7.8]{BSS} assure que
$\boldsymbol{{\sf t}}''$ est égal à $\boldsymbol{{\sf t}}'$.
Enfin, supposant que $\t^{\phantom{'}}_1\circ\s^{\phantom{'}}_1=\t_1^{-1}$, 
on~a~:
\begin{equation*}
\t_2\circ\s_2
=\boldsymbol{{\sf t}}'(\t_1\circ\s_1)
=\boldsymbol{{\sf t}}''(\t_1^{-1})
=\boldsymbol{{\sf t}}(\t_1^{\phantom{'}})^{-1}
=\t_2^{-1},
\end{equation*}
l'avant-dernière égalité venant de ce que le transfert est compatible à 
l'inversion. 
\end{proof}

\subsection{}
\label{invTTmoinsproto}

Nous terminons cette section par
des lemmes généraux qui seront très utiles
par la suite. 

\begin{lemm}
\label{etakappasigmageneral}
\label{etakappasigmageneralnonsplit}
\label{kappatauautoduale}
Soit $\tau$ un automorphisme continu de $G$,
soit $\t$ un caractère simple maximal~de $G$ tel que $\t\circ\tau=\t^{-1}$, 
soient $\BJI_\t$ son normalisateur dans $G$
et $\n$ sa représentation de Heisenberg. 
\begin{enumerate}
\item 
La représentation $\n^{\vee\tau}$ est isomorphe à $\n$. 
\item 
Pour toute représentation $\bk$ de $\BJI_\t$ prolongeant $\n$,
il y a un unique caractère $\bx$ de $\BJI_\t$ trivial sur $\BJ_\t^1$
tel que $\bk^{\vee\tau}$ soit isomorphe à $\bk\bx$. 
\item
Notons ${\rm val}_F$ la valuation sur $F$ et supposons que~:
\begin{equation}
\label{condsigmanonsplit}
{\rm val}_F \circ \Nrd_{A/F} \circ\, \tau = {\rm val}_F \circ \Nrd_{A/F}.
\end{equation}
Il existe une représentation $\bk$ de $\BJI_\t$ prolongeant $\n$ 
telle que $\bk^{\vee\tau}$ soit isomorphe à $\bk$. 
\end{enumerate}
\end{lemm}

\begin{proof}
L'hypothèse faite sur $\t$ entraîne que $\BJ_\t$ est stable par $\tau$,~et
il en~va de même de~$\BJ^0$, son sous-groupe compact maximal, 
et de $\BJ^1$, son pro-$p$-sous-groupe~dis\-tin\-gué maximal.~En
outre,
la restriction de $\n^{\vee\tau}$ à $H^1_{\t}$ contient $\t^{-1}\circ\tau=\t$.
Le premier point se déduit de la pro\-prié\-té d'unicité de $\n$.

Ensuite, soit $\bk$ une représentation de $\BJ_\t$ prolongeant $\n$.
L'existence et l'unicité du caractère $\bx$ vient de \eqref{fannyv}. 

Enfin,
pour prouver (3),
on choisit $\bk$ dont le déterminant soit d'ordre une puissance de $p$~(voir
le lemme \ref{bkpprimaire}).
Il lui correspond un caractère $\bx$
dont l'ordre est une puissance de $p$. 
Soit $[\aa,\b]$ une strate simple maximale telle que 
$\t\in\Cc(\aa,\b)$, et posons $E=\F[\b]$.
Identifions la~restriction de~$\bx$ à $\BJ^0$ à un caractère 
de $\BJ^0/\BJ^1$,
et identifions ce groupe-ci à $\GL_m(\ee)$,
où $\ee$ est~le corps~résiduel de $C$ et $m$ est défini par \eqref{egalizink}.
Comme $p\neq2$,
ses caractères se factorisent par le déter\-mi\-nant~:
ils sont donc d'ordre premier à $p$.
Par~con\-séquent, $\bx$ est trivial sur $\BJ^0$.

Enfin,
rappelons que $\BJ_\t$ est engendré par $\BJ^0$ et une uniformisante 
$\w$ de $C$ (selon la remarque \ref{malinlechien}).
Un caractère~de~$\BJ_\t$ trivial sur $\BJ^0$ est donc déterminé par~sa
valeur en $\w$. 
\'Ecrivons $\tau(\w)$ sous la forme $\w^{k}x$ pour un $k\in\ZZ$~et~un $x\in\BJ^0$.
La valuation de la norme réduite de $\w$ est égale à $mf_{E/F}$
et celle de $x$ est nulle~car $\BJ^0$ est~compact. 
Compte tenu de \eqref{condsigmanonsplit}, on trouve que $k=1$.
Choisissons 
un~caractère $\boldsymbol{\zeta}$ de~$\BJ_\t$~tri\-vial~sur $\BJ^0$
tel que~$\boldsymbol{\zeta}^2=\bx$,
ce qui équivaut à 
$\boldsymbol{\zeta}(\boldsymbol{\zeta}\circ\tau)=\bx$ d'après
ce qui~précède.~La~re\-pré\-sen\-ta\-tion
$\bk\boldsymbol{\zeta}$ a la propriété voulue.
\end{proof}

\begin{rema}
\label{etakappasigmageneralrema}
Si $\tau$ est une involution de $G$,
la propriété d'unicité du caractère $\bx$ donnée par le lemme 
\ref{etakappasigmageneral}(2)
implique que celui-ci vérifie l'identité $\bx\circ\tau=\bx$. 
\end{rema}

Sous les mêmes hypothèses que le lemme précédent,
on a le lemme suivant. 

\begin{lemm}
\label{etatau}
\label{kappatau}
Notons $\BJ^1$ le pro-$p$-sous-groupe~dis\-tin\-gué maximal de $\BJ_\t$. 
\begin{enumerate} 
\item
Le $\FC$-espace vectoriel $\Hom_{\BJ^{1}\cap\G^\tau}(\n,\FC)$
est de dimension $1$.
\item
Soit $\bk$ une représentation de $\BJI_{\t}$ prolongeant $\n$ et 
soit $\bx$ le caractère de $\BJ_\t$ qui lui est associé par le
lemme \ref{kappatauautoduale}(2).  
\begin{enumerate} 
\item 
Il y a un unique caractère $\chi$ de $\BJI_{\t}\cap\G^\tau$ trivial sur 
$\BJ^{1}\cap\G^\tau$ tel que~:
\begin{equation*}
\Hom_{\BJ^{1}\cap\G^\tau}(\n,\FC) = \Hom_{\BJI_{\t}\cap\G^\tau}(\bk,\chi^{-1})
\end{equation*}
et la restriction de $\bx$ à $\BJI_{\t}\cap\G^\tau$ est égale à $\chi^2$.
\item
Soit $\bt$ une représentation de~$\BJI_{\t}$ triviale sur~$\BJ^1$.
L'application linéaire canonique~:
\begin{equation*}
\Hom_{\BJ^{1}\cap\G^\tau}(\n,\FC) \otimes
\Hom_{\BJI_{\t}\cap\G^\tau}(\bt,\chi)
\to \Hom_{\BJI_{\t}\cap\G^\tau}(\bk\otimes\bt,\FC) 
\end{equation*}
est un isomorphisme. 
\end{enumerate}
\end{enumerate}
\end{lemm}

\begin{proof}
La preuve de (1) est la même que pour \cite[Proposition 6.12]{VSANT19} 
(voir aussi~\cite[Pro\-po\-sition 6.14]{ZouU}).
La preuve de (2.b) et de la première partie de (2.a)
est la même que pour~\cite[Lemma 6.20]{VSANT19}.
Pour prouver la seconde partie de (2.a), on écrit~:
\begin{equation*}
\Hom_{\BJ^{1}\cap\G^\tau}(\n,\FC)
= \Hom_{\BJI_{\t}\cap\G^\tau}(\bk,\chi^{-1})
\simeq \Hom_{\BJI_{\t}\cap\G^\tau}(\chi,\bk^{\vee\tau})
\simeq \Hom_{\BJI_{\t}\cap\G^\tau}(\chi\bx^{-1},\bk)
\end{equation*}
le premier isomorphisme étant obtenu par application du foncteur
$\pi\mapsto\pi^{\vee\tau}$,
et on conclut grâce au fait que $\BJI_{\t}\cap\G^\tau$ est compact
modulo le centre et à l'unicité de $\chi$.
\end{proof}

\section{Représentations cuspidales autoduales de $\GL_{2n}(F)$}
\label{SEC3}

Dans cette section,
nous étudions les propriétés des représentations 
cuspidales autoduales de
$\GL_{2n}(F)$ du point de vue de la théorie des types. 
Plus précisément,
nous montrons que~toute~re\-pré\-sentation cuspidale autoduale de
$\GL_{2n}(F)$ contient un caractère simple
-- et même un type~--
possédant des propriétés de dualité remarquables.

\subsection{} 
\label{P11}

Fixons un entier $n\>1$, et écrivons $\G=\GL_{2n}(\F)$ et $\A=\Mat_{2n}(\F)$.
Notre point de départ sera le résultat suivant de Blondel.

\begin{theo}[\cite{BlondelAENS04} Theorem 1]
\label{Blondel1}
Soit $\pi$ une représentation cuspidale autoduale de niveau non nul de $\G$. 
Il existe une strate simple $[\aa,\b]$ dans $\A$, 
un caractère simple $\t\in\Cc(\aa,\b)$ contenu dans $\pi$
et un élément $u\in\aa^\times$ tels que 
$\t^u=\t^{-1}$ et $\b^u=-\b$.
\end{theo}

Faisons tout de suite quelques remarques sur ce résultat. 

\begin{rema}
\label{firstremark}
\label{secondremark}
Posons $E=\F[\b]$, notons $\B$ le cen\-tralisateur de $\E$ dans $\A$
et $\BJ$ le norma\-li\-sateur de $\t$ dans $\G$.
\begin{enumerate}
\item
Le centralisateur de $u$ dans $E$ est le sous-corps
$\E_0=\F[\b^2]$, l'extension
$\E/\E_0$ est~quadra\-tique et $\Ad(u)$,
l'automorphisme intérieur de conjugaison par $u$ dans~$G$,
induit sur $\E$~l'automor\-phisme non trivial de $\E/\E_0$.
En outre,
on a $u\notin\BJ$ et $u^2\in\aa^\times\cap\B=\bb^\times$.
\item
Si $(\BJ,\bl)$ est le type (unique d'après la remarque \ref{chafouin})
contenu dans $\pi$ auquel le caractè\-re~sim\-ple $\t$ est attaché,
alors $\BJ^u=\BJ$ et $\bl^u$ est isomorphe à $\bl^\vee$. 
\end{enumerate}
\end{rema}

\subsection{} 
\label{P12}

Nous allons préciser le théorème \ref{Blondel1},
en montrant qu'on peut prendre pour $u$ un élément d'une forme bien précise. 

Soit $\ss$ un élément de $\G$ tel que $\ss^2=1$ et dont les valeurs 
propres $-1$ et $1$ aient même~multi\-plicité $n$, 
c'est-à-dire que le polynôme caractéristique 
de $\s$ sur $F$ est égal à $(X^2-1)^{n}$.
Le~cen\-tra\-lisateur de $\ss$ dans $G$ est
un sous-groupe de Levi 
conjugué à $\GL_n(F)\times\GL_n(F)$.
L'au\-to\-mor\-phisme~in\-té\-rieur $\Ad(\ss)$ de $\G$
sera simplement noté $\ss$. 
Ainsi, on notera $\ss(x)=\ss x\ss^{-1}$ pour tout élément $x\in\G$.

\begin{defi}
\label{defsad}
\begin{enumerate}
\item 
Un caractère simple $\t$ de $\G$ est $\s$-\textit{autodual} 
si $\t\circ\s=\t^{-1}$.
\item 
Un type $(\BJ,\bl)$ de $\G$ est $\s$-\textit{autodual} 
si $\BJ$ est stable par $\ss$ et $\bl^\ss$ est~iso\-mor\-phe à $\bl^\vee$.
\item 
Une strate simple $[\aa,\b]$ de $\A$ est $\s$-\textit{autoduale} 
si $\aa$ est stable par $\ss$ et $\s(\b)=-\b$.
\end{enumerate}
\end{defi}

On a le résultat suivant,
que l'on comparera à \cite[Theorem 4.2, Corollary 4.21]{AKMSS}.

\begin{prop}
\label{redformadap}
Soit $\TT$ une endo-classe autoduale non nulle, de degré divisant $2n$.
\begin{enumerate}
\item 
Il~y a une strate simple maxi\-ma\-le $\s$-auto\-duale $[\aa,\b]$ de 
$\A$ ayant les propriétés suivantes~:
\begin{enumerate}
\item 
L'ensemble $\Cc(\aa,\b)$ contient 
un caractère simple $\s$-autodual
d'endo-classe $\TT$.
\item
Posant $\E=\F[\b]$, $E_0=\F[\b^2]$ et $m=2n/[E:F]$, 
il y a un isomor\-phisme~:
\begin{equation}
\label{empereur}
\A\to\Mat_m(\End_F(E))
\end{equation}
de $\F$-algè\-bres identifiant 
$\aa\cap\B$ à l'ordre maximal standard $\Mat_m(\Oo_E)$
et l'action de $\s$ sur~$\B$
à celle de l'au\-to\-mor\-phisme~non trivial de $E/E_0$ sur $\Mat_m(E)$.  
\end{enumerate}
\item
Pour tout isomorphisme \eqref{empereur} comme en {\rm (1.b)},
notant $\ee$ le corps résiduel de $E$ et $\ee_0$ celui de $\E_0$, 
l'action~in\-dui\-te par $\s$ sur 
$\BJ^0(\aa,\b)/\BJ^1(\aa,\b)$ s'identifie \textit{via} \eqref{JJ1UU1GLmax}
à l'action du générateur~de $\Gal(\ee/\ee_0)$ sur $\GL_m(\ee)$. 
\end{enumerate} 
\end{prop}

\begin{proof}
Notons $d$ le degré de $\TT$. 
Soit $\t_{\bullet}\in\Cc(\aa_{\bullet},\b_{\bullet})$
un caractère simple d'endo-classe $\TT$, 
où $[\aa_{\bullet},\b_{\bullet}]$ est une strate simple maximale de
$\Mat_d(F)$. 
Posons $G_{\bullet}=\GL_d(F)$.

\begin{lemm}
\label{petitpois}
{Il y a une représentation irréductible cuspidale autoduale de $\G_{\bullet}$ 
contenant~$\t_{\bullet}$.}
\end{lemm}

\begin{proof}
Comme $\TT$ est autoduale,
il y a un $v\in\G_{\bullet}$ tel que $\t_{\bullet}^{-1}$
soit égal à~$\t_{\bullet}^{v}$.
Selon le lemme \ref{etakappasigmageneral}
appliqué à l'automorphisme intérieur de conjugaison par $v$, 
qui vérifie bien la~con\-dition \eqref{condsigmanonsplit},
il existe une représentation $\bk_{\bullet}$,
prolongeant la représentation de Heisenberg de $\t_{\bullet}$
à son normalisateur $\BJ_{\bullet}$ dans $\G_{\bullet}$, 
et telle que $\bk_{\bullet}^{\vee v}$ soit isomorphe à $\bk_{\bullet}$. 
Selon le théorème \ref{THEOCLASSIFTYPES},
l'induite compacte de $\bk_{\bullet}$ à $\G_{\bullet}$
est une représentation cuspidale autoduale de $\G_{\bullet}$ 
contenant le caractère simple $\t_{\bullet}$.
\end{proof}

Soit donc $\pi_{\bullet}$ une représentation cuspidale autoduale de $\G_{\bullet}$ 
contenant $\t_{\bullet}$.
Comme son endo-classe est non nulle,
\cite[Theorem 1]{BlondelAENS04} assure que $d$ est pair.
Le théorème \ref{Blondel1}~ap\-pli\-qué à~$\pi_{\bullet}$~four\-nit
une strate simple $[\aa_*,\b]$ de $\Mat_d(F)$,
un caractère simple $\t_*\in\Cc(\aa_*,\b)$ et~un élément
$u\in\aa_*^\times$ tels que $\t_*^{u}=\t_*^{-1}$
et $\b^u=-\b$.
On pose $E=F[\b]$ et $E_0=\F[\b^2]$. 

\'Ecrivons $2n=md$. 
Quitte à remplacer $\s$ par un de ses conjugués,
ce qui ne change rien au résultat à prouver,
on peut identifier $A$ à $\Mat_m(\Mat_d(F))$ et mettre $\s$ 
sous la forme~:
\begin{equation}
\label{redm1xiaomi}
{\rm diag}(\s_{*},\dots,\s_{*}),
\quad
\text{$\s_{*}\in\GL_d(F)$ de polynôme caractéristique $(X^2-1)^{d/2}$}.
\end{equation}
Ensuite, identifions les $F$-agèbres $\Mat_{d}(F)$ et $\End_F(E)$ 
de sorte que $\s_{*}$ soit l'automorphisme~défi\-ni~par 
$\s_{*}(x+\b y) = x-\b y$ pour tout $x,y \in \E_0$. 
Ainsi $\s_*$ normalise $E$, 
et induit sur $E$~l'au\-to\-mor\-phisme~non trivial de $E/E_0$.
Comme $E$ est maximal dans $\Mat_d(F)$,
l'ordre $\aa_*$ est l'unique ordre~princi\-pal de $\Mat_d(F)$ 
norma\-li\-sé~par $E^\times$,
qui est donc normalisé par $\s_*$.
La strate simple~maxi\-ma\-le~$[\aa_*,\b]$ est donc
$\s_*$-au\-to\-duale.
L'in\-vo\-lution $\s_*$ agit sur le cen\-tralisateur de $E$ 
dans $\Mat_d(F)$, égal à $E$, 
com\-me l'au\-to\-mor\-phisme~non trivial de $E/E_0$,
et sur le quotient $\BJ^0(\aa_*,\b)/\BJ^1(\aa_*,\b)$,~na\-turelle\-ment
isomorphe à 
$\Oo_E^\times/(1+\p_E)=\ee^\times$,
comme le générateur~de $\Gal(\ee/\ee_0)$. 

Enfin,
posons $\aa=\Mat_m(\aa_*)$ et identifions $E$ à son image diagonale dans $A$. 
Notons $\t\in \Cc(\aa,\b)$ le~transfert de $\t_*$,
qui est $\s$-autodual grâce au lemme \ref{lemmedetransfertautodual}.
On obtient une strate~sim\-ple maxi\-ma\-le
$\s$-autoduale $[\aa,\b]$
de $\A$ possédant les propriétés voulues. 
\end{proof}

On en déduit le corollaire suivant,
que l'on comparera à \cite[Theorem 4.1]{AKMSS},
\cite[Remark 5.13]{VSANT19}.

\begin{coro}
\label{MAIN1}
Pour qu'une représentation cuspidale de $\G$ contienne un type 
$\s$-autodual,~il faut et suffit qu'elle soit autoduale. 
En particulier, 
toute représentation cuspidale autoduale de~$\G$ contient 
un caractère simple $\s$-autodual. 
\end{coro}

\begin{proof}
Soit $\pi$ une représentation cuspidale de $\G$.
Si $\pi$ contient un type $\s$-autodual $(\BJ,\bl)$ de $G$,
sa contragrédiente $\pi^\vee$ contient le type $(\BJ,\bl^\vee)$.
On a donc~:
\begin{equation*}
\pi^\vee\simeq\ind^\G_{\BJ}(\bl^\vee)\simeq
\ind^\G_{\BJ}(\bl^\s) \simeq \pi.
\end{equation*}
Supposons maintenant que $\pi$ soit autoduale. 

Si la représentation $\pi$ est de niveau $0$, elle contient un type $(\BJ,\bl)$
avec $\BJ=\F^\times\GL_{2n}(\Oo_\F)$.
On peut~sup\-po\-ser que $\ss$ est diagonal,
auquel cas il est dans $\BJ$.
Alors $\pi$ contient à la fois $\bl\simeq\bl^\s$~et $\bl^\vee$,~qui
sont donc conjugués par un élément $g\in\G$ normalisant $\BJ$.
On a donc $g\in\BJ$, et il s'ensuit que $(\BJ,\bl)$ est $\s$-autodual.

{Supposons maintenant que $\pi$ soit de niveau non nul. 
Son endo-classe, notée $\TT$, est non nulle et auto\-duale. 
D'après la proposition \ref{redformadap},
il existe une strate simple maximale $\s$-autoduale $[\aa,\b]$ dans
$\A$~et un caractère simple (maximal) $\s$-autodual $\t\in\Cc(\aa,\b)$
d'endo-classe $\TT$.
En particulier,
le normalisateur $\BJ$ de $\t$ dans $\G$ est normalisé par $\s$.
D'après le corollaire \ref{entrenouscor},
le caractère $\t$~est contenu dans $\pi$.
Soit $(\BJ,\bl)$ le type contenu dans $\pi$ auquel $\t$ est
attaché (remarque \ref{chafouin}(3)).
Les types $(\BJ,\bl^\s)$ et $(\BJ,\bl^\vee)$ sont 
contenus dans $\pi$ et contiennent $\t^\s=\t^{-1}$.
Ils sont donc (d'après le théorème \ref{THEOCLASSIFTYPES})
conjugués par un $g\in\G$ normalisant ce caractère~:
on a donc $g\in\BJ$,
\ie que $(\BJ,\bl)$ est $\s$-autodual.}
\end{proof}

\begin{rema}
\label{lll}
La preuve du corollaire \ref{MAIN1} montre que toute
représentation cuspidale auto\-duale de $\G$
de niveau non nul
contient non seulement un type $\s$-autodual,
mais plus précisément un
type $\s$-auto\-dual associé~à une strate simple~$\ss$-autoduale 
$[\aa,\b]$ dans $\A$ possédant la proprié\-té (1.b) 
-- donc aussi la propriété (2) -- de la proposition \ref{redformadap}.
Il est naturel de se demander~si~un caractère simple maximal
-- ou un type --
$\s$-autodual est~tou\-jours
associé à une strate simple $\ss$-autoduale.
Le lemme \ref{gandalf} répondra~ à cette question par l'affirmative
pour les caractères,
dans un cadre plus~gé\-né\-ral.~Une telle strate ne vérifiera 
cependant pas toujours~les propriétés (1.b)~et 
(2) de la proposi\-tion~\ref{redformadap}~:
on renvoie à la proposition \ref{indiceduntype} pour une classification complète.
\end{rema}

\subsection{}
\label{invTTmoins}
\label{invTT}

Soit $\TT$ une endo-classe autoduale non nulle
de degré divisant $2n$.
Selon la proposition~\ref{redformadap},
il existe une strate simple maxi\-ma\-le $\s$-auto\-duale $[\aa,\b]$ dans 
$\A$ et un
caractère simple $\s$-autodual $\t\in\Cc(\aa,\b)$ d'endo-classe 
$\TT$.
Notons $\BJ=\BJ(\aa,\b)$, ainsi que $\BJ^0=\BJ^0(\aa,\b)$ et
$\BJ^1=\BJ^1(\aa,\b)$.~Le~ca\-ractère $\t$ étant $\s$-autodual,
son normalisateur $\BJ$ est stable par $\s$,
et il en va de même de~$\BJ^0$, son sous-groupe compact maximal, 
et de $\BJ^1$, son pro-$p$-sous-groupe distingué maximal.

Posons $E=F[\b]$ et $E_0=\F[\b^2]$.
Soit $T$ la sous-extension~modéré\-ment ramifiée maximale de $E$ sur $F$. 
Alors $T_0=\T\cap\E_0$
est la sous-extension modérément~rami\-fiée ~maximale de $E_0$ sur $F$
et l'extension $T/T_0$ est quadratique. 

\begin{rema}
\label{invTTrem}
D'après \cite[Proposition 6.7]{BHS}, 
la classe de $F$-isomorphisme de $T/T_0$ dépend uni\-quement de $\TT$, 
et ni du choix de $n$ ni de celui de $[\aa,\b]$.
Voir \cite[6.3, 6.5]{BHS} 
(et le paragraphe \ref{voiraussi110} de l'introduction)
pour une interprétation de cette extension 
quadratique du côté galoisien
de~la correspondance de Langlands locale. 
\end{rema}

Les lemmes suivants joueront un rôle important par la suite.

\begin{lemm}
\label{lemmedeparite}
Posons $m=2n/\deg(\TT)$.
Pour qu'il existe une représentation cuspidale auto\-duale de $\G$ 
d'endo-classe $\TT$,
il faut et suffit que $m$ vérifie les conditions suivantes~:
\begin{enumerate}
\item
si $\T/\T_0$ est non ramifiée, alors $m$ est impair,
\item
si $\T/\T_0$ est ramifiée, alors $m$ est pair ou égal à $1$. 
\end{enumerate}
\end{lemm}

\begin{proof}
On peut supposer, ce que l'on fera, que la strate simple $[\aa,\b]$ fixée ci-dessus
possède également la pro\-priété (1.b) de la proposition \ref{redformadap}.
Le degré de $\TT$ est égal à $[\E:\F]$.
  
Soit $\n$ la représentation de Heisenberg associée à $\t$,
et soit $\bk$ une représentation de $\BJI$ prolon\-geant $\n$ telle que
$\bk^{\vee\s}$ soit isomorphe à $\bk$,
dont l'existence est assurée par le lemme \ref{etakappasigmageneral}
appliqué à l'involution $\s$.

Supposons d'abord qu'il y ait une représentation cuspidale auto\-duale
$\pi$ de $\G$ d'endo-classe~$\TT$.
Selon le corollaire \ref{entrenouscor},
elle contient $\t$.  
Soit $(\BJ,\bl)$ le type contenu dans $\pi$ auquel $\t$ est
attaché (voir la remarque \ref{chafouin}(3)).
Il est $\s$-autodual (voir~la preuve du corollaire \ref{MAIN1}).
Soit $\bt$ la repré\-sen\-ta\-tion irréductible de
$\BJ$ triviale sur $\BJ^1$ telle que $\bl$ soit iso\-mor\-phe à
$\bk\otimes\bt$.
Par unicité,
$\bt$~est $\s$-autoduale.
La représentation de $\GL_m(\ee)$
dont la restriction de $\bt$ à $\BJ^0$ est l'inflation,
notée $\rho$, est également $\s$-autoduale.

Si $T/T_0$ est non ramifiée,
$\s$ agit sur $\GL_m(\ee)$ comme l'automorphisme non trivial de $\ee/\ee_0$.
Dans ce cas, le résultat est donné par \cite[Lemma 2.3]{VSANT19} 
appliqué à $\rho$.

Si $T/T_0$ est ramifiée,
$\s$ agit trivialement sur $\GL_m(\ee)$,
\ie que $\rho$ est autoduale. 
Dans~ce cas, le résultat est donné par \cite[Lemma 2.17]{VSANT19} 
appliqué à $\rho$.

Inversement,
supposons que l'entier $m$ vérifie les conditions de l'énoncé.
Invoquant à nouveau \cite[Lemmas 2.3, 2.17]{VSANT19}, 
on en déduit l'existence d'une représentation cuspidale
$\s$-au\-toduale~$\rho$~de $\GL_m(\ee)$.
L'action de $E^\times$ par conjugaison sur $\BJ^0/\BJ^1$ étant triviale,
l'infla\-tion~de $\rho$ à $\BJ^0$ se prolon\-ge~à $\BJ$.
Choisissons un tel prolongement $\bt$.
La représentation $\bl=\bk\otimes\bt$ est un type auquel $\t$ est 
attaché. 
D'après le théorème \ref{THEOCLASSIFTYPES},
son induite compacte à $G$, notée $\pi$,
est~une représentation cus\-pidale d'endo-classe $\TT$
telle que $\pi^\vee$, $\pi$ soient inertiellement équiva\-len\-tes.
Tordant $\pi$ par~un caractère non ramifié convenable de $G$,
on obtient une représentation cuspidale auto\-duale de $\G$ 
d'endo-classe $\TT$. 
\end{proof}

Terminons cette section en énonçant
l'analogue du lemme \ref{lemmedeparite} pour les représentations
cuspi\-dales autoduales de niveau $0$, 
issu de \cite[Theorem~7.1]{Adler} 
(qui omet le cas de $F^\times$,
qui a quatre~re\-pré\-sentations cuspidales autoduales du fait que $p\neq2$). 

\begin{lemm}
\label{lemmedepariteniveauzero} 
Soit un entier $m\>1$. 
Pour qu'il y ait une représentation cuspidale auto\-duale de $\GL_m(F)$ 
de niveau $0$, il faut et suffit que $m$ soit pair ou égal à $1$.  
\end{lemm}

\section{Involutions intérieures sur les formes intérieures 
de $\GL_{2n}(F)$} 
\label{SEC4}

Dans cette section, nous élargissons le cadre de notre étude
en remplaçant le groupe $\GL_{2n}(F)$ par une forme inté\-rieure
et l'involution $\s$ par une involution plus générale. 
Ce nouveau cadre~est fixé au paragraphe \ref{cyprien}.

\subsection{}
\label{cyprien}

Soit $\A$ une $\F$-algèbre centrale simple de degré réduit $2n$.
Fixons un $A$-module à gauche~sim\-ple $V$, 
et notons $D$ la $F$-algèbre opposée à $\End_A(V)$.
C'est une $F$-algèbre à division centrale,~$V$
est un $D$-espace vectoriel à droite
et $A$~s'iden\-ti\-fie naturellement à la $F$-algèbre $\End_D(V)$.
Notons $r$ la dimension de $V$ sur $D$
et $d$ le degré réduit de $D$ sur $F$,
de sorte que $2n=rd$.

Posons $\G=\mult\A$,
fixons un $\a\in\F^\times$
et fixons~un $\dd\in\G$ non central tel que $\dd^2=\a$ .
Ceci~définit une involu\-tion 
$\tau:g\mapsto\dd g\dd^{-1}$ du groupe $\G$.
Remplacer $\dd$ par $x\dd x^{-1}$ pour un $x\in\G$
a pour effet de changer $\tau$ en l'involution conjuguée~:
\begin{equation*} 
x\cdot\tau 
= \Ad(x)\circ\tau \circ\Ad(x)^{-1}
\end{equation*}
et on a $\G^{x\cdot\tau}=x\G^{\tau}x^{-1}$.
Par ailleurs, 
remplacer $\dd$ par $\l\dd$ pour un $\l\in\F^\times$ a pour effet de
changer $\a$ en $\a\l^{2}$ sans changer $\tau$.
On peut donc considérer $\a$ modulo le sous-groupe $\F^{\times2}$.

\begin{lemm}
\label{phlebel}
Soit $\Oo(\a)$ l'ensemble des éléments $\dd\in\G$ non centraux
tels que $\dd^2=\a$. 
\begin{enumerate}
\item 
Si $\a$ n'est pas un carré de $\F^\times$, 
alors $\Oo(\a)$ est formé d'une seule clas\-se de $\G$-conju\-gaison.
\item
S'il existe un $\l\in\F^{\times}$ tel que $\a=\l^2$, alors
pour tout entier $i\in\{1,\dots,r-1\}$,
il existe~une uni\-que classe de $\G$-conjugaison dans $\Oo(\a)$
dont le polynôme caractéristique réduit sur $F$ soit~égal à
$(X+\l)^{di}(X-\l)^{d(r-i)}$.
\end{enumerate}
\end{lemm}

\begin{proof}
L'assertion (1) est une conséquence du théorème de~Sko\-lem-Noether.
S'il~y~a un $\l\in\F^\times$ tel que $\a=\l^2$,
on se ramène au cas où $\a=1$ en~appli\-quant $\dd\mapsto\dd\l^{-1}$.
Supposons donc que $\a=1$,
et soit $\dd\in\G$ tel que $\dd^2=1$.
Considéré comme un $F$-endomor\-phis\-me de $V$,
il définit une décomposition $V=\Ker(\dd-{\rm id})\oplus\Ker(\dd+{\rm id})$
en sous-$F$-espa\-ces pro\-pres.
Comme $\dd$ est $D$-linéaire,
ces $F$-espaces propres sont des $D$-espaces vectoriels. 
Le~polynô\-me~ca\-rac\-téristique réduit de $\dd$ sur $F$
est donc de la forme~:
\begin{equation}
\label{formepcrd}
{\rm Pcrd}_{A/F}(\dd) = (X+1)^{di}(X-1)^{d(r-i)} \in \F[X]
\end{equation}
pour un unique $i\in\{0,\dots,r\}$, 
les éléments centraux correspondant à $i=0$ et $i=r$. 
Ceci prouve l'assertion (2).
\end{proof}

\begin{rema}
On observera en particulier que, si $A$ est une algèbre à division, 
\ie si~$r=1$,
les seuls $\s\in\G$ tels que $\s^2=1$ sont $1$ et $-1$,
\ie que $\Oo(1)$ est vide. 
\end{rema}

Par la suite, nous adopterons la convention suivante. 

\begin{enonce}{Convention}
\label{CONV2} 
Nous considérerons uniquement les cas suivants~:
\begin{enumerate}
\item
ou bien $\a\notin\F^{\times2}$,
\item
ou bien $\a=1$, $r$ est pair et ${\rm Pcrd}_{A/F}(\dd)=(X^2-1)^{n}$.
\end{enumerate} 
Dans tous les cas, le polynôme caractéristique réduit de $\dd$ sur $F$
est donc égal à $(X^2-\a)^{n}$.
\end{enonce}

\begin{rema}
\label{pluiedautomne}
En d'autres termes, 
si $\a$ est un carré de $\F^{\times}$,
nous ne traitons pas le cas~où ${\rm Pcrd}_{A/F}(\dd)$ est~de la forme 
\eqref{formepcrd} avec $i\neq r-i$.
Les méthodes employées dans cet article peuvent en principe 
être adaptées à ce cas~;
cependant,
si $i\neq r-i$,
on s'attend à ce qu'aucune~re\-pré\-sentation cuspidale de $\G$ ne soit 
distinguée par $\GL_i(D)\times\GL_{r-i}(D)$.
Voir \cite[Theorem 3.1]{MatringeCRAS} si $D$ est égal à $F$.
\end{rema}

\subsection{}
\label{tetrobot1}

Comme au paragraphe précédent,
on fixe un $\a\in\F^\times$ et un $\dd\in\G$ non central tel que
$\dd^2=\a$, dé\-fi\-nissant une involution $\tau$.
La définition suivante généralise naturellement la définition \ref{defsad}.

\begin{defi}
\label{deftad}
\begin{enumerate}
\item 
Un caractère simple $\t$ dans $\G$ est $\tau$-\textit{autodual} 
si $\t\circ\tau=\t^{-1}$.
\item 
Un type $(\BJ,\bl)$ dans $\G$ est $\tau$-\textit{autodual} 
si $\BJ$ est stable par $\tau$ et $\bl^\tau$ est~iso\-mor\-phe à $\bl^\vee$.
\item 
Une strate simple $[\aa,\b]$ dans $\A$ est $\tau$-\textit{autoduale} 
si $\aa$ est stable par $\tau$ et $\tau (\b)=-\b$.
\end{enumerate}
\end{defi}

Contrairement à ce que nous avons vu au paragraphe \ref{P12},
nous allons voir qu'il n'y a pas~tou\-jours,
étant donné une endo-classe autoduale non nulle $\TT$,
de caractère simple maximal $\tau$-auto\-dual 
d'endo-classe $\TT$ dans $G$.

\begin{lemm}
\label{gandalf}
\label{gandalfcor1}
Soit $\t$ un caractère simple maximal $\tau$-autodual dans $G$.
\begin{enumerate}
\item
Il y a une strate simple maximale
$\tau$-autoduale $[\aa,\b]$ dans $\A$ telle que $\t\in\Cc(\aa,\b)$.
\item
Soit $[\aa,\b]$ une strate simple $\tau$-autoduale dans $A$, 
et posons $E=\F[\b]$ et $\E_0=\F[\b^2]$.
\begin{enumerate}
\item
L'involution $\tau$ induit sur $E$ l'automorphisme non trivial de $E/E_0$. 
\item
Cette involution s'étend en une $E_0$-involution du centralisateur $B$ de $E$ 
dans $A$.
\item
L'intersection $\aa\cap\B$ est un ordre maximal de $B$ stable par $\tau$.
\end{enumerate}
\end{enumerate} 
\end{lemm}

\begin{proof}
Soit $\t$ un caractère simple maximal dans $G$, 
et soit $[\aa,\g]$ une strate simple dans $\A$ telle que 
$\t\in\Cc(\aa,\g)$. 
Alors $[\aa^\dd,-\g^\dd]$ est une strate simple de $\A$,
et $\t^{-1}\circ\tau$ est~un~carac\-tère~simple
dans $\Cc(\aa^\dd,-\g^\dd)$. 
Supposons maintenant que $\t$ soit $\tau$-autodual.
Alors la~propo\-sition \ref{rosimond} im\-pli\-que que $\aa$ est stable par $\tau$.
Puis, adaptant \cite[Proposition 1.10, Theorem 6.3]{SSPLMS01} 
au~cas de l'involution $\tau$,
on en déduit qu'il y a un $\b\in\A$
tel que $[\aa,\b]$ soit une strate simple $\tau$-autoduale dans $A$ 
et $\t\in\Cc(\aa,\b)$.
{(Voir \cite[\S 4.2]{SkodlerackRT20} 
pour l'adaptation de \cite{SSPLMS01} à un autre type d'involution de $G$.)}  
L'assertion (2) est immédiate. 
\end{proof}

Nous verrons que ces conditions imposent des contraintes à 
$\TT$ vis-à-vis de $A$ et de $\a$. 
Le~ré\-sultat principal de cette section,
le théorème \ref{THMEXISTENCECARACSIMPLE},
donne une con\-di\-tion néces\-saire et~suffisan\-te d'existence
d'un caractère simple maximal $\tau$-autodual dans une~re\-présenta\-tion
cus\-pi\-dale autoduale de $G$ de niveau non nul. 
Le cas des~re\-pré\-sentations de niveau $0$ sera traité 
à~la~sec\-tion~\ref{theriverofnoreturn}. 

\subsection{}
\label{Biotabstait0}
 
Les notations~du paragraphe \ref{cyprien} sont en vigueur. 
Soit $E_0$ une extension finie de $F$ de~de\-gré divisant $n$,
et soit $E$ une extension~qua\-dra\-tique de $E_0$.

\begin{lemm}
\label{gandalfcarre} 
Supposons que $r=2k$ pour un entier $k\>1$.
Fixons un élément $\s\in\G$ tel que $\s^2=1$,
de polynôme caractéristique réduit $(X^2-1)^n$.
Il existe un plongement de $F$-al\-gè\-bres 
de $E$ dans $A$ tel que la conjugaison par $\s$ induise
sur $\E$ l'automorphisme non trivial de $E/E_0$.
\end{lemm}

\begin{proof}
Fixons un générateur $\b$ de $\E$ sur $E_0$ tel que $\b^2\in\E_0$, 
et identifions $A$ à la $F$-algèbre $\Mat_r(D)$ de façon que~:
\begin{equation*}
\s = \begin{pmatrix} -{\rm id}&0\\0&{\rm id}\end{pmatrix} 
\in \Mat_{k}(D)\times \Mat_{k}(D) \subseteq \A
\end{equation*}
où ${\rm id}$ est l'identité de $\Mat_{k}(D)$.
La $F$-algèbre $A^\s$ s'identifie donc à
$\Mat_k(D)\times\Mat_k(D)$ dans laquelle il suffit de plonger $E_0$
diagona\-lement,
ce qui est possible car $[E_0:F]$ divise $kd=n$.
Il suffit~alors de plonger $E$ de façon que l'image de $\b$ soit de
la forme~: 
\begin{equation*}
\b = \begin{pmatrix} 0&\g\\1&0\end{pmatrix}
\end{equation*}
pour un $\g\in\E_0^\times$.
\end{proof}

\subsection{}

Nous traitons maintenant le cas plus difficile
où $\a$ n'est pas un carré de $\F^\times$,
en commençant par le cas particulier où $E_0=F$.

\begin{lemm}
\label{gandalf0}
Supposons que $\a$ ne soit pas un carré de $\F^\times$
et que $E_0=F$.
Pour qu'il~exis\-te un plongement de~$F$-al\-gèbres 
de $E$ dans $A$ tel que
la conjugai\-son par $\dd$
induise sur $\E$~l'auto\-mor\-phisme non trivial de $E/F$, 
il faut et suffit que~:
\begin{equation*}
\bo_{E/F}(\a)^{n} = (-1)^{r}
\end{equation*}
où $\bo_{E/F}$ est le caractère de $F^\times$ de noyau 
$\N_{E/F}(E^\times)$. 
\end{lemm}

\begin{proof}
Fixons un élément $\b\in\E^\times$ tel que $E=F[\b]$ et
$\g=\b^2\in\F^\times$.
Notons $H$ la $F$-algèbre~de dimension $4$ de base
$(1,\boldsymbol{i}, \boldsymbol{j}, \boldsymbol{k})$ avec les relations~:
\begin{equation*}
\boldsymbol{i}^2=\a,
\quad
\boldsymbol{j}^2=\g,
\quad
\boldsymbol{i}\boldsymbol{j}=-\boldsymbol{j}\boldsymbol{i}=\boldsymbol{k}.
\end{equation*}
C'est une $F$-algèbre centrale simple de degré réduit $2$.
Notons $(\a,\g)_F$ le symbole~de Hilbert~de $\a$ et~$\g$ sur $F$.
On a $(\a,\g)_F=\bo_{E/F}(\a)$.
La $F$-algèbre $H$ est~:
\begin{itemize}
\item[(a)] 
déployée si $(\a,\g)_F=1$, \ie si $\a\in\N_{E/F}(\E^\times)$,
\item[(b)]
une~algèbre~à~di\-vision si $(\a,\g)_F=-1$,
\ie si $\a\notin\N_{E/F}(\E^\times)$.
\end{itemize} 
Pour qu'il existe un plongement de~$F$-algè\-bres de $E$ dans $A$
vérifiant la condition demandée,~il 
faut et suffit qu'il existe un plongement de $F$-algèbres 
$\phi:H\to A$ tel que $\phi(\boldsymbol{i})=\dd$ et 
$\phi(\boldsymbol{j})=\b$. 
Comme $\a$ n'est pas un carré dans $\F^\times$,
le théorème de Skolem-Noether assure que deux~élé\-ments
de $A$ de carré égal à $\a$ sont con\-jugués sous $G$.
Nous aurons donc résolu notre~pro\-blème une fois déterminées les 
conditions auxquelles il existe un plongement de~$F$-algè\-bres
de $H$ dans $A$.

\begin{lemm}
\label{plonge1}
Soit un entier $k\>1$.
Pour qu'il y ait un plongement de $\Mat_{k}(F)$ dans $A$, 
il faut et suffit que $k$ divise $r$. 
\end{lemm}

\begin{proof}
C'est certainement une condition suffisante. 
Supposons que $\Mat_k(F)$ se plonge dans $A$. 
Son centralisateur $C$ étant de même classe de Brauer que $A$,
il est isomorphe à $\Mat_s(D)$ pour un certain $s\>1$,
et son degré réduit $N=sd$ vérifie $kN=rd$,
donc $k$ divise $r$.
\end{proof}

Par conséquent, 
si $H$ est déployée,
\ie si $\a\in\N_{E/F}(\E^\times)$,
il y a un plongement de $\H$ dans $A$ si et seulement si $r$ est pair. 

\begin{lemm}
Pour qu'il existe un plongement de $\H$ dans $A$, 
il  faut  et   suffit  qu'il  existe  un  plon\-gement   de  $\Mat_4(F)$  dans
$A\otimes_FH$. 
\end{lemm}

\begin{proof}
C'est une condition nécessaire car $H\otimes_FH\simeq\Mat_4(F)$.
Supposons que $\Mat_4(F)$ se plonge dans $A\otimes_FH$. 
Tensorisant par $H$, il s'ensuit que $\Mat_4(H)$ se plonge dans
$\Mat_4(A)$.
Soit $C$ son centralisateur.
On a un isomorphisme de $F$-algèbres 
$\Mat_4(A)\simeq\Mat_4(H)\otimes_FC$.
Considérons maintenant $H\otimes_FC$. 
C'est une $F$-algèbre centrale simple de même dimension et de même
classe de Brauer que $A$,~de sorte qu'elle lui est isomorphe.
La $F$-algèbre $\H$ se plonge donc dans $A$.
\end{proof}

Supposons maintenant que $H$ soit une algèbre~à~di\-vi\-sion.
Il s'agit de prouver que $H$ se plonge dans $A$ si et seulement si 
$r$ et $n$ ont même parité. 
Soit $h/d\in\QQ/\ZZ$ l'invariant de Hasse de $D$,~où 
$h$ est un entier premier à $d$.
Alors $A\otimes_FH$~est isomorphe à $\Mat_{v}(U)$,
où $v\>1$ est un entier, $U$~une $F$-algèbre à division centrale
de degré réduit $u$ et d'invariant de Hasse $k/u$ donnés par~:
\begin{equation*}
\frac k u = \frac h d + \frac 1 2 = \frac {2h+d} {2d},
\quad
u = \frac {2d} {(2d,2h+d)},
\quad
uv=2rd.
\end{equation*} 
Selon le lemme \ref{plonge1},
l'algèbre $\Mat_4(F)$
se plonge dedans si et seulement si $v=r(2d,2h+d)$ est~di\-visible par $4$.
C'est le cas si et seulement si
$4$ divise $r$, ou $r$ et $d$ sont pairs,
ou $d$ est divisible par $2$ mais pas par $4$, 
\ie si et seulement si $r$ et $n=rd/2$ ont même parité. 
\end{proof}

\begin{rema}
\label{lapinfacetieux} 
S'il y a un plongement de~$F$-al\-gèbres $\iota$ de $E$ dans~$A$
sa\-tis\-faisant aux~con\-di\-tions du lemme \ref{gandalf0},
{le centralisateur de $\dd$ et $\iota\E$ dans $A$}
est une $F$-algèbre~cen\-trale simple~de même classe de Brauer que
$\A\otimes_\F\H$, isomorphe à $\Mat_{w}(U)$ avec~:
\begin{equation*}
u = \left\{ 
\begin{array}{rl}
d & \text{si $\a\in\N_{E/F}(\E^\times)$
ou si
$4$ divise $d$}, \\
2d & \text{si $\a\notin\N_{E/F}(\E^\times)$ et $d$ est impair}, \\
d/2 & \text{si $\a\notin\N_{E/F}(\E^\times)$
et $d$ est est divisible par $2$ mais pas par $4$},
\end{array}\right.
\end{equation*}
et $v=4w$.
\end{rema}

\subsection{}
\label{pelleteuse}

Nous traitons maintenant le cas général.

\begin{prop}
\label{gandalf0coro}
Pour qu'il existe un~plon\-ge\-ment de $F$-algèbres 
de $E$ dans $A$ tel que $\tau$~in\-duise sur $\E$ l'auto\-mor\-phisme non trivial de $E/E_0$, 
il faut et suffit que~:
\begin{equation}
\label{bizuth}
\bo_{E/E_0}(\a)^{2n/[E:F]} = (-1)^{r}.
\end{equation} 
\end{prop}

\begin{proof}
Le cas où $\a$ est un carré dans $\F^\times$ découle du lemme
\ref{gandalfcarre} (car $r$ est pair~dans ce cas d'après la convention \ref{CONV2}).
Traitons maintenant le cas où $\a$ n'est pas un carré dans
$\F^\times$.
Il faut d'abord plonger $E_0$ dans le centralisateur de $\dd$ dans $A$.
La $F$-algèbre $K=\F[\dd]$ est~une extension quadratique~;
son centralisateur dans $A$ est donc une
$K$-algèbre cen\-tra\-le simple $S$~de degré réduit $n$. 
Y plonger $E_0$ équivaut à fixer un morphisme de $K$-algèbres de
$K\otimes_FE_0$ dans~$S$.
Soit $E_1$ un facteur direct~de $K\otimes_FE_0$.
C'est une extension de $K$ de~degré égal soit à $[E_0:F]$,
soit à $[E_0:F]/2$.
Dans les deux cas, on peut la plonger dans $S$ car $[E_0:F]$ divise $n$.
Ceci étant fait,
il reste~à plonger $E$ comme $E_0$-algèbre
dans le centralisateur $A_0$ de~$E_0$ dans $A$, qui 
est une $E_0$-algèbre~cen\-trale simple contenant $\dd$, 
de degré réduit $2n/[E_0:F]$.
Pour cela,~on applique le lemme \ref{gandalf0}
en remplaçant $F$, $A$ par $E_0$, $A_0$,~le
rôle de $r$ étant joué par $t=2n/{\rm ppcm}(d,[E_0:F])$.
On obtient donc la condition~:
\begin{equation*}
\bo_{E/E_0}(\a)^{2n/[E:F]} = (-1)^t.
\end{equation*} 
Il ne reste plus qu'à observer que $t$ est pair si et seulement si
ppcm$(d,[E_0:F])$ divise $n$,~c'est-à-dire
si et seu\-lement si $d$ divise $n=rd/2$, \ie si et seulement si $r$ est pair.
\end{proof}

\begin{rema}
\label{confino}
\begin{enumerate}
\item 
Selon la remarque \ref{lapinfacetieux},
si l'on pose $g=[E_0:F]$ et
s'il y a un~plon\-gement de~$F$-al\-gèbres $\iota$ de $E$~dans $A$
sa\-tis\-faisant aux conditions de la proposition \ref{gandalf0coro},
le~cen\-tralisateur $B_0$ de $\dd$ et $\iota\E$ dans $A$ est
isomorphe à une algèbre de matrices à coefficients dans une 
$E_0$-algèbre à division centrale $C_0$
dont le degré réduit est égal à~:
\begin{equation*}
c_0 = \left\{ 
\begin{array}{rl}
d/(d,g) & \text{si $\a\in\N_{E/E_0}(\E^\times)$
ou si 
$4$ divise $d/(d,g)$}, \\
2d/(d,g) & \text{si $\a\notin\N_{E/E_0}(\E^\times)$ et $d/(d,g)$ est impair}, \\
d/2(d,g) & \text{si $\a\notin\N_{E/E_0}(\E^\times)$
et $d/(d,g)$ est divisible par $2$ mais pas par $4$}.
\end{array}\right.
\end{equation*} 
Cette formule est aussi valable dans le cas où $\a$ est un carré dans $\F^\times$.
\item 
Observons que $2n/[E:F]=mc$, et que $m$ et $c$ sont liés à $r$ et $d$
par les~for\-mu\-les \eqref{egalizink}.
Si l'on pose $k=r/m$,
la formule \eqref{bizuth} prend la forme plus symétrique~:
\begin{equation*}
\label{MadeleineForestierSym}
\bo_{E/E_0}(\a)^{mc} = (-1)^{mk} 
\end{equation*}
où $c$ et $k$ sont premiers entre eux.
\end{enumerate}
\end{rema}

\subsection{}
\label{Biotabstait}

Dans les paragraphes \ref{Biotabstait} et \ref{cestlong},
$E$~est une extension finie de $F$
et $B$ est une $E$-algèbre~cen\-trale simple munie~d'une
$F$-involution $*$.
Le~cen\-tre $E$ de $B$ est stable par $*$.
Notons $E_0$
le sous-corps des~élé\-ments de $E$ in\-variants par $*$.
On suppose que $E/E_0$ est quadratique, 
\ie que $*$ n'est pas triviale sur~$E$.

\begin{lemm}
\label{bilbon1}
La sous-algèbre $\B_0$ des éléments de $B$ invariants par $*$ 
est une $E_0$-algèbre cen\-trale simple,~et~le 
morphisme naturel de $E$-algèbres
de $\B_{0}\otimes_{E_0}E$ dans $B$
est un isomorphisme. 
\end{lemm}

\begin{proof}
D'après \cite[Lemma 2.2.2]{GilleSzamuely},
il suffit de prouver que $\B_{0}\otimes_{E_0}E$ est isomorphe à $B$. 
Fixons un $\b\in\E$ tel que $\b\notin\E_0$ et $\b^2\in\E_0$.
Tout $x\in\B$ s'écrit $x'+\b x''$ avec~:
\begin{equation*}
x' = \frac{x+x^*}{2}\in\B_{0},
\quad
x'' = \frac{x-x^*}{2\b}\in\B_{0}.
\end{equation*}
Le morphisme naturel de $E$-algèbres de $\B_{0}\otimes_{E_0}E$ dans $B$
est donc surjectif.
Ces deux $E$-algèbres ayant la même dimension,
c'est un isomorphisme.
\end{proof}

De par l'isomorphisme naturel entre $\B_{0}\otimes_{E_0}\E$ et $B$,
l'action de $*$ sur $B$ s'identifie à l'action du géné\-ra\-teur
de $\Gal(E/E_0)$ sur $\B_{0}\otimes_{E_0}\E$.
Par abus de notation, nous noterons $*$ le géné\-ra\-teur
de $\Gal(E/E_0)$.

Prouvons d'abord le résultat suivant,
qui sera utile au paragraphe \ref{simples}. 
Notons $\Nrd_{B/E}$~la~nor\-me réduite de $B$ sur $E$.

\begin{lemm}
\label{NRDBE}
Pour tout $x\in\B$, on a $\Nrd_{B/E}(x^*)=\Nrd_{B/E}(x)^*$.
\end{lemm}

\begin{proof}
Si $P\in\E[X]$, on notera $P^*$ le polynôme obtenu en appliquant $*$
à tous~ses coefficients.
Nous allons prouver l'égalité 
${\rm Pcrd}_{B/E}(x^*)={\rm Pcrd}_{B/E}(x)^*$
entre poly\-nômes~carac\-téristique réduits,
dont nous déduirons l'égalité voulue en comparant les termes~cons\-tants. 

Selon \cite[17.2 Proposition 2]{BkiALG8}, il suffit de prouver que $x$ et 
$x^*$, considérés comme endo\-mor\-phismes (par multiplication à gauche)
de $B$ considéré comme $E$-espace vectoriel,
ont des~poly\-nômes caractéristi\-ques
conjugués l'un de l'autre par $*$.
Il suffit pour cela
de montrer que,~pour tout endo\-mor\-phisme $u\in\End_E(B)$,
les po\-lynômes caractéristi\-ques de $u$ et de
$u^*:x\mapsto u(x^*)^*$ sont conju\-gués~l'un de l'autre par $*$.
Pour le prouver, 
fixons une base de $B$ sur $E$ qui soit~une base~de $B_0$ sur $E_0$.
On voit alors que, pour $\l\in\E$,
la matrice de $\l\cdot{\rm id}_B-u^*$
est~conju\-guée~par $*$ à celle de 
$\l^*\cdot{\rm id}_B- u$,
ce qui donne le résultat voulu. 
\end{proof}

Fixons maintenant une $E_0$-algèbre à~di\-vi\-sion cen\-trale $C_0$
et un entier $m_0$ tels que~:
\begin{equation*}
B_0 \simeq \Mat_{m_0}(C_0).
\end{equation*}
On en déduit un isomorphisme de $E$-algèbres
$\B \simeq\Mat_{m_0}(C_0\otimes_{E_0} E)$
à travers lequel 
l'action de $*$ sur $B$~dé\-ri\-ve~naturellement de celle de $*$
sur $C_0\otimes_{E_0} E$.
Notons $c_0$ le de\-gré réduit de $C_0$ sur~$E_0$.
Le lemme suivant sera essentiel dans la preuve du théorème 
\ref{THMEXISTENCECARACSIMPLE}. 

\begin{lemm}
\label{bilbon3}
Pour qu'il existe un ordre maximal stable par $*$ dans $\B$,
il faut et suffit que~$c_0$ soit impair, ou que $c_0$ soit pair et que $E$
soit ramifié sur $E_0$. 
\end{lemm}

\begin{proof}
Pour prouver le lemme, 
nous allons décrire la $E$-algèbre $C_0\otimes_{E_0}E$ munie de
l'involution $*$,
selon la parité de $c_0$.

Si $c_0$ est im\-pair, alors $C=C_0\otimes_{E_0}E$ est une $E$-algèbre 
à division centrale de degré réduit~$c_0$.
Elle a un unique ordre maximal $\Oo_C$, qui est donc~sta\-ble par $*$.
On en déduit que $\Mat_{m_0}(\Oo_C)$ est un ordre maximal de
$B\simeq\Mat_{m_0}(C)$ stable par $*$.

On suppose maintenant que $c_0$ est pair.
Dans ce cas,
$C_0\otimes_{E_0}E$ est isomorphe à l'algèbre~des matrices de
taille $2$ à coefficients dans une $E$-algèbre à division centrale $C$
de~de\-gré réduit~$c$~égal à $c_0/2$.
Précisons cette description. 
Fixons pour cela une
{uniformisante $\w_{E_0}$ de $E_0$,
ainsi qu'une ex\-ten\-sion non~ra\-mi\-fiée $L_0$ de $E_0$ dans $C_0$
de de\-gré~$c_0$ et une uni\-for\-mi\-sante $\w_0$ de $C_0$ normalisant
$L_0$ telle que $\w_0^{c_0}$ soit égale à $\w_{E_0}$
(voir \cite[\S 9.13]{Jacobson}).}
Ainsi $\w_0$ définit,
par conjugaison sur $L_0$,~un générateur de~$\Gal(L_0/E_0)$.

Si $E/E_0$ est non ramifiée,
on plonge $E$ dans $L_0$
et on note $C$ le commutant de $E$ dans $C_0$.~On observera que
$\w_0^2$ agit trivialement par conjugaison sur $E$ et
est une uniformisante de $C$. 

Si $E/E_0$ est ramifiée,
{il y a un choix de $\w_{E_0}$ tel que $E$ soit engendrée sur 
$E_0$ par une racine carrée de $\w_{E_0}$.
On peut donc identifier $E$ à l'extension de $E_0$ dans $C_0$
engendrée par $\w_0^{c}$,
dont~le carré est $\w_{E_0}$.
À travers cette identification,
$\w_E=\w_0^{c}$ est une uniformisante de $E$.}
On note $C$ le commutant de~$E$ dans $C_0$,
et on fixe une racine de~l'uni\-té $\xi\in\L_0$ d'ordre premier à $p$
telle que $\w_E\xi=-\xi\w_E$.
On observera que $\w_0$ appartient à $C$ et en est une uniformisante.

Posons $\Omega=\w_0$ si $E/E_0$ est non ramifiée,
et $\Omega=\xi$ si $E/E_0$ est ramifiée.
Dans tous les cas, $\Omega$ dé\-fi\-nit,
par conjugaison sur $E$,
le générateur de $\Gal(E/E_0)$,
il normalise $C$ et on a $\Omega^2\in\C$.

Con\-si\-dé\-rons maintenant
$C_0$ comme le $C$-espace vectoriel à droite $C+\Omega C$.
L'application~:
\begin{equation}
\label{plongemap}
x + \Omega y
\mapsto 
\begin{pmatrix}
x & \Omega y\Omega\\
y & \Omega^{-1}x\Omega
\end{pmatrix},
\quad
x,y\in\C,
\end{equation}
est un plongement de $E_0$-algèbres de $C_0$ dans $\Mat_2(C)$,
induisant un isomorphisme de $E$-algèbres, noté $\iota$, 
entre $C_0\otimes_{E_0}E$ et $\Mat_2(C)$.
On en déduit un isomorphisme
entre $B$ et la $E$-algèbre $\Mat_{m}(C)$ pour $m=2m_0$.
{Prouvons que l'action de $*$ sur $C_0\otimes_{E_0}E$
se transporte à $\Mat_2(C)$ en l'involution~:
\begin{equation}
\label{actiondestar}
\begin{pmatrix} x & z \\ y & u \end{pmatrix}
\mapsto
\begin{pmatrix}
\Omega {u} \Omega^{-1} & \Omega {y}\Omega \\
\Omega^{-1} {z} \Omega^{-1} & \Omega^{-1} {x} \Omega 
\end{pmatrix} 
\quad
x,y,z,u\in\C.
\end{equation}
Pour cela,
on fixe,
comme dans la preuve du lemme \ref{bilbon1},
un $\b\in E$ tel que $\b\notin\E_0$ et $\b^2\in\E_0$~et,
pour tout $g\in\Mat_2(C)$ qu'on écrit sous la forme du membre de
gauche de \eqref{actiondestar},
on vérifie par un calcul explicite à partir de \eqref{plongemap}
que l'antécédent de $g$ dans $\C_0\otimes_{E_0}E$ est l'élément~:
\begin{equation*}
(x+\Omega u\Omega^{-1}+\Omega y+z\Omega^{-1})\otimes \frac 1 2
+(x-\Omega u\Omega^{-1}+\Omega y-z\Omega^{-1})\b^{-1}\otimes \frac {\b} 2
\end{equation*} 
dont le conjugué par $*$ est égal à~: 
\begin{equation*}
(\Omega u\Omega^{-1}+x+z\Omega^{-1}+\Omega y)\otimes \frac 1 2
+(\Omega u\Omega^{-1}-x+z\Omega^{-1}-\Omega y)\b^{-1}\otimes \frac {\b} 2 
\end{equation*} 
car $\b^*$ est égal à $-\b$.}

Prouvons maintenant le lemme dans le cas où $c_0$ est pair. 
Si $E/E_0$ est ramifiée,
l'ordre maxi\-mal $\mathfrak{c}$ de $C_0\otimes_{E_0}E$ corres\-pon\-dant
à $\Mat_2(\Oo_C)$
est stable~par $\Gal(E/E_0)$ car $\Omega=\xi$ normalise~$\Oo_C$.
On en déduit que $\Mat_{m_0}(\mathfrak{c})$ est un ordre maximal de
$B\simeq\Mat_{m_0}(C_0\otimes_{E_0}E)\simeq\Mat_{m}(C)$ stable~par~$*$.
Supposons~main\-te\-nant que $E/E_0$ soit non ramifiée
et qu'il y ait dans $B$ un ordre maximal~$\bb$ sta\-ble par $*$.
Dans ce cas,
$\Omega$ est une uniformisante de $C_0$ normalisant $C$ et,
selon \cite[Lemme~4.4]{BTSMF1},
l'application
$\bb_0\mapsto\bb_0\otimes\Oo_E$ est~une bijection entre ordres de $B_0$
et ordres de $B$ stables par $*$,~le
produit~ten\-soriel étant pris sur $\Oo_{E_0}$.
Il existe donc~un unique ordre $\bb_0$ de $B_0$ correspondant~à~$\bb$
par cette~bi\-jection,~et $\bb_0$ est un ordre maximal de~$B_0$.
Comme l'action de $B_0^\times$ par conjugaison sur~$B$
com\-mu\-te à $*$,
on peut même supposer,
quitte~à~con\-ju\-guer $\bb$,
que $\bb_0$ est égal à $\Mat_{m_0}(\Oo_{C_0})$.
Par~conséquent,
$\Oo_{C_0}\otimes\Oo_{E}$ est ma\-ximal dans 
$\C_0\otimes_{E_0}\E$
et stable par $*$. 
Mais (toujours d'après \cite[Lemme~4.4]{BTSMF1})
$\C_0\otimes_{E_0}\E$ admet un unique ordre stable par $\Gal(E/E_0)$.
{Appliquant la formule \eqref{actiondestar},
et compte tenu de ce que $\Oo_C$ et $\p_C$ sont stables
par conjugaison par $\Omega$
et $\Omega^2$ est une~uni\-formisante de $C$,
on vérifie que l'ordre~:
\begin{equation*}
\begin{pmatrix}
\Oo_C & \p_C \\ \Oo_C & \Oo_C
\end{pmatrix}
\end{equation*}
est stable par $\Gal(E/E_0)$.
Celui-ci est minimal et non pas maximal,
ce qui nous donne la~con\-tradiction voulue.}
\end{proof}

\subsection{}
\label{cestlong}

Nous sommes toujours dans la situation du paragraphe \ref{Biotabstait}.
Supposons en outre que $B$~ad\-met\-te un ordre maximal stable par $*$.
Définissons $c_0$ comme au paragraphe \ref{Biotabstait}.
D'après le lemme \ref{bilbon3},
l'extension $E/E_0$ est ramifiée~si $c_0$ est pair.

\begin{lemm}
\label{PoneyFringant}
Soient $C$ une $E$-algèbre à division réduite et $m$ un entier tel que
$\B$~soit~iso\-morphe à $\Mat_m(C)$.
Notons $\ee$ le corps résiduel de $C$ et $c$ son degré réduit sur $E$.
\begin{enumerate}
\item 
Il existe un ordre maximal $\bb$ stable par $*$ 
et un isomor\-phisme de~$E$-al\-gè\-bres $\B\simeq\Mat_m(C)$ tels que
l'image de $\bb$ soit égale à l'ordre maximal standard $\Mat_m(\Oo_C)$
et l'action induite
par~$*$ sur le groupe $\bb^\times/\BU^1(\bb)\simeq\GL_m(\ee)$ soit~:
\begin{enumerate}
\item
l'action de l'automorphisme d'ordre $2$ de $\ee$
si $E/E_0$ est non ramifiée,
\item
l'action par conjugaison d'un $\v\in\GL_m(\ee)$ 
tel que $\v^2\in\ee^{\times}$ et $\v^2\notin\ee^{\times2}$ si
$E/E_0$~est~ra\-mi\-fiée et $c_0$ est pair,
\item
triviale si $E/E_0$ est ramifiée et $c_0$ est impair.
\end{enumerate}
\item
Le normalisateur de $\bb$ dans $B^\times\simeq\GL_m(C)$
est engendré par $\bb^\times$ 
et une uniformisante $\w$ de $C$
telle que $\w^c$ soit une uniformisante de $E$ et~:
\begin{equation*}
\label{choixweqB}
\w^* = \left\{ 
\begin{array}{rl}
-\w & \text{si $E/E_0$ est ramifiée et $c_0$ est impair}, \\ 
\w & \text{sinon.}
\end{array}\right.
\end{equation*}
\end{enumerate}
\end{lemm}

\begin{proof}
Comme au paragraphe précédent, 
on identifie $B$ à $\Mat_{m_0}(C_0\otimes_{E_0}E)$.
Il suffit de prouver le lemme lorsque $m_0=1$,
ce que l'on suppose par la suite.
On a donc $B=C_0\otimes_{E_0}E$.

Si $c_0$ est impair, alors $m=1$ et $c=c_0$,
et $B$ est isomorphe à $\C$.
L'unique ordre maximal~de~$B$ a~les~pro\-priétés requises {en (1)}.
En outre, lorsque $E/E_0$ est non ramifiée,
on choisit pour $\w$ une uni\-formisan\-te~$\w_0$ de $C_0$
telle que $\w_0^{c_0}$ soit une~uni\-formisan\-te~de $E_0$ et,
lorsque $E/E_0$ est ramifiée,
on choi\-sit~une uniformisante $\w_0$ de $C_0$
et une uniformisante $\w_E$ de $E$
telles que $\w_0^{c_0}=\w_E^2$ soit une~unifor\-mi\-sante
de $E_0$,
et on pose $\w=\w_E^{}\w_0^{(1-c_0)/2}$,
{qui vérifie $\w^2=\w_0$.} 
{Dans les deux cas, un tel $\w$ vérifie les conditions de (2).}

Si $c_0$ est pair,
auquel cas $E/E_0$ est ramifiée,
alors $m=2$ et $c_0=2c$,
et on fixe comme~au~pa\-ra\-graphe
\ref{Biotabstait}~un~iso\-morphis\-me entre
$B$ et la $E$-algèbre $\Mat_2(C)$,
sur laquelle $*$ agit par \eqref{actiondestar}
avec $\Omega=\xi\in\L_0$ et $\Omega^2\in\L_1=\L_0\cap\C$.
On identifie,
grâce à l'inclusion de $L_0$ dans $\C_0$, 
le corps~ré\-si\-duel~de $L_0$ au corps~ré\-siduel $\ee_0$
de $C_0$.
De même,
on identifie le corps résiduel de $L_1$ à~$\ee$.
Aussi $\ee_0$~est-il une extension quadratique de $\ee$.
L'automorphisme~de $\ee$ induit par conjugaison par $\Omega$,
qui appartient à $L_0$, est donc trivial.~On
en déduit que
l'action~induite~par~$*$ sur 
$\bb/\p_{\bb}\simeq\Mat_2(\ee)$ est $\ee$-li\-né\-aire.
D'après le théorè\-me~de Skolem-Noether, c'est~donc l'ac\-tion par conjugai\-son
d'un élément $\v\in\GL_2(\ee)$ tel que $\v^2\in\ee^{\times}$.~Mais,
{selon \eqref{actiondestar},}
les points fixes~de $*$ dans $\Mat_2(\ee)$ sont les~:
\begin{equation}
\label{ptfix}
\begin{pmatrix}
x & \e y \\
y & x
\end{pmatrix},
\quad
x,y\in\ee,
\end{equation}
où $\e$ désigne l'image de $\Omega^2$ dans $\ee$.
Montrons que $\e$ est un générateur de $\ee^\times$.
En effet,
si $\zeta$~est un générateur de $\bm_{L_0}$,
le fait que $\w_E\xi=-\xi\w_E$ entraîne que
$\xi\in\zeta^{(Q+1)/2}\bm_{L_1}$,
où $Q$ est le cardinal de $\ee$,
donc que $\e$ appartient à $\N_{\ee_0/\ee}(\zeta)\ee^{\times2}$.
On en déduit que les points $*$-fixes \eqref{ptfix}
forment une extension quadrati\-que de $\ee$,
donc que $\v^2\notin\ee^{\times2}$.
{On a donc (1) et,
pour obtenir (2), 
on choisit pour $\w$~une~uni\-for\-mi\-sante de $C_0$
telle que $\w_0^{c}$ soit une~unifor\-mi\-sante de $E$.}
\end{proof}

\subsection{} 
\label{motifs58}

Nous allons 
reprendre la~dis\-cussion du paragraphe \ref{tetrobot1}
à la lumière du travail~ef\-fectué~dans les paragraphes précédent. 
Commençons par généraliser 
les résultats du paragraphe \ref{invTTmoins}.

\begin{lemm}
\label{connemaraTT0}
Soit $\TT$ une endoclasse autoduale non nulle de degré divisant $2n$. 
Soit $\t$ un~ca\-ractère simple maximal de $G$ d'endo-classe $\TT$,
et soit $u\in\G$ tel que $\t^{u}=\t^{-1}$.
Suppo\-sons qu'il existe une strate simple maximale
$[\aa,\b]$ de $A$ telle que $\t\in\Cc(\aa,\b)$ et $\b^{u}=-\b$.  
Posons $E=F[\b]$ et $E_0=F[\b^2]$,
no\-tons $T$ la sous-extension modérément ramifiée ma\-xi\-male~de $\E$
sur $F$ et posons $T_0=T\cap \E_0$.  
Alors la classe de $F$-isomorphisme de $T/T_0$ dépend~uni\-que\-ment de $\TT$.  
\end{lemm}

\begin{proof}
En d'autres termes, nous allons prouver que 
la classe de $F$-isomorphisme de $T/T_0$ 
ne dépend du choix ni de $G$, ni de $\t$, ni de $[\aa,\b]$, ni de $u$.
Posons~:  
\begin{equation*}
\s = \begin{pmatrix} -{\rm id}&0\\0&{\rm id}\end{pmatrix} 
\in \Mat_{n}(F)\times \Mat_{n}(F) \subseteq \Mat_{2n}(F)
\end{equation*}
où ${\rm id}$ est l'identité de $\Mat_{n}(F)$.
Soit $[\aa_1,\g]$ une strate simple maximale~$\s$-au\-to\-duale 
de $\Mat_{2n}(F)$ telle que $\Cc(\aa_1,\g)$ contienne un
caractè\-re~simple~$\s$-au\-todual $\t_1$
d'endo-classe $\TT$,
dont l'existence est assurée par la proposition \ref{redformadap}.
Posons $P=F[\g]$ et $P_0=F[\g^2]$.
Notons $L$ la sous-extension mo\-dérément ramifiée ma\-xi\-ma\-le de $P$ 
sur $F$, et posons $L_0=L\cap P_0$.
D'après la remarque \ref{invTTrem}, 
la classe de $F$-isomorphisme de $L/L_0$ dépend~uni\-quement de $\TT$.
Nous allons prou\-ver que 
$L/L_0$ et $T/T_0$ sont $F$-isomorphes. 

D'après le lemme \ref{gandalfcarre},
on peut plonger $E$ dans $\Mat_{2n}(F)$ de façon que la
conjugaison par $\s$~in\-duise l'automorphisme non~trivial de $E/E_0$.
Le centralisateur $B'$ de $E$ dans $\Mat_{2n}(F)$ est une~$E$-algèbre
centrale simple munie d'une $F$-involution $\s$,
entrant dans le cadre du paragraphe~\ref{Biotabstait}.~La
$E_0$-algèbre $B'^\s$ est isomorphe à une algèbre de matrices à
coefficients dans $E_0$. 
Selon le lemme \ref{bilbon3},
il y a donc dans $B'$ un ordre maximal $\bb'$ stable par $\s$. 
Notons $\aa'$ l'unique ordre nor\-malisé par $\E^\times$ 
tel que $\aa'\cap B'=\bb'$.
Par unicité, il est stable par $\s$, donc $[\aa',\b]$ est 
une strate simple maximale $\s$-autoduale de $\Mat_{2n}(F)$.
D'après le lemme \ref{lemmedetransfertautodual},
le transfert $\t'\in\Cc(\aa',\b)$ de $\t$
est un caractère simple maximal $\s$-autodual d'endo-classe $\TT$.
D'après la remarque \ref{invTTrem}, 
les extensions $L/L_0$ et $T/T_0$ sont $F$-isomorphes.
\end{proof}

Le degré paramé\-trique d'une~représen\-ta\-tion cuspi\-dale~de $\G$
est défini par \eqref{pardeg}.

\begin{lemm}
\label{fromagequipue}
\label{fromagequipuerec}
Soit $\TT$ une endo-classe autoduale de degré divisant $2n$
et soit un entier $N\>1$ divisant $2n/\deg(\TT)$. 
On pose~:
\begin{equation}
\label{bananerotie}
c=\frac d {(d,\deg(\TT))},
\quad
m=\frac {2n}{c\cdot\deg(\TT)},
\end{equation} 
et $\d=N\cdot\deg(\TT)$.
Si $\TT$~est non nulle,
on note $T/T_0$ l'extension~qua\-dratique 
qui lui est associée au lemme \ref{connemaraTT0}. 
Pour qu'il y ait une~re\-pré\-sentation cuspidale autoduale de $G$ 
d'en\-do-classe $\TT$~et de~de\-gré~para\-métrique $\d$,
il faut et suffit~:
\begin{enumerate}
\item 
d'une part que $m=N/(c,N)$,
\item
d'autre part que $N$ vérifie les conditions~sui\-van\-tes~:
\begin{enumerate}
\item
si $\TT$ est non nulle et $T/T_0$ est non ramifiée,
alors $N$ est impair,
\item
si $\TT$ est non nulle et $T/T_0$ est ramifiée, 
alors $N$ est pair ou égal à $1$,
\item
si $\TT$ est nulle,
alors $N$ est pair ou égal à $1$.
\end{enumerate}
\end{enumerate}
\end{lemm}

\begin{proof}
Observons tout d'abord que~:
\begin{equation*}
\label{truffe}
\frac {\d} {(d,\d)} = \frac {N} {(c,N)} \cdot \frac {\deg(\TT)} {(d,\deg(\TT))}.
\end{equation*}
Soit $\pi$ une représentation cuspidale de $G$
d'en\-do-classe $\TT$ et de~de\-gré~para\-métrique $\d$.
De \eqref{rsprime} on tire~:
\begin{equation*}
\frac {N} {(c,N)} \cdot c \cdot \deg(\TT)
= rd = m \cdot c \cdot \deg(\TT),
\end{equation*}
ce qui donne $m=N/(c,N)$.
Comme au~para\-graphe \ref{regimbard},
le transfert de Jacquet-Lang\-lands de $\pi$ à $\GL_{2n}(F)$
est de la forme $L(\pi'_{0},s)$,
où~$s$ est égal à $2n/\d$ selon \eqref{scoprimem}
et $\pi'_0$ est une~repré\-senta\-tion cuspidale de $\GL_{\d}(F)$.
Celle-ci est auto\-duale~si et seu\-le\-ment si $\pi$ l'est. 
On a le ré\-sultat~sui\-vant,~dû à \cite{SZ1,BHJL3,SeStLinked} et 
\cite[Theo\-rem~4.8]{Dotto}.  

\begin{theo}
\label{conjecjl}
Les~re\-présen\-tations $\pi$ et $\pi'_{0}$~ont la même endo-classe.
\end{theo}

Supposons maintenant que $\pi$ soit autoduale.
Alors $\pi'_{0}$ est elle-même autoduale,
et la condition (2) sur $N$ est une conséquen\-ce~des lemmes
\ref{lemmedeparite} et \ref{lemmedepariteniveauzero} 
appliqués à $\pi'_0$. 

Inversement, 
supposons vérifiées les conditions du lemme.
En vertu des lemmes
\ref{lemmedeparite}, \ref{lemmedepariteniveauzero},~il
y a une représentation cuspi\-dale~autoduale $\pi'_{0}$
de $\GL_{\d}(F)$ d'endo-classe~$\TT$.
Posons $s=d/(d,\d)$ et
formons la représentation essentielle\-ment de carré 
intégrable $L(\pi'_{0},s)$ de $\GL_{2n}(F)$.
La condition (1) sur $N$
entraîne~que $r=\d/(d,\d)$,
qui est premier à $s$. 
D'après la remarque \ref{marecage},
le transfert de Jacquet-Lang\-lands~de $L(\pi'_{0},s)$
à $G$ est une repré\-sentation cuspidale $\pi$
de degré paramétrique $\d$, 
qui~est au\-to\-duale car $\pi'_{0}$ l'est,
et son endo-classe est égale à $\TT$ d'après le théorème \ref{conjecjl}. 
\end{proof}

Terminons par le lemme suivant,
qui nous sera utile par la suite. 

\begin{lemm}
\label{connemara}
Soit $E_0$ une extension finie de $F$ de degré $g$ divisant $n$,
et soit $E$ une extension qua\-dra\-tique de $E_0$. 
No\-tons $T$ la sous-extension modérément ramifiée ma\-xi\-male~de $\E$
sur $F$ et posons $T_0=T\cap \E_0$.
\begin{enumerate}
\item
L'extension $E/E_0$ est ramifiée si et seulement si $T/T_0$ est ramifiée,
et on a~:
\begin{eqnarray}
\label{whimsical1}
\N_{E/E_0}(\E^\times)\cap\F^\times &=& \N_{T/T_0}(\T^\times)\cap\F^\times, \\
\label{whimsical2}
\bo_{E/E_0}(\a)^{2n/[E:F]} &=& \bo_{T/T_0}(\a)^{2n/[T:F]}.
\end{eqnarray} 
\item
Supposons qu'il existe~un plon\-ge\-ment de $F$-algèbres de $E$ dans $A$ tel
que $\tau$ induise sur $\E$ l'automorphisme non trivial de $E/E_0$.
Alors~:
\begin{enumerate}
\item
le centralisateur $B$ de $E$ dans $A$ est une $E$-algèbre centrale simple
munie de la $F$-in\-vo\-lu\-tion $\tau$, 
\item
et $B_0=B^\tau$ est une $E_0$-algèbre centrale simple isomorphe à une algèbre 
de matrices~à coefficients dans une $E_0$-algèbre à division centrale de degré
réduit égal à~: 
\begin{equation}
\label{marquisat}
c_0 = \left\{ 
\begin{array}{rl}
d/(d,g) & \text{si $\a\in\N_{T/T_0}(\T^\times)$
ou si $4$ divise $d/(d,g)$}, \\
2d/(d,g) & \text{si $\a\notin\N_{T/T_0}(\T^\times)$ et $d/(d,g)$ est impair}, \\
d/2(d,g) & \text{si $\a\notin\N_{T/T_0}(\T^\times)$
et $d/(d,g)$ est divisible par $2$ mais pas par $4$}.
\end{array}\right.
\end{equation} 
\end{enumerate}
\end{enumerate} 
\end{lemm}

\begin{proof} 
La première assertion provient 
du fait que $E$ est totalement sau\-va\-gement~ra\-mi\-fiée sur $T$
et que $p\neq2$,
ce qui entraîne en particulier que $[E:T]$ est impair.
Pour~obte\-nir~la seconde assertion,
on applique la remarque \ref{confino} et \eqref{whimsical1}.
\end{proof}

\subsection{}
\label{degueulis}
\label{vitevite}

Nous en arrivons au résultat principal de cette section,
qui donne une condition nécessai\-re et suffisante d'existence d'un
caractère simple $\tau$-autodual dans une représentation cuspidale
auto\-duale de $G$.
Nous traitons ici le cas des représentations de niveau non nul. 
Le cas des~re\-pré\-sentations de niveau $0$ sera traité à la section
\ref{theriverofnoreturn}.  

On fixe dans tout le reste de cette section une
endo-classe autoduale non nulle $\TT$ de degré~di\-visant $2n$.
On note $T/T_0$ l'extension~qua\-dra\-tique qui lui est associée,
on définit des entiers $m$ et $c$ par \eqref{bananerotie}
et un entier $c_0$ par \eqref{marquisat} avec $g=\deg(\TT)/2$. 

\begin{theo}
\label{THMEXISTENCECARACSIMPLE}
\label{tomwolfe}
\begin{enumerate}
\item
Pour qu'il existe un caractère simple maximal $\tau$-autodual 
d'endo-classe $\TT$ dans $G$,
il faut et suffit que~:
\begin{enumerate}
\item 
d'une part~: 
\begin{equation}
\label{bizuthrecap1} 
\bo_{\T/\T_0}(\a)^{2n/[T:F]} = (-1)^{r},
\end{equation}
\item
d'autre part que,
si $T/T_0$ est non ramifiée,
l'entier $c_0$ soit impair. 
\end{enumerate}
\item 
Soit $\pi$ une représentation cuspidale autoduale de $G$
d'endo-classe $\TT$.
Pour que $\pi$ contienne un caractère simple $\tau$-autodual,
il faut et suffit que l'égalité \eqref{bizuthrecap1} soit vérifiée.  
\end{enumerate}
\end{theo}

\begin{proof}
Soit d'abord $\t$ un ca\-ractère simple maximal $\tau$-autodual 
d'endo-classe $\TT$ de $G$.
Soit $[\aa,\b]$ une~strate simple 
$\tau$-au\-to\-duale de $\A$ telle que $\t\in\Cc(\aa,\b)$,
dont l'existence~est~as\-su\-rée par~le lemme \ref{gandalf}.
On pose $E=F[\b]$, $E_0=F[\b^2]$.
Le lemme \ref{gandalfcor1}(2.a) et la proposition~\ref{gandalf0coro}
entraînent \eqref{bizuth}.
On déduit du lemme \ref{connemara}(1) d'une part l'égalité \eqref{bizuthrecap1},
d'autre part que $E/E_0$ est ra\-mi\-fiée si et seulement si $T/T_0$ l'est. 
Soit $B$ le~cen\-tra\-lisateur de~$E$ dans $A$.
Appliquant~le~lem\-me \ref{connemara}(2), la 
$E_0$-algèbre $B_0=B^\tau$ est 
isomorphe à une algèbre de matrices~à~coeffi\-cients dans une
$E_0$-algèbre à division de degré~ré\-duit $c_0$. 
Comme $\aa\cap\B$ est un ordre maximal de $B$ stable par~$\tau$, 
il suit du lemme \ref{bilbon3} que, 
si $E/E_0$
(ou de façon équivalente $T/T_0$)
est non ramifiée, $c_0$ est impair.  

Inversement,
suppo\-sons~que~\eqref{bizuthrecap1} soit vérifiée.
Soit $[\aa_1,\b]$ une strate simple maximale $\s$-au\-to\-duale de $\Mat_{2n}(F)$ 
telle que $\Cc(\aa_1,\b)$ contienne un caractère~sim\-ple $\s$-autodual $\t_1$
d'endo-classe $\TT$,
dont l'existence est assurée par la pro\-position \ref{redformadap}.
Posons $E=\F[\b]$ et $\E_0=\F[\b^2]$.
D'après le lemme \ref{connemaraTT0},
on peut identifier $T$~à~la sous-extension modérément ramifiée maximale de
$E$~sur $F$ et $T_0$ à $T\cap E_0$.
Compte tenu de \eqref{whimsical2} et de la proposition \ref{gandalf0coro},
il existe un plon\-ge\-ment~de $F$-al\-gè\-bres de $E$ dans $A$ tel que 
$\tau$ induise sur $E$ l'auto\-mor\-phis\-me~non trivial de $E/E_0$.
Fixons~un tel plon\-ge\-ment.
Alors~:
\begin{itemize}
\item 
d'une part,
le centralisateur $B$ de $E$ dans $A$ est isomorphe à une algèbre
$\Mat_m(C)$ de matri\-ces~à~coef\-ficients dans une $E$-algèbre à division
centrale $C$ de degré réduit $c$,
\item 
et d'autre part,
la $E_0$-algèbre $B_0=B^\tau$ est isomorphe à une algèbre de matrices
$\Mat_{m_0}(C_0)$~à coeffi\-cients dans une
$E_0$-algèbre à division centrale $C_0$ de degré réduit $c_0$. 
\end{itemize}

Supposons maintenant en outre que $c_0$ soit impair 
si $T/T_0$ est non ramifiée.
Selon le lemme \ref{bilbon3}, 
il existe dans $B$ un ordre maximal $\bb$ stable par $\tau$.
Soit~$\aa$ l'unique ordre de $A$ normalisé par $E^\times$ 
tel que $\aa\cap\B=\bb$ (lemme \ref{patelmax}). 
Par  unicité, $\aa$~est stable par $\tau$.
La strate simple maximale $[\aa,\b]$ est donc $\tau$-autoduale.  
Le transfert de $\t_1$ à $\Cc(\aa,\b)$ est $\tau$-autodual
d'après le lemme \ref{lemmedetransfertautodual}.
Ceci prouve la première partie du théorème. 

Prouvons la seconde partie du théorème. 
Soit $\pi$ une représentation cuspidale autoduale de $G$
d'endo-classe $\TT$.
Pour prouver que $\pi$ contient un caractère simple $\tau$-autodual,
il suffit, d'après le corollaire \ref{entrenouscor}, 
de~prou\-ver l'existence d'un caractère simple maximal $\tau$-autodual 
de $G$ d'endo-classe $\TT$.
D'après la première partie du théorème,
il suffit donc de prouver que $c_0$ est impair~si $T/T_0$ est non ramifiée.
Supposons donc que $T/T_0$ soit non ramifiée,
et notons $\d$ le degré pa\-ra\-mé\-trique de $\pi$.
D'après le lemme \ref{fromagequipue},
l'entier $N=\d/\deg(\TT)$ est impair.
Comme c'est un multiple de $m$ d'après \eqref{pardeg}, 
on en déduit que $m$ est impair.
Il suit du paragraphe \ref{cestlong}
que $c_0$ est impair (et que $c=c_0$ l'est aussi),
sans quoi on aurait $c_0=2c$ et $m=2m_0$. 
\end{proof}

On tire de cette preuve le corollaire suivant,
qui précise le lemme \ref{fromagequipue}.

\begin{coro}
\label{fromagequipuecor}
Soit $\pi$ une~repré\-sentation cuspi\-dale autoduale de
$G$~d'en\-do-classe $\TT$. 
Supposons que $\pi$ contienne un caractère simple $\tau$-autodual.
\begin{enumerate}
\item
Si $T/T_0$ est non ramifiée,
alors $m$ et $c$ sont impairs.
\item
Si $T/T_0$ est ramifiée, 
alors $m$ est pair ou égal à $1$.
\item
Si $r$ est impair, alors $m$ et $c$ sont impairs.
\end{enumerate}
\end{coro}

\begin{proof}
Reprenons les notations de la preuve du théorème 
\ref{THMEXISTENCECARACSIMPLE}.
D'après \eqref{pardeg},
il~y~a un diviseur $b$ de $c$ tel que $N=mb$. 
L'assertion (1) a déjà été prouvée à la fin de la preuve~pré\-cé\-dente.
Si $T/T_0$ est~ra\-mi\-fiée,
le lemme \ref{fromagequipue} assure que $N$ est pair ou égal à $1$.
Si $c_0$ est pair,~alors $m=2m_0$ est pair.
Si $c_0$~est impair, alors $c=c_0$ aussi,
donc $b$ aussi,
donc $m$ est pair ou égal~à $1$.
Enfin,
si $r$ est impair, 
les formules \eqref{bizuthrecap1} et \eqref{whimsical2}
impliquent que $2n/[E:F]=mc$ est impair. 
\end{proof}

Voici une autre conséquence, simple mais importante,
du théorème \ref{THMEXISTENCECARACSIMPLE}.
Rappelons que, si $r$ est~pair,
on a fixé un $\s\in\G$ tel que $\s^2=1$
et de polynôme caractéristique réduit $(X^2-1)^n$.

\begin{coro}
\label{tolkienelrond}
\begin{enumerate}
\item 
Supposons que $r$ soit pair.
Toute représentation cuspidale autoduale
de~ni\-veau non nul
de $G$
contient un~ca\-rac\-tère simple $\s$-autodual.
\item
Supposons que $m$ soit pair.
Toute représentation cuspidale autoduale de~ni\-veau non nul~de $G$
contient un~ca\-rac\-tère simple $\tau$-autodual,
quel que soit $\a\in\F^\times$.
\end{enumerate}
\end{coro}

\begin{proof}
La première assertion est une conséquence immédiate du 
théorème \ref{THMEXISTENCECARACSIMPLE}
dans le cas où $r$ est pair et $\a$ est un carré dans $F^\times$.
{Pour ce qui est de la seconde assertion,
$m$ divisant à la fois $2n/[T:F]$ et $r$ (voir \eqref{egalizink}), 
on voit que l'égalité \eqref{bizuthrecap1} est vérifiée
lorsque $m$ est pair.}
\end{proof}

\begin{rema}
Voici un exemple de représentation cuspidale autoduale 
ne contenant pas de caractère simple $\tau$-autodual
et pour laquelle les conclusions du corollaire 
\ref{fromagequipuecor} ne sont pas~vé\-rifiées. 
Supposons que $T/T_0$ soit non ramifiée,
choisissons un $n\>1$ tel que $2n/\deg(\TT)$ soit pair 
et supposons que $r=1$.
D'après~le lemme \ref{fromagequipuerec}, 
il y a une représentation cuspidale autoduale $\pi$ de~$G$
d'en\-do-classe $\TT$ et de~degré
paramétrique $\deg(\TT)$.
D'après le théorème \ref{THMEXISTENCECARACSIMPLE},
la représen\-ta\-tion
$\pi$ ne contient pas de caractère simple $\tau$-au\-to\-dual,
quel que soit $\a$.
Enfin, $c=2n/\deg(\TT)$ est pair. 
\end{rema}

\begin{rema}
\label{carpasrep}
Voici un exemple de caractère simple $\tau$-autodual 
de $G$ qui n'est contenu dans aucune représentation cuspidale
autoduale de $G$.
Soit $D$ une $F$-algèbre à division centrale de degré réduit $2$
et soit $A=\Mat_2(D)$. 
Supposons que $T/T_0$ soit non ramifiée,
que $\a\notin\N_{T/T_0}(T^\times)$
et que $\TT$ soit de degré
$2$,~de sorte que $E=T$, $E_0=T_0=F$ et $m=2$.
L'identité \eqref{bizuthrecap1} est~véri\-fiée 
et l'entier~$c_0$ défini par \eqref{marquisat}
vaut $1$.
D'après le théorème \ref{THMEXISTENCECARACSIMPLE}, 
il existe un caractère~simple $\tau$-autodual de $G$.
Mais il ne peut pas exister de représentation cuspidale autoduale
de $G$ le~con\-te\-nant, sans quoi le corollaire \ref{fromagequipuecor}
serait contredit car $m=2$.
\end{rema}

\subsection{}
\label{indiceduntypepara}

On suppose dans ce paragraphe
qu'il y a un caractère simple maximal $\tau$-autodual de
$G$ dont l'endo-classe $\TT$ est celle fixée au~para\-gra\-phe \ref{vitevite}.
Les conditions (1.a) et (1.b) du théo\-rème 
\ref{THMEXISTENCECARACSIMPLE}~sont donc véri\-fiées.

Pour la définition de l'ensemble $\iso(\aa,\b)$
associé à une strate simple maximale $[\aa,\b]$
et du corps fini $\ee$,
on renvoie au paragraphe~\ref{saucisse}.

\begin{prop}
\label{indiceduntype}
\begin{enumerate}
\item 
Le nombre de classes de $\G^\tau$-conjugaison de 
caractères simples $\tau$-au\-to\-duaux 
contenus dans la classe de $G$-conjugaison d'un
caractère simple maximal $\tau$-autodual de $G$
d'endo-classe $\TT$ est~:
\begin{enumerate}
\item 
égal à $1$ si $T/T_0$ est non ramifiée,
ou si $T/T_0$ est ramifiée et $c_0$ est pair,
\item
égal à $\lfloor m/2\rfloor+1$ sinon.
\end{enumerate}
\item
Si $\t$ est un caractère simple maximal $\tau$-autodual d'endo-classe $\TT$ 
et si $[\aa,\b]$ est une~strate simple maximale $\tau$-autodua\-le telle que 
$\t\in\Cc(\aa,\b)$,
alors il existe un~isomor\-phis\-me de $F$-algè\-bres $\phi\in\iso(\aa,\b)$
identifiant $\BJ^0(\aa,\b)/\BJ^1(\aa,\b)$ à $\GL_m(\ee)$ et 
l'action~de $\tau$ sur ce groupe à~:
\begin{enumerate}
\item 
l'ac\-tion de l'automorphisme d'ordre $2$ de $\ee$
si $T/T_0$ est non ramifiée, 
\item
l'action par conjugaison d'un $\v\in\GL_m(\ee)$ 
tel que $\v^2\in\ee^{\times}$ et $\v^2\notin\ee^{\times2}$
si $T/T_0$ est~ra\-mi\-fiée et $c_0$ est pair,
\item
l'action par conjugaison de~:
\begin{equation}
\label{defsi42}
\s_i = {\rm diag}(-1,\dots,-1,1,\dots,1) \in \GL_m(\ee)
\end{equation}
où $-1$ apparait avec multiplicité $i$,
pour un unique entier $i\in\{0,\dots,\lfloor m/2\rfloor\}$,
si $T/T_0$ est~ra\-mi\-fiée et $c_0$ est impair.
\end{enumerate}
\item
Notant $B$ le centralisateur de $E=F[\b]$ dans $A$,
le normalisateur de $\t$ dans $G$ est engen\-dré~par
$\BJ^0(\aa,\b)$ et un élément $\w\in\B^\times$
tel que $\w^c$ soit une uniformi\-san\-te~de $E$ et~:
\begin{equation*}
\label{choixweqB}
\tau(\w) = \left\{ 
\begin{array}{rl}
-\w & \text{si $T/T_0$ est ramifiée et $c_0$ est impair}, \\ 
\w & \text{sinon.}
\end{array}\right.
\end{equation*}
\end{enumerate}
\end{prop}

\begin{defi}
\label{defindicetheta}
L'entier $i$ de {\rm (2.c)} est appelé l'\textit{indice} de $\t$.
Il ne dépend que de sa classe~de $G^\tau$-con\-ju\-gaison, 
et est déterminé par le fait que 
$(\BJ^0(\aa,\b)/\BJ^1(\aa,\b))^\tau\simeq\GL_i(\ee)\times\GL_{m-i}(\ee)$.
\end{defi}

\begin{proof}
Fixons un caractère simple $\tau$-autodual $\t\in\Cc(\aa,\b)$
d'endo-classe $\TT$.
D'après le lemme \ref{gandalf},
on peut supposer que la strate simple $[\aa,\b]$ est $\tau$-autoduale. 
Posons $E=F[\b]$~et notons $B$ son centralisateur dans $A$.
Munie de l'in\-vo\-lu\-tion $\tau$,
la $E$-algèbre $B$ entre dans le~cadre du paragraphe~\ref{Biotabstait}
et l'ordre maximal $\bb=\aa\cap\B$ est stable par~$\tau$. 
D'après le lemme \ref{PoneyFringant},
il y a une $E$-algèbre~à division centrale $C$ et
un isomor\-phisme de $F$-algè\-bres $\phi\in\iso(\aa,\b)$
induisant un isomor\-phisme~de $F$-al\-gè\-bres entre $B$ et 
$\Mat_m(C)$ tel que l'ordre maximal $\bb'$ de $B$ correspondant~à
$\Mat_m(\Oo_C)$ 
ait les~pro\-prié\-tés dé\-crites dans ce lemme.
Quitte à conjuguer $\t$~par~un élément de $B^\times$,
on peut~sup\-po\-ser~que $\bb'$ est égal à $\bb$.
En effet, si $y\in\B^\times$ est tel que $\bb'=\bb^y$,
le fait que $\bb$ et $\bb'$ soient tous deux stables par $\tau$
entraîne que $\tau(y)y^{-1}$ normalise $\bb$,
donc $\t$.
Il s'ensuit que $\t^y$~est~un caractère simple maximal $\tau$-autodual
d'endo-classe $\TT$. 
Posant $\BJ^0=\BJ^0(\aa,\b)$ et $\BJ^1=\BJ^1(\aa,\b)$,~le
quotient $\BJ^0/\BJ^1$ s'identifie donc
à $\GL_{m}(\ee)$ muni d'une action de
$\tau$ décrite par le lemme \ref{PoneyFringant}.
Selon les lemmes \ref{connemaraTT0} et \ref{connemara},
$E/E_0$ est ramifiée si et seulement si $T/T_0$ est ramifiée.

La proposition \ref{entrenous} dit qu'un 
caractère simple maximal d'endo-classe $\TT$ de $G$
est de la forme $\t^y$ pour un $y\in\G$,
et un tel caractère est $\tau$-autodual 
si et seulement si $w=\tau(y)y^{-1}$ ap\-par\-tient au normalisateur de 
$\t$ dans $\G$. 
Comme $\tau (w)=w^{-1}$, cela équivaut même à ce que $w$
appar\-tien\-ne à $\BJ^0$. 
Notons $u$ l'image de $w$ dans $\BJ^0/\BJ^1\simeq\GL_{m}(\ee)$. 

Si $c_0$ est impair et si $E$ est non ramifiée sur $E_0$,
l'involution $\tau$ agit sur $\GL_m(\ee)$ comme~l'auto\-morphis\-me
non trivial de $\ee/\ee_0$.
Comme $u\tau(u)=1$,
il y a un $z\in \GL_m(\ee)$~tel que
$u=z\tau (z)^{-1}$.~En\-suite, si $c_0$ est pair et si $E$ est ramifiée sur $E_0$, 
l'involution $\tau$ agit sur $\GL_m(\ee)$ par
conjugaison par un élément $\v$ 
tel que $\v^2\in\ee^\times$ et $\v^2\notin\ee^{\times2}$.
{\'Ecrivons~:
\begin{equation*}
(uv)^2 = uvuv = u\tau(u)v^2 = v^2. 
\end{equation*}
D'après le théorème de~Sko\-lem-Noether,
il existe un $z\in \GL_m(\ee)$ tel que $uv=zvz^{-1}$,
\ie que $u=z\tau (z)^{-1}$.}
Dans les deux cas,
{remplaçant $y$ par $jy$,
où $j$ est un~re\-lè\-ve\-ment de $\tau(z)^{-1}$~au groupe $\BJ^{0}$,}
on peut supposer que $w\in\BJ^{1}$.
Le premier~en\-sem\-ble~de~coho\-mologie de
$\langle\tau\rangle$ dans~$\BJ^{1}$ étant 
trivial, il y a un $x\in\BJ^{1}$ tel que $w=\tau (x)x^{-1}$.
On a donc $x^{-1}y\in\G^\tau$.~Comme~$\BJ^1$~norma\-lise $\t$, 
il s'ensuit que $\t^y$ est con\-jugué à $\t$ sous $\G^\tau$
(par $x^{-1}y$).
On en déduit (1.a),
et les assertions (2.a) et (2.b) suivent du lemme \ref{PoneyFringant}.
Comme le normalisateur de $\t$ dans $G$ est,
d'après la proposition \ref{patel},
engendré par $\BJ^0$ et le normalisateur de $\bb$ dans $B^\times$,
on déduit également de ce même lemme l'assertion (3)
dans le cas où $E/E_0$ est non ramifiée ou $c_0$ est pair. 

Supposons maintenant que $c_0$ soit impair et que
$\E/\E_0$ soit ramifiée.
L'action de $\tau$ sur $\GL_{m}(\ee)$ étant triviale, on a $u^2=1$,
\ie qu'il y a un $z\in\GL_{m}(\ee)$ et un $i\in\{0,\dots,m\}$ tels que~:
\begin{equation*}
z u z^{-1} = \s_i
\end{equation*}
où $\s_i$ est la matrice diagonale définie par \eqref{defsi42}.
{Remplaçant $y$ par $jy$ où $j\in\BJ^{0}$ est un relèvement de $z$,} 
on peut supposer que $s_i w\in\BJ^{1}$,
où $s_i$ est la matrice 
$ {\rm diag}(-1,\dots,-1,1,\dots,1)\in\BJ^{0}$~re\-le\-vant $\s_i$.
Selon le lemme \ref{PoneyFringant}(2),
le normalisateur de $\bb$ dans $\B^\times$ est~en\-gendré par $\bb^\times$ 
et une~uni\-formi\-sante $\w\in\C$ telle que $\w^c$ soit une
uniformi\-san\-te~de $E$ et $\tau(\w)=-\w$.
Posons~: 
\begin{equation*}
t_i = {\rm diag}(\w,\dots,\w,1,\dots,1) \in \B^\times=\GL_m(C),
\end{equation*}
où $\w$ apparaît $i$ fois.
On remarque que $s_it_i=\tau(t_i)$.
Comme $\BJ^{1t_i}$ est un pro-$p$-groupe $\tau$-stable,
le premier ensemble de cohomologie de $\langle\tau\rangle$
dans $\BJ^{1t_i}$  est trivial. 
Il s'ensuit que $\t^y$ est conjugué~à $\t_i=\t^{t_i}$ sous $G^\tau$. 
Enfin, pour vérifier que $\t_i$ et $\t_k$ sont 
$G^\tau$-conjugués si et seulement si $i=k$ ou
$i+k=m$, 
on raisonne comme dans la preuve~de \cite[Lemma 4.25]{AKMSS}. 
Prouvons enfin (3).
D'après~la remarque \ref{malinlechien}, 
le normalisateur de $\t$ dans $G$ est engen\-dré~par
$\BJ^0$ et $\w$.
Celui de $\t_i$ est donc engendré par
$\BJ^{0t_i}$ et l'élément $t_i^{-1}\w t_i^{}=\w$,
qui a les propriétés requises. 
\end{proof}

\subsection{}

Avant de clore cette section,
voici un résultat faisant pendant au
co\-rol\-laire \ref{tolkienelrond}(1).~L'en\-do-classe
$\TT$ est~tou\-jours celle qui a été fixée au paragraphe 
\ref{vitevite}.

\begin{prop}
\label{qqq}
Soit $\pi$ une représentation cuspidale autoduale de $G$
d'endo-classe $\TT$.
On suppose que $r$ est impair. 
Il y a une strate simple maximale $[\aa,\b]$, 
un caractère simple $\t\in\Cc(\aa,\b)$ contenu dans $\pi$ et un $u\in\G$
tels que, 
posant $E=F[\b]$ et $E_0=F[\b^2]$,
on ait $\t^{-1}=\t^{u}$, $\b^{u}=-\b$
et $\Nrd_{B/E}(u^2)\notin\N_{E/E_0}(\E^\times)$. 
\end{prop}

\begin{proof}
Observons d'abord que la condition $\b^{u}=-\b$ entraîne que
$u^2\in\B^\times$,
ce qui assure qu'on peut former la norme réduite de $u^2$.
Comme $r$ est impair et $rd=2n$,~l'entier~$d$~est pair. 
Comme dans la preuve du théorème \ref{THMEXISTENCECARACSIMPLE}, 
fixons un $\s\in\GL_{2n}(F)$ tel que~$\s^2=1$,
de~poly\-nô\-me caractéristique $(X^2-1)^n$.
D'après la proposition \ref{redformadap},
il existe une strate simple~maxi\-male $\s$-autoduale $[\aa_1,\b]$
de $\Mat_{2n}(F)$ 
telle que $\Cc(\aa_1,\b)$ contienne un caractère simple maximal
$\s$-auto\-dual 
$\t_1$ d'endo-classe $\TT$. 
Posons $E=\F[\b]$ et $E_0=\F[\b^2]$.

Fixons un plongement de $F$-al\-gè\-bres~de~$E_0$ dans $A$.
Identifions $E_0$ à son image par ce plongement
et le centralisateur $A_0$ de $E_0$ dans $A$ à~une $E_0$-algèbre 
$\Mat_k(U_0)$, où $U_0$ est une $E_0$-algèbre à division,
de degré réduit noté $u_0$, et où $k$ est un entier strictement positif.
Selon les formules~de \eqref{egalizink}, 
on a $ku_0=2n/[E_0:F]$,
qui est pair car $[\E:\F]=2[E_0:F]$ divise $2n$,~et $k$ divise~$r$.~On
en déduit que $k$ est impair et que $u_0$ est pair.
On peut donc plonger $E$ dans $U_0$ comme $E_0$-al\-gè\-bre,
ce que l'on fait. 
Soit $C$ le~cen\-tra\-li\-sateur de $E$ dans $U_0$.
Le centralisateur $B$ de $E$ dans $A$ s'identifie à $\Mat_k(C)$.
On a donc $k=m$.
Notons $\bb$ l'ordre maximal standard $\Mat_m(\Oo_C)$ de~$B$,
et $\aa$ l'ordre~hé\-ré\-ditaire de $\A$ normalisé par $\E^\times$
tel que $\aa\cap\B=\bb$. 
Soit $\t\in\Cc(\aa,\b)$ le transfert de $\t_1$. 
Observons que $m$ est impair, 
car c'est un diviseur de $r$ d'après \eqref{egalizink}.

\begin{lemm}
\label{lemmey}
Pour tout $u\in U_0^\times$ tel que $\b^u =-\b$,
on a $\t^{-1}=\t^{u}$.
\end{lemm}

\begin{proof}
Un tel élément $u$ normalise $E$, donc son centralisateur $C$.
Il normalise donc $\Oo_C$, donc aussi $\bb$ et $\aa$. 
Le résultat découle alors du lem\-me \ref{lemmedetransfertautodual}. 
\end{proof}

Fixons une extension non ramifiée maximale~$L_0$ de $E_0$ dans $U_0$
et une uniformisante~$\w_0$~de~$U_0$ nor\-malisant $L_0$ 
telle que $\w_0^{u_0}$ soit une uniformisante de $E_0$. 
L'élément $\w_0$ induit, par conjugaison sur $L_0$, 
un générateur~de $\Gal(L_0/E_0)$. 

Supposons d'abord que $E/E_0$ soit non ramifiée. 
On peut supposer que $E$ est inclus dans $L_0$, au\-quel 
cas il est le corps des points fixes de $\w_0^2$ dans $L_0$.
On a donc $\w^{\phantom{'}}_0\b\w_0^{-1}=-\b$.
D'autre part, $\w_0^2$ étant une uniformisante de $\C$,
sa norme réduite relativement à $B/E$ est de valuation $m$~dans $E$.
L'entier $m$ étant impair, 
celle-ci n'appartient pas à $\N_{E/E_0}(\E^\times)$.
Appliquant le~lem\-me~\ref{lemmey}~à~$u=\w_0$,
on trouve que $\t^{-1}=\t^{u}$.

Supposons maintenant que $E/E_0$ soit ramifiée. 
\'Ecrivons $u_0=2c$.
On peut supposer que~$\w_0^{c}$ est une uniformisante $\w_E$ de $E$.
Celle-ci induit,
par conjugaison sur $L_0$,
l'unique automorphisme d'or\-dre $2$ de $L_0/E_0$.
Il existe donc 
une racine de l'unité $\xi\in\boldsymbol{\mu}_{L_0}$ telle que
$\w_E\xi = -\xi\w_E$.
À~son tour, 
cette racine de l'unité induit, 
par conjugaison sur $E$,
l'automorphisme non trivial de $E/E_0$.
Comme $\b\notin\E_0$ et $\b^2\in\E_0$,
on en déduit que $\xi\b\xi^{-1}=-\b$.
Enfin, 
comme $\xi^2\in\C$,
on a~:
\begin{equation}
\label{normeBExi2}
\Nrd_{B/E}(\xi^2) = \N_{L/E}(\xi^2)^m
\end{equation}
où $L/E$ est une extension non ramifiée maximale de $E$ dans $C$.
L'entier $m$ étant impair et $L/E$ étant non ramifiée,
cette quantité appartient à $\N_{E/E_0}(E^\times)$ si et seulement si
$\N_{\ee/\kk_E}(\xi^2)\in\kk_E^{\times2}$,~où $\ee$ est le corps résiduel de $L$,
et cette dernière condition équivaut à $\xi^2\in\ee^{\times2}$,
\ie à $\xi\in\ee^{\times}$.
Comme $\xi$ ne commute pas à $E$,
on en déduit que \eqref{normeBExi2} n'appartient pas à
$\N_{E/E_0}(\E^\times)$.
En~appli\-quant le lemme \ref{lemmey} à $u=\xi$,
on trouve que $\t^{-1}=\t^{u}$.
\end{proof}

\section{Caractères simples $\tau$-autoduaux et distinction}
\label{SEC5}

Dans cette section,
on suppose être dans la situation introduite dans la section \ref{SEC4}.
On~a donc une $F$-algèbre centrale simple $A$ de degré réduit $2n$,
un élément $\a\in\F^{\times}$ et un $\dd\in\A^\times$ non central
tel que $\dd^2=\a$.
On pose $\G=\A^\times$ et on note $\tau$ l'automorphisme
involutif de conjugaison~par~$\dd$.
Rappelons que,
compte tenu de la convention \ref{CONV2},
le poly\-nôme caractéristique réduit de $\dd$ sur~$F$ est égal à $(X^2-\a)^n$.
En particulier,
lorsque $\a\in\F^{\times2}$,
auquel cas on suppose que $\a=1$,
l'entier $r$ est pair. 
Nous al\-lons prouver deux résultats essentiels pour la suite,
faisant tous deux un lien entre distinction par $G^\tau$ et existence d'un
caractère simple maximal $\tau$-autodual~:
\begin{itemize}
\item 
nous prouvons d'abord,
dans le paragraphe \ref{simples},
que toute représentation cuspidale
$\G^\tau$-dis\-tin\-guée de niveau non nul 
de $G$ con\-tient un caractère simple $\tau$-autodual~; 
\item
nous prouvons ensuite,
dans le paragraphe \ref{malinlegarconcor},
que l'ensemble des représentations cuspi\-da\-les autoduales
de~$G$~d'en\-do-classe non nulle donnée,
s'il contient une représentation $\G^\tau$-dis\-tin\-guée,
en est constitué au moins pour moitié.
\end{itemize}

Sauf dans le paragraphe \ref{guopar},
nous ne considérerons dans cette section
que des représentations de niveau non nul de $G$.
Le cas de représentations de niveau $0$ sera 
traité dans la section \ref{theriverofnoreturn}.

\subsection{}
\label{guopar}

Le résultat suivant jouera un rôle important par la suite.  

\begin{theo}
\label{GuoBM}
Toute représentation cuspidale $\G^\tau$-distinguée de $G$
est autoduale.  
\end{theo}

\begin{proof}
Lorsque $\a$ n'est pas un carré de $F^\times$,
le résultat est initialement dû à
Guo~\cite{GuoPJM97} si $d\<2$ et si $F$ est 
de ca\-ractéristique nulle. 
Dans le cas d'une forme intérieure quelconque et~d'un corps 
$F$ de carac\-téristique résiduelle $p\neq2$,
il est prouvé par Chommaux-Matringe
\cite[Pro\-po\-si\-tion 3.2]{ChommauxMatringe}.

Lorsque $\a$ est un carré dans $F^\times$,
le résultat est dû à Jac\-quet-Rallis \cite[Theorem 1.1]{JR} si~$G$ est déployé
et $F$ est de caractéristique nulle.
Pour traiter le cas d'une forme intérieure
quelconque sur un corps $F$ de caractéristique résiduelle impaire,
il suffit d'utiliser un argument de globalisation similaire à celui de 
\cite{BroussousMatringe} et \cite{ChommauxMatringe}.
Comme $\a\in F^{\times 2}$,
rappelons que $\A\simeq\Mat_r(D)$ avec $r$ pair.

Soit $k$ un corps global possédant une place $w$ divisant $p$
pour laquelle $k_w$ soit isomorphe à $F$,
et soit $l$ une extension quadratique de $k$
telle que $l_w=l\otimes_{k}k_w$ soit décomposée sur
$k_w$.~Soit~$\mathbf{A}$ l'anneau des adèles de $k$.
Fixons des isomorphismes $k_w\simeq\F$, $l_w\simeq\F\oplus\F$. 
Selon~\cite[Theo\-rem 1.12]{PlatoRapin},
il y a une $k$-algèbre à division $\mathbb{D}$ telle que~:
\begin{enumerate}
\item 
il existe un isomorphisme de $F$-algèbres
$\mathbb{D}\otimes_{k}F\simeq\A$, 
\item
il existe un plongement de $k$-algèbres de $l$ dans $\mathbb{D}$ 
vérifiant les conditions suivantes:
\begin{enumerate}
\item 
l'image 
du centralisateur de $l_w$ dans $(\mathbb{D}\otimes_k F)^\times$
par l'isomorphisme ci-dessus est $\G^\tau$,
\item
pour presque toute place finie $v$ de $k$,
la $k_v$-algèbre $\mathbb{D}\otimes_k k_v$ est déployée et le
centra\-lisateur de $l_v=l\otimes_k k_v$ dans
$(\mathbb{D}\otimes_k k_v)^\times\simeq\GL_{2n}(k_v)$
est isomorphe à $\GL_n(k_v)\times\GL_n(k_v)$~si $v$ est décomposée
et à $\GL_n(l_v)$ si $v$ est inerte.
\end{enumerate}
\end{enumerate}
(Il suffit de fixer un ensemble fini $S$ de places finies de $k$ contenant $w$
et d'impo\-ser que $\mathbb{D}\otimes_k k_v$
soit~déployée pour toute place $v\notin S$,
auquel cas $l_v$ se plonge dans
$\mathbb{D}\otimes_k k_v\simeq\Mat_{2n}(k_v)$
comme souhaité,
et de fixer pour toute place $v\in S$ une $k_v$-algèbre à division centrale
$D_v$ de degré~ré\-duit~$d_v$ et d'invariant de Hasse $h_v/d_v\in\QQ/\ZZ$
tels que~:
\begin{equation*}
\sum\limits_{v\in S} \frac {h_v} {d_v} \in \ZZ
\end{equation*}
et $D_w\simeq D$.
On peut alors imposer à $\mathbb{D}$, 
outre les conditions (1) et (2) ci-dessus,
la condition~sup\-plémentaire
$\mathbb{D}\otimes_k k_v\simeq\Mat_{2n/d_v}(D_v)$ pour
toute $v\in S$.)
On pose $\mathbb{G}=(\mathbb{D}\otimes_{k}\mathbf{A})^\times$
et on note~$\mathbb{H}$~le centralisateur de $l$ dans $\mathbb{G}$. 
\`A la place $w$, on a donc $\mathbb{G}_{w}\simeq\GL_{r}(D)$
et $\mathbb{H}_{w}\simeq\GL_{s}(D)\times\GL_{s}(D)$
avec $r=2s$.

Soit $\pi$ une représentation cuspidale de $G\simeq\mathbb{G}_{w}$
distinguée par $G^\tau\simeq\mathbb{H}_{w}$.
Selon \cite[Theo\-rem 4.1]{PSP} et \cite[Theo\-rem 1.3]{GL}, 
il existe une représentation automorphe cuspidale $\Pi$ de $\mathbb{G}$
telle~que la composante locale de $\Pi$ en $w$
soit isomorphe à $\pi$,
et telle que $\Pi$ ait une période non nulle~re\-lativement à $\mathbb{H}$.
Pour~toute place finie $v$ de $k$,
la composante locale $\Pi_v$ de $\Pi$ en $v$ est distinguée par 
$\mathbb{H}_v$.
En presque~toute place finie $v$,
le groupe $\mathbb{G}_v$ est isomorphe à $\GL_{2n}(k_v)$
et $\mathbb{H}_v$ est~iso\-mor\-phe à $\GL_n(k_v)\times\GL_n(k_v)$
ou à $\GL_n(l_v)$
(selon que $v$ est décomposée ou inerte dans $l$),
et $\Pi_v$ est générique et non ramifiée. 
Ainsi \cite[Proposition 3.1]{ChommauxMatringe} entraîne que
$\Pi_v$ est autoduale en~presque toute place finie $v$.
Appliquant~le théorème de multiplicité $1$ forte pour les formes
intérieures~de $\GL_n$ sur $k$
(voir \cite[Theorem 5.1]{BaduINVENT} en carac\-té\-ris\-tique nulle,
\cite[Theorem 3.3]{BaduRoche} en~caractéristique~$p$),
on en déduit que les représen\-ta\-tions automorphes cuspidales $\Pi$ et 
$\Pi^\vee$ sont isomorphes.
En~par\-ti\-culier, la représentation $\Pi_w\simeq\pi$ est autoduale. 
\end{proof}

\subsection{}
\label{P39}

Prouvons la variante\footnote{Je remercie Jiandi~Zou d'avoir attiré mon attention
sur le fait qu'une telle variante est possible.}
suivante de \cite[Lemma 6.5]{VSANT19}.

\begin{lemm}
\label{bathmologieabstraite}
Soit $\iota$ une involution de $G$,
soit $H$ un pro-$p$-sous-groupe ouvert compact de~$G$~et
soit $\chi$ un caractère de $H$.
Supposons qu'il y ait un élément $w\in\G$ tel que~:
\begin{equation*}
\iota(H)=H^w
\quad\text{et}\quad
\chi^{-1}\circ\iota=\chi^w.
\end{equation*}
Pour tout $g\in\G$, le caractère $\chi^g$ est trivial sur
$H^g\cap\G^\iota$ si et seulement si $w\iota(g)g^{-1}$
entrelace $\chi$.
\end{lemm}

\begin{proof}
Soit $g\in\G$.
Exactement comme dans la preuve de \cite[Lemma 6.5]{VSANT19},
on~prou\-ve que $\chi^g$ est trivial sur
$H^g\cap\G^\iota$ si et seulement si~:
\begin{equation}
\label{W1}
\chi^g(\iota(x)) = \chi^g(x)^{-1},
\quad
\text{pour tout $x\in\H^{g}\cap\iota(\H^{g})$}.
\end{equation}
Ensuite, par hypothèse sur $\chi$, on a~:
\begin{equation*}
\label{W2} 
\chi^g\circ\iota=(\chi\circ\iota)^{\iota(g)}
=(\chi^{-1})^{w\iota(g)}=(\chi^{w\iota(g)})^{-1}
\end{equation*}
sur $\H^{w\iota(g)}$.
Si l'on pose $\g=w\iota(g)g^{-1}$, alors \eqref{W1} 
est équivalent à l'identité $\chi(h)=\chi^{\g}(h)$
pour tout $h\in\H\cap\H^{\g}$,
\ie que $\g$ entrelace $\chi$.
\end{proof}

\subsection{}
\label{simples}

On en déduit le résultat suivant,
donnant une condition nécessaire de $\G^\tau$-dis\-tinc\-tion~pour les 
représentations cuspidales de $G$ en termes de théorie des types,
et qui est le premier résultat principal de cette section. 

\begin{prop}
\label{grammage}
Toute représentation cuspidale $\G^\tau$-distinguée
de niveau non nul de $G$ 
contient un caractère simple $\tau$-autodual. 
\end{prop}

\begin{proof}
Soit $\pi$ une représentation cuspidale $\G^\tau$-distinguée de $G$, 
de niveau non nul.
{D'après le théorème \ref{GuoBM},
elle est autoduale.}
On note $\TT$ son endo-classe, et on définit des entiers $m$ et $c$ par 
\eqref{bananerotie}.

\begin{lemm} 
\label{tracteur}
Il y a une strate simple maximale $[\aa,\b]$, 
un caractère simple $\t\in\Cc(\aa,\b)$ contenu dans $\pi$ et un $u\in\G$
tels que, 
posant $E=F[\b]$ et $E_0=F[\b^2]$,
on ait $\t^{-1}=\t^{u}$, $\b^{u}=-\b$ et~:
\begin{equation*}
\bo_{E/E_0}(\Nrd_{B/E}(u^2))=(-1)^{r}.
\end{equation*}
\end{lemm}

\begin{proof}
Si $r$ est impair,
le résultat est donné par la proposition \ref{qqq}.

Si $r$ est pair,
fixons comme au paragraphe \ref{Biotabstait0}
un~élé\-ment $\s\in\G$ tel que $\ss^2=1$, 
de~poly\-nô\-me ca\-rac\-téristique réduit $(X^2-1)^{n}$.
D'après le corollaire \ref{tolkienelrond},
la représentation $\pi$ contient 
un ca\-rac\-tère simple maximal $\s$-autodual $\t$.
Fixons une strate simple $\s$-autoduale $[\aa,\b]$ de $A$ telle que
$\t\in\Cc(\aa,\b)$,
dont l'existence est assurée par le lemme \ref{gandalf}.
On obtient le résultat voulu avec $u=\s$.
\end{proof}

Soient $\t\in\Cc(\aa,\b)$ et $u\in\G$ comme dans le lemme \ref{tracteur}.
Notons $\BJI_{\t}$ le~norma\-li\-sa\-teur de~$\t$~dans $G$, 
et notons respectivement $\BJ^0$ et $\BJ^1$ 
le sous-groupe compact maximal et le 
pro-$p$-sous-groupe dis\-tingué~maximal de $\BJ_\t$. 
Po\-sons $E=\F[\b]$ et $E_0=\F[\b^2]$.
L'induite compacte de~$\t$ à $\G$,
dont~$\pi$ est~un quotient,
admet une forme linéaire $G^\tau$-invariante non nulle. 
Appliquant la formule~de~Mac\-key,
il existe un~$g\in\G$ tel que
$\t^g$ soit trivial sur $\H^{1}(\aa,\b)^g\cap\G^\tau$. 
On observera que
$u$ normalise $B^\times$ et $u^2\in\BJI_\t\cap\B^\times$.

Posons maintenant $w=u\dd$,
de sorte que $\t^{-1}\circ\tau=\t^w$. 
Appliquant~le lemme \ref{bathmologieabstraite} 
au caractère $\t$~et à l'involution $\tau$,
on en déduit que $\g=w\tau(g)g^{-1}$~en\-tre\-la\-ce $\t$,
c'est-à-dire, d'après la proposi\-tion \ref{patel}\eqref{bourrin6}, 
que $\g\in\BJ^1\B^\times\BJ^1$,
où $\B$ désigne le~cen\-tra\-li\-sa\-teur de $E$ dans $A$.
L'élément $\g$ vérifie~: 
\begin{equation}
\label{marmouset2}
u\g u^{-1}=\a u^2\g^{-1}.
\end{equation}
Fixons un élément $x\in\B^\times$ tel que $\g\in\BJ^1 x\BJ^1$.
Le groupe $\BJI_{\t}$ est stable par conjugaison par $u$, 
car c'est à la fois le normalisateur de $\t$ et de $\t^{-1}=\t^u$.
Il en est donc de même de $\BJ^1$, ce qui donne~:
\begin{equation}
\label{marmouset3}
\BJ^1 uxu^{-1}\BJ^1=\BJ^1 \a u^2x^{-1}\BJ^1.
\end{equation} 
Prenant l'inter\-sec\-tion avec $\B^\times$
et appliquant \cite[Corollaire 3.3]{VS1}, 
l'égalité \eqref{marmouset3} devient~:
\begin{equation}
\label{marmouset35}
\BU^1(\bb) uxu^{-1}\BU^1(\bb) = \BU^1(\bb) \a u^2x^{-1} \BU^1(\bb). 
\end{equation}
Puis, prenant la norme réduite de $B$ sur $E$, on obtient l'identité~:
\begin{equation}
\label{marmouset4}
\Nrd_{B/E}(x) \cdot \Nrd_{B/E}(uxu^{-1}) \equiv \a^{2n/[E:F]} 
\cdot \Nrd_{B/E}(u^2)
\end{equation}
dans $E^\times/(1+\p_E)$, 
du fait que l'image de $\BU^1(\bb)$ par $\Nrd_{B/E}$ est égale à $1+\p_E$. 
D'après le lemme \ref{NRDBE}, 
le membre de gauche de \eqref{marmouset4} est égal à 
$\N_{E/E_0}(\Nrd_{B/E}(x))$ et $\Nrd_{B/E}(u^2)\in\E_0^\times$.
L'identité \eqref{marmouset4} est donc valable dans $E_0^\times/(1+\p_{E_0})$.
On peut ainsi lui appliquer le caractère quadratique $\bo_{E/E_0}$, ce qui donne~:
\begin{equation*}
\bo_{E/E_0}(\a)^{2n/[E:F]} = \bo_{E/E_0}(\Nrd_{B/E}(u^2)) = (-1)^r.
\end{equation*}
Appliquant le lemme \ref{connemara} et le théorème 
\ref{THMEXISTENCECARACSIMPLE}, 
on en déduit que $\pi$ contient un caractère simple $\tau$-autodual. 
\end{proof}

\subsection{}
\label{couragelesgars}
\label{viteA}

Fixons un caractère simple maximal
$\tau$-autodual non trivial $\t$ de $G$.~On
note $\BJI_{\t}$ le normali\-sateur de $\t$ dans $G$,
$\BJ^0$ le sous-groupe compact maxi\-mal de $\BJI_{\t}$ 
et $\BJ^1$ son pro-$p$-sous-groupe~nor\-mal~maximal.
L'objectif des paragraphes \ref{couragelesgars} à \ref{eb} qui suivent
est de montrer que, sous certaines hypothèses,
la représentation de Heisenberg $\n$ admet un prolongement à $\BJI_{\t}$,
ou à certains sous-groupes d'indice fini de $\BJI_{\t}$,
possédant des propriétés remarquables vis-à-vis de $\tau$~:
voir les~corol\-laires \ref{exkappadistaddetp} et \ref{exkappadistaddetpnr}.
On fixe une strate simple $\tau$-autoduale $[\aa,\b]$ telle que
$\t\in\Cc(\aa,\b)$,
dont l'exis\-ten\-ce est assurée par le lemme \ref{gandalf}, 
et on utilise les notations~de la section \ref{SEC4}.
En particulier~:
\begin{itemize}
\item
on pose $E=F[\b]$ et $E_0=F[\b^2]$,
\item
on note $B$ le centralisateur~de~$E$ dans $A$
et $m$, $c$ les entiers définis par~\eqref{egalizink},
\item
on note $c_0$ l'entier défini par \eqref{marquisat} avec $g=[E_0:F]$.
\end{itemize}
D'après la proposition \ref{indiceduntype}, 
on peut iden\-tifier $\BJ^0/\BJ^1$ à $\GL_m(\ee)$ de
façon que l'action de $\tau$ soit~: 
\begin{enumerate}
\item 
l'action de l'automorphisme d'ordre $2$ de $\ee$
si $E/E_0$ est non ramifiée (donc $c_0$ impair), 
\item
l'action par conjugaison d'un élément $\v\in\GL_m(\ee)$
si $E/E_0$ est ramifiée, avec~:
\begin{enumerate}
\item 
$\v^2\in\ee^\times$ et $\v^2\notin\ee^{\times2}$ si $c_0$ est pair,
\item
$\v$ est diagonal, $\v^2=1$ et les multiplicités des valeurs propres 
$-1$ et $1$ sont respectivement $i$ et $m-i$,
avec $i\in\{0,\dots,\lfloor m/2\rfloor\}$,
si $c_0$ est impair. 
\end{enumerate}
\end{enumerate}
(Dans ce dernier cas, 
rappelons que l'entier $i$ s'appelle l'indice de $\t$
(voir la définition \ref{defindicetheta}).)
On fixe également un élément $\w$ de $\BJ_\t\cap\B^\times$ comme dans
la proposition \ref{indiceduntype}(3), 
\ie que $\w^c$ est une uniformisante de $E$,
le groupe $\BJ_\t$ est engendré par $\BJ^0$ et $\w$ et~: 
\begin{equation*}
\label{choixweqBbis}
\tau(\w) = \left\{ 
\begin{array}{rl}
-\w & \text{si $T/T_0$ est ramifiée et $c_0$ est impair}, \\ 
\w & \text{sinon.}
\end{array}\right.
\end{equation*}
On ne suppose pas ici que $\t$ apparaît dans une représentation
cuspidale autoduale de $G$.
On ne peut donc pas utiliser le corollaire \ref{fromagequipuecor}.

\begin{lemm}
\label{choixw}
\label{descrJtau} 
Le groupe $\BJI_{\t}\cap\G^\tau$ est engendré par $\BJ^{0}\cap\G^\tau$ et
un élément $\w'$ tel que~:
\begin{enumerate}
\item
$\w'=\w$ si $\E/\E_0$ est non ramifiée,
ou si $E/E_0$ est ramifiée et $c_0$ est pair,
\item
$\w'=\w^2$ si $\E/\E_0$ est ramifiée, $c_0$ est impair et {$m\neq2i$,}
\item
$\w'=\w t$ pour un élément $t\in\BJ^0$
dont la réduction modulo $\BJ^1$ est~:
\begin{equation}
\label{deftantidiag}
\begin{pmatrix}0&{\rm id}\\{\rm id}&0\end{pmatrix}
\in \GL_m(\ee)
\end{equation}
si $\E/\E_0$ est ramifiée, $c_0$ est impair et $m=2i$,
où ${\rm id}$ est l'identité de $\GL_i(\ee)$.
\end{enumerate}
\end{lemm}

\begin{proof}
Soit un élément $x\in\BJI_{\t}\cap\G^\tau$, qu'on écrit $\w^ky$
avec $k\in\ZZ$ et $y\in\BJ^0$.
Si $\E/\E_0$ est non~ra\-mi\-fiée, ou~si $E/E_0$ est ramifiée et $c_0$ est 
pair, alors 
$\w$ est invariant par $\tau$, donc $y$ l'est aussi,
ce qui~trai\-te~le~cas (1). 
Si $\E/\E_0$ est ramifiée et $c_0$ est impair, 
on a $\tau(\w)=-\w$.~Ainsi, 
ou bien $k$ est pair et $\tau(y)=y$,
ou bien $k$ est impair et $\tau(y)=-y$.
Supposons qu'on soit~dans~le second~cas.
Réduisant $y$ modulo $\BJ^1$,
on obtient un $z\in\GL_m(\ee)$ tel que $\tau(z)=-z$.
Comme $\tau$~agit sur $\GL_m(\ee)$ par conjugaison par 
l'élément $\s_i$ défini par \eqref{defsi42}, 
on en déduit que $\s_i$ et $-\s_i$ sont conjugués, 
donc que $m=2i$.

Inversement,
si $m=2i$,
on note $s$ l'élément \eqref{deftantidiag},
qui est anti-invariant par $\tau$.
Si $u\in\BJ^0$ relève $s$,
on a donc $v=-\tau(u)u^{-1}\in\BJ^1$.
Comme $\tau(v)v=1$,
il existe un $w\in\BJ^1$ tel que $v=\tau(w)w^{-1}$.
L'élément $t=w^{-1}u$ a la propriété requise.
\end{proof}

\subsection{}
\label{onnoubliepas}
\label{ea}

La situation est la même qu'au paragraphe précédent. 
On suppose en outre que,
si $E/E_0$ est ramifiée et $c_0$ est impair,
$m$ est pair ou égal à $1$
et $\t$ est d'indice $i=\lfloor m/2\rfloor$.
(Voir la remarque \ref{pluiedautomne}.)
On a donc $m=1$ ou $m=2i$.
Rappelons que,
comme $p\neq2$, tout caractère de $\GL_m(\ee)$ se fac\-torise
par le déterminant.

\begin{prop}
\label{exkappadist}
Supposons que l'on soit dans l'un des cas suivants~:
\begin{enumerate}
\item 
$E/E_0$ est non ramifiée, ou $E/E_0$ est ramifiée et $c_0$ est pair,
\item
$E/E_0$ est ramifiée, $c_0$ est impair, $m$ est pair ou égal à $1$
et $\t$ est d'indice $i=\lfloor m/2\rfloor$.
\end{enumerate}
Alors il existe une représentation
$\BJI_{\t}\cap\G^\tau$-distinguée de $\BJI_{\t}$ prolongeant $\n$. 
\end{prop}

\begin{proof}
Soit $\bk$ une représentation $\tau$-autoduale de $\BJI_{\t}$ prolongeant $\n$,
dont l'existence est assurée par la proposition \ref{kappatauautoduale}.
Selon le lemme \ref{kappatau}, il lui~correspond
un caractère $\chi$ du~grou\-pe $\BJI_{\t}\cap\G^\tau$
trivial sur $\BJ^1\cap\G^\tau$ tel que 
$\Hom_{\BJI_{\t}\cap\G^\tau}(\bk,\chi^{-1})$ soit non nul.
Comme $\bk$ est $\tau$-autoduale,
$\chi^2$ est trivial. 
On cherche un caractère $\qq$ de $\BJI_{\t}$ trivial sur $\BJ^1$ tel que 
$\bk\qq$ soit $\BJI_{\t}\cap\G^\tau$-distinguée, \ie telle
que la restriction de $\qq$ à 
$\BJI_{\t}\cap\G^\tau$ soit égale à $\chi$.
Soit $\w\in\BJI_{\t}\cap\B^\times$ comme dans le paragraphe \ref{couragelesgars}.
Commençons par traiter le cas (1).

Si $E/E_0$ est non ramifiée,
notons $\ee_0$ le sous-corps des invariants de $\ee$ par son
automorphisme d'ordre $2$.
La restriction de $\chi$ à $\BJ^0\cap\G^\tau$ s'identifie à un
caractère de $\GL_m(\ee_0)$,
qui se factorise sous la forme $\delta\circ\det$
pour un unique caractère quadrati\-que $\d$ de $\ee_0^\times$,
qu'on écrit $\d=\e\circ\N_{\ee_0/\kk_{E_0}}$
pour un unique caractère quadrati\-que $\e$ de $\kk_{E_0}^\times$.
Soit $\g$ un~caractère~de~$\kk_E^\times$~pro\-longeant $\e$.
On pose $\h=\g\circ\N_{\ee/\kk_{E}}$.
Le degré de $\ee/\kk_E$,
qui est égal à $c_0$,
est impair.
La~restric\-tion~de $\h$ à $\ee_0^\times$ est donc égale à $\d$. 

Si $E/E_0$ est ramifiée et $c_0$ est pair,
soit $\ee_0$ l'extension quadratique de $\ee$ 
engendrée par $\v$~dans $\Mat_m(\ee)$.
La restriction de $\chi$ à $\BJ^0\cap\G^\tau$ s'identifie à un
carac\-tè\-re~de $\GL_{m/2}(\ee_0)$
de la forme $\delta\circ\det$
où $\delta$ est un caractère quadrati\-que de $\ee_0^\times$.
On note $\varphi$ l'unique~ca\-ractère quadrati\-que
de $\ee^\times$ tel que
$\delta$ soit égal à $\varphi\circ\N_{\ee_0/\ee}$.

Dans ces deux cas,
$\h$ est invariant par $\Gal(\ee/\kk_E)$.
Le caractère $\qq$ de $\BJ^0$ trivial sur $\BJ^1$
obtenu par inflation  du caractère
$\varphi\circ\det$ de $\GL_{m}(\ee)$
est donc normalisé par $\w$,
et son unique prolonge\-ment~à $\BJI_{\t}$ tel que $\qq(\w)=\chi(\w)$ a la 
propriété voulue.
Traitons maintenant le cas (2).

Dans le cas où $m=1$, 
on a $\BJ^0=(\BJ^{0}\cap\G^{\tau})\BJ^1$.
Il existe donc un unique caractère $\qq$ de~$\BJ^0$
trivial sur $\BJ^1$ et 
coïncidant avec $\chi$ sur $\BJ^0\cap\G^{\tau}$,
et il est normalisé par $\BJI_\t$.
On le prolonge à $\BJI_\t$
en demandant que $\qq(\w)^2=\chi(\w^2)$.
On obtient un caractère ayant la propriété voulue. 

Enfin,
dans le cas où $m=2i$,
on suit l'argument~de la preuve de \cite[Lemma 7.10]{VSANT19}
dont~on repro\-duit ici les grandes lignes pour la commodité du lecteur.
On identifie
la restriction de~$\chi$~à $\BJ^0\cap\G^\tau$ à un
caractère de $\GL_{i}(\ee)\times\GL_{i}(\ee)$,
qu'on écrit 
$(\delta_1\circ\det)\otimes(\delta_2\circ\det)$,
où $\delta_1$, $\delta_2$ sont des
caractères quadrati\-ques de $\ee^\times$.
Comme l'élément $t$ du lemme \ref{choixw}(3)
normalise $\bk$ et le sous-groupe $\BJ_\t\cap\G^\tau$,~il
normalise aussi $\chi$.
On en déduit que $\delta_1=\delta_2$,
puis que
la restriction de $\chi$ à~$\BJ^0\cap\G^\tau$
se~pro\-longe en un unique caractère $\qq$ de $\BJI^0$ trivial sur $\BJ^1$. 
Celui-ci est normalisé par $\BJI_\t$.~En le~pro\-longeant à $\BJI_\t$
en demandant que $\qq(\w)=\chi(\w t)\qq(t)^{-1}$, 
on obtient un caractère ayant la propriété voulue. 
\end{proof}

\subsection{}

Dans le cas où $\E/\E_0$ est ramifiée, on en déduit le résultat suivant. 

\begin{prop}
\label{exkappadistad}
Supposons que l'on soit dans l'un des cas suivants~:
\begin{enumerate}
\item 
$E/E_0$ est ramifiée et $c_0$ est pair,
\item
$E/E_0$ est ramifiée, $c_0$ est impair, $m$ est pair ou égal à $1$
et $\t$ est d'indice $i=\lfloor m/2\rfloor$.
\end{enumerate}
Alors toute représentation $\BJI_{\t}\cap\G^\tau$-distinguée de
$\BJI_{\t}$ prolongeant $\n$ est $\tau$-autoduale. 
\end{prop}

\begin{proof}
Soit $\bk$ une représentation $\BJI_{\t}\cap\G^\tau$-distinguée
de $\BJI_{\t}$ prolongeant $\n$.
D'après le lemme \ref{etakappasigmageneral},
il y a un unique caractère $\bx$ de $\BJI_{\t}$ trivial sur $\BJ^1$
tel que $\bk^{\vee\tau}$ soit isomorphe à~$\bk\bx$.
Comme $\bk$ est $\BJI_{\t}\cap\G^\tau$-distinguée,
la restriction de $\bx$ à $\BJ_\t\cap\G^\tau$ est triviale.
Il s'agit de prouver que $\bx$ est trivial sur $\BJI_{\t}$. 
Soit $\w\in\BJI_{\t}\cap\B^\times$ comme dans le paragraphe 
\ref{couragelesgars}. 

Commençons par le cas (1).
La restriction de $\bx$ à $\BJ^0$ s'identifie à un
caractère de $\GL_{m}(\ee)$~tri\-vial~sur $\GL_{m/2}(\ee_0)$,
où $\ee_0$ est l'extension quadratique de $\ee$ 
engendrée par $\v$~dans $\Mat_m(\ee)$.
Celui-ci se factorise sous la forme $\upsilon\circ\det$
où $\upsilon$ est un caractère de $\ee^\times$
trivial sur $\N_{\ee_0/\ee}(\ee_0^\times)=\ee^\times$,
\ie~que $\bx$ est trivial sur $\BJ^0$.
Il l'est aussi sur $\BJI_\t$,
car $\bx(\w)=1$
d'après le lemme \ref{choixw}.

Passons au cas (2).
Si $m=1$, la paire 
$(\BJ_\t,\bk)$ est~un type $\BJ_\t\cap\G^\tau$-distingué de $G$.
Son indui\-te~com\-pacte $\pi$ est donc une représentation cuspidale
$\G^\tau$-distinguée de $G$.
D'après le théorème \ref{GuoBM}, celle-ci est autoduale.
Par conséquent, les types $\bk$ et $\bk^{\vee\tau}\simeq\bk\bx$ sont
tous les deux contenus dans $\pi$,
ce dont on déduit que $\bx$ est trivial.
Si $m=2i$, la restriction de $\bx$ à $\BJ^0$ s'identifie à un
caractère de $\GL_{m}(\ee)$ tri\-vial sur
$\GL_{i}(\ee)\times\GL_{i}(\ee)$.
On en déduit que $\bx$~est 
tri\-vial sur $\BJ^0$.
Puis on a $\bx(\w t)=1$
(voir le lemme \ref{choixw})
et $\bx(t)=1$, donc $\bx(\w)=1$.  
\end{proof}

Nous en déduisons l'existence et l'unicité d'un prolongement remarquable
de $\n$ à $\BJI_{\t}$.

\begin{coro}
\label{exkappadistaddetp} 
Supposons que l'on soit dans l'un des cas considérés à la proposition
\ref{exkappadistad}.
\begin{enumerate}
\item
Le nombre de représentations~de $\BJI_{\t}$ prolongeant $\n$
étant à la fois $\tau$-autoduales et $\BJI_{\t}\cap\G^\tau$-distinguées est~: 
\begin{enumerate}
\item 
égal à $2$ si $E/E_0$ est ramifiée, $c_0$ est impair et $m=1$,
\item
égal à $1$ sinon.
\end{enumerate}
\item 
Il existe une unique représentation $\bk_{\t}$ de $\BJI_{\t}$ prolongeant $\n$
étant à la fois $\tau$-autoduale et $\BJI_{\t}\cap\G^\tau$-distinguée, 
et dont le déterminant soit d'ordre une puissance de $p$. 
\end{enumerate}
\end{coro}

\begin{proof}
Soit $\bk$ une représentation de $\BJI_{\t}$ prolongeant $\n$
qui soit à la fois $\tau$-autoduale et $\BJI_{\t}\cap\G^\tau$-distinguée.
Son~dé\-ter\-minant $\boldsymbol{\d}$
est un caractère $\tau$-autodual de $\BJI_{\t}$
trivial sur $\BJI_{\t}\cap\G^\tau$.
Le lemme~\ref{choixw} assure qu'il est d'ordre fini.
Le facteur $\boldsymbol{\omega}$ de $\boldsymbol{\d}$ 
d'ordre premier à $p$ est $\tau$-autodual 
et~$\BJI_{\t}\cap\G^\tau$-distingué.
Raisonnant comme dans la preuve du lemme 
\ref{bkpprimaire},~on~mon\-tre que,
tordant $\bk$ par $\boldsymbol{\omega}^{-a}$ pour un entier $a$ 
convenable,
on obtient une représentation $\bk_{\t}$ 
à la fois $\tau$-autoduale et $\BJI_{\t}\cap\G^\tau$-distinguée,
et dont le déterminant est d'ordre une puissance de $p$.~Une
représentation de $\BJ_\t$ prolongeant $\n$
à la fois $\tau$-autoduale et $\BJI_{\t}\cap\G^\tau$-distinguée
s'écrit donc sous la forme $\bk_{\t}\bx$ pour un caractère $\bx$ 
de $\BJI_{\t}$ trivial sur $(\BJI_{\t}\cap\G^\tau)\BJ^1$ et $\tau$-autodual.

Nous allons voir que $\bx$ est~trivial sur $\BJ^0$.
Identifions la restriction de $\bx$ à $\BJ^0$ à un
caractère de $\GL_m(\ee)$~tri\-vial~sur~$\GL_m(\ee)^\tau$,
se factorisant sous la forme $\upsilon\circ\det$
où $\upsilon$ est un caractère de $\ee^\times$.
\begin{itemize}
\item 
Si $E/E_0$ est ramifiée et $c_0$ est pair,
$\upsilon\circ\det$ est trivial sur
$\GL_m(\ee)^\tau=\GL_{m/2}(\ee_0)$,
donc $\upsilon$ est trivial sur
$\N_{\ee_0/\ee}(\ee_0^\times)=\ee^\times$.
\item 
Si $E/E_0$ est ramifiée, $c_0$ est impair et $m=2i$,
alors 
$\upsilon\circ\det$ est trivial sur
le~sous-grou\-pe $\GL_m(\ee)^\tau=\GL_{i}(\ee)\times\GL_{i}(\ee)$, 
donc $\upsilon$ est trivial.
\item 
Si $E/E_0$ est ramifiée, $c_0$ est impair et $m=1$,
alors $\ee^{\times\tau}=\ee^\times$ 
donc $\upsilon$ est trivial.
\end{itemize}
Ensuite, d'après le lemme \ref{choixw}, on a~:
\begin{itemize}
\item 
$\bx (\w)=1$
si $E/E_0$ est ramifiée et $c_0$ est pair, 
\item
$\bx (\w t)=1$ si $E/E_0$ est ramifiée, $c_0$ est impair et $m=2i$, 
\item
$\bx (\w^2)=1$ si $E/E_0$ est ramifiée, $c_0$ est impair et $m=1$.
\end{itemize}
Dans les deux premiers cas, le caractère $\bx$ est trivial
(dans le second cas car $\bx (t)=1$ puisque $\bx$ est trivial sur $\BJ^0$).
Dans le troisième et dernier cas, $\bx$ est quadratique et
trivial sur $\BJ^0$~;
il y a exactement deux tels caractères,
dont un seul qui soit d'ordre une puissance de $p$.
\end{proof}

\subsection{}
\label{eb}

Dans ce paragraphe,
on s'intéresse au cas où $\E/\E_0$ est non ramifiée.
Si $\rho$ est une~re\-pré\-sen\-tation irréductible cuspidale de $\GL_m(\ee)$,
on note $s(\rho)$ l'ordre du stabili\-sa\-teur de la
classe~d'iso\-morphisme de $\rho$ sous l'action de $\Gal(\ee/\kk_E)$.
D'après \cite[Corollaire 3.9]{MSjl}, 
c'est un diviseur de $c$ premier à $m$. 

\begin{lemm}
\label{hamptons}
On suppose que $mc$ est impair. 
Soit un entier $s\>1$. 
Pour qu'il existe~une~re\-pré\-sentation cuspidale $\tau$-auto\-dua\-le $\rho$
de $\GL_m(\ee)$ telle $s(\rho)=s$, il faut et suffit que
$s$ soit un~diviseur de~$c$ premier à $m$. 
\end{lemm}

\begin{proof}
Comme $\Gal(\ee/\kk_E)$ est d'ordre $c$,
le fait que $s$ divise $c$ est certainement une condition nécessaire.
Supposons donc désormais que $s$ divise $c$. 
Soit $\tt$ une extension de $\ee$ de~de\-gré $m$ dans $\Mat_m(\ee)$,
et considérons $\tt^\times$ comme un tore maximal de $\GL_m(\ee)$.
Soit $\xi$ un caractère $\ee$-régulier de $\boldsymbol{t}^\times$
(\ie que ses conjugués sous $\Gal(\tt/\ee)$ sont distincts deux à deux)
et soit $\rho$ la représentation cuspidale de $\GL_m(\ee)$ qui lui
correspond au sens de \cite{Green}, \ie que~:
\begin{equation}
\label{GreenParam}
\tr\ \rho (x) = (-1)^{m-1} \cdot \sum\limits_{\g} \xi^{\g}(x)
\end{equation}
pour tout $x\in\tt^\times$ dont le polynôme caractéristique est
irréductible sur $\ee$,
la somme étant prise~sur $\Gal(\tt/\ee)$.  
Notons $q_0$ le cardinal du corps résiduel de $E_0$. 
Alors \cite[Lemma 2.3]{VSANT19} (et plus précisément sa preuve)
assure que $\rho$ est $\tau$-autoduale si et seulement si
l'ordre de $\xi$ divise $q_0^{mc}+1$.

\'Ecrivons $c=bs$ et fixons un caractère $\xi$ de $\boldsymbol{t}^\times$
d'ordre $q_0^{mb}+1$, ce qui est possible car $q_0^{mb}+1$ divise l'ordre de 
$\tt^\times$.
L'entier $c$ (donc $s$) étant impair,
$q_0^{mb}+1$ divise $q_0^{mc}+1$. 
L'ordre de $q_0^2$ dans $(\ZZ/(q_0^{mb}+1)\ZZ)^\times$ 
est égal à $mb$, donc 
celui de $q_0^{2c}$ est égal à $mb/(mb,c)=m$
car $s$ est premier à $m$.
Par conséquent, le caractère $\xi$ est $\ee$-régulier,
la représentation cuspidale $\rho$ qui lui correspond est $\tau$-autoduale
et~$s(\rho)$ est égal à $s$.
\end{proof}

\begin{prop}
\label{flossdenaley}
On suppose que $mc$ est impair.
Soit $b$ un diviseur de $c$ tel que $c/b$ soit premier à $m$,
et soit $\BJ_{b}$ le sous-groupe de $\BJI_{\t}$ engendré par $\w^b$
et $\BJ^0$. 
Il existe une représentation de $\BJ_{b}$ prolon\-geant $\n$ qui est 
à la fois $\tau$-autoduale et $\BJ_{b}\cap\G^\tau$-distinguée. 
\end{prop}

\begin{proof}
Nous allons suivre l'argument de la preuve de \cite[Proposition 9.4]{VSANT19}.
Soit~une représentation $\tau$-autoduale $\bk$~de $\BJI_{\t}$ prolongeant 
$\n$,
{dont l'existence est assurée par la proposition \ref{kappatauautoduale}.}
D'après le lemme \ref{hamptons},
il existe une représentation cuspidale $\tau$-autoduale $\rho$
de $\GL_m(\ee)$~telle $s(\rho)=c/b$.
D'après \cite{Gow}, la représentation $\rho$ est $\GL_m(\ee_0)$-distinguée,
où $\ee_0$ est le sous-corps de $\ee$ invariant par son automorphisme
d'ordre $2$.
Notant encore $\rho$ son~infla\-tion à $\BJ^0$,
son normalisateur dans $\BJI_{\t}$ est égal à $\BJ_{b}$
car l'action de $\w$ par conjugaison induit sur $\ee$
l'action d'un générateur de $\Gal(\ee/\kk_E)$
(voir la remarque \ref{malinlechien}).

\begin{lemm}
\label{hamptons2}
Il existe une représentation de $\BJ_{b}$ prolongeant $\rho$ qui est 
à la fois $\tau$-autoduale et $\BJ_{b}\cap\G^\tau$-distinguée.
\end{lemm}

\begin{proof}
Soit $\bt$ une représentation de $\BJ_{b}$ prolongeant $\rho$.
Tout autre prolongement~de $\rho$ à $\BJ_{b}$ s'obtient en tordant $\bt$
par un caractère non ramifié de $\BJ_b$.
Quitte à tordre par un~caractère non ramifié convenable,
on peut donc supposer que $\bt$ est $\tau$-autoduale,
ce que nous ferons.
{Selon \cite[Theorem 3.6]{Gow}, l'espace~:
\begin{equation}
\label{toposannele}
\Hom_{\GL_m(\ee_0)}(\rho,\FC) = \Hom_{\BJ^0\cap\G^\tau}(\bt,\FC)
\end{equation}
est de dimension $1$.
Faisant agir le groupe $\BJ_{b}\cap\G^\tau$ sur cet espace,
on en déduit un unique~ca\-rac\-tère $\zeta$ de $\BJ_{b}\cap\G^\tau$
trivial sur $\BJ^0\cap\G^\tau$ tel que 
celui-ci soit égal à $\Hom_{\BJ_{b}\cap\G^\tau}(\bt,\zeta)$. 
Par uni\-ci\-té de $\zeta$ et comme $\bt$ est $\tau$-autoduale,
on a $\zeta^2=1$.}
Soit $\bt^*$ le prolongement $\tau$-auto\-dual de $\rho$~ob\-te\-nu en tordant 
$\bt$ par le caractère non ramifié $\boldsymbol{\varepsilon}$
de $\BJ_b$ d'ordre $2$.
Alors l'une~des~repré\-sen\-ta\-tions
$\bt$ et $\bt^*$ est $\BJ_{b}\cap\G^\tau$-distinguée. 
\end{proof}

\begin{rema}
Plus précisément,
une seule des deux repré\-sen\-ta\-tions $\bt$ et $\bt^*$ est 
$\BJ_{b}\cap\G^\tau$-distinguée, car l'autre,
qui est $\boldsymbol{\varepsilon}$-distinguée, 
ne peut pas être distinguée.
En effet,
si c'était le cas,
on aurait deux formes linéaires non nulles sur l'espace de $\rho$,
non colinéaires car l'une est $\BJ_{b}\cap\G^\tau$-invariante et l'autre
$\boldsymbol{\varepsilon}$-équivariante.
Elles sont pourtant toutes les deux $\GL_m(\ee_0)$-invariantes,
ce qui contredit le fait que \eqref{toposannele} est de dimension $1$.
\end{rema}

Soit $\bt$ une représentation comme au lemme \ref{hamptons2},
et soit $\bl=\bk\otimes\bt$.
La paire $(\BJ_{b},\bl)$ est un type $\tau$-autodual.
Soit $\chi$ le caractère de $\BJI_{\t}\cap\G^\tau$ associé à $\bk$
par le lemme \ref{kappatau}.
On a alors~:
\begin{equation*}
\Hom_{\BJ_{b}\cap\G^\tau}(\bl,\chi^{-1}) \neq \{0\}.
\end{equation*}
Comme dans \cite[Lemma 9.2]{VSANT19}, on prouve que 
tout caractère de $\BJ_{b}\cap\G^\tau$ trivial sur $\BJ^{1}\cap\G^\tau$~se
pro\-longe en un caractère de $\BJ_{b}$ trivial sur $\BJ^{1}$.
Soit donc $\qq$ un caractère de $\BJ_{b}$ trivial sur $\BJ^{1}$
{prolongeant $\chi^{-1}$,}
et posons $\bl'=\bl\qq^{-1}$.
La paire $(\BJ_{b},\bl')$ est un type $\BJ_{b}\cap\G^\tau$-distingué. 
Son induite compacte~à $\G$, notée $\pi'$,
est une représentation irréductible cuspidale $\G^\tau$-distinguée.
D'après le théorème~\ref{GuoBM}, elle est autoduale.
Elle contient donc à la fois $\bl'$ et
$\bl'^{\vee\tau}\simeq\bl'\qq(\qq\circ\tau)$,
ce dont on déduit~que~le caractère 
$\qq(\qq\circ\tau)$ est trivial sur $\BJ_{b}$.
Il s'ensuit que la représentation $\bk'=\bk\qq^{-1}$ prolonge $\n$,~et est
à la fois $\tau$-autoduale et $\BJ_{b}\cap\G^\tau$-distinguée.
\end{proof}

Le corollaire suivant fait pendant au corollaire \ref{exkappadistaddetp}.

\begin{coro}
\label{exkappadistaddetpnr} 
Il existe une unique représentation $\bk_{b}$ de $\BJI_{b}$ prolongeant $\n$
étant à la fois $\tau$-autoduale et $\BJI_{b}\cap\G^\tau$-distinguée,
et dont le déterminant soit d'ordre une puissance de $p$.
\end{coro}

\begin{proof}
Pour prouver l'existence de $\bk_{b}$,
on raisonne exactement comme au début de la preuve du corollaire
\ref{exkappadistaddetp}.
Prouvons maintenant l'unicité.
Si une représentation $\bk$ de $\BJ_b$ a les mêmes propriétés que $\bk_{b}$,
alors elle est de la forme $\bk_{b}\bx$ où $\bx$ est un caractère
de $\BJI_{b}$ trivial sur $(\BJI_{b}\cap\G^\tau)\BJ^1$,
d'ordre une puissance de $p$ et $\tau$-autodual.
{Un caractère du groupe $\BJ^0/\BJ^1\simeq\GL_m(\ee)$
se factorise en un caractère de $\ee^\times$,
qui est d'ordre premier à $p$.}
Le caractère $\bx$ est donc trivial sur $\BJ^0$.
Enfin,
d'après le lemme \ref{choixw},
on a $\bx (\w^b)=1$,
ce qui implique que 
$\bx$ est trivial.
\end{proof}

\subsection{}
\label{viteB}

Nous discutons maintenant l'existence de types $\tau$-autoduaux
dans une représentation~cus\-pi\-dale autoduale de $\G$ de niveau non nul.
Le cas du niveau $0$ sera traité à la section 
\ref{theriverofnoreturn}.~La~pro\-po\-sition suivante généralise le corollaire \ref{MAIN1}.

\begin{prop}
\label{MAIN1prop}
Une représentation cuspidale autoduale de $\G$
de niveau non nul 
contient un caractère simple $\tau$-autodual si et seulement si
elle contient un type $\tau$-autodual.
\end{prop}

\begin{proof}
L'une de ces implications est immédiate, 
car le caractère simple attaché à un type $\tau$-autodual est 
$\tau$-autodual.
Soit maintenant $\pi$ une représentation cuspidale autoduale~de~$\G$ 
contenant un caractère simple $\tau$-autodual $\t\in\Cc(\aa,\b)$.
D'après le lemme \ref{gandalf},
on peut~supposer que la strate simple $[\aa,\b]$ est $\tau$-autoduale,
ce que nous ferons. 
Soit $(\BJ,\bl)$ un type contenu~dans~$\pi$
auquel $\t$ soit attaché.
Alors $\tau$~stabi\-lise le normalisateur $\BJI_{\t}$ de
$\t$, donc~aussi son sous-groupe~com\-pact maximal $\BJ^0$.
Le quotient de $\BJI_{\t}$ par $F^\times\BJ^0$ étant~cycli\-que,
il possède un unique sous-groupe d'indice donné,
donc $\tau$ nor\-ma\-lise $\BJ$.
Aussi~le ty\-pe $(\BJ,\bl^{\vee\tau})$, qui est con\-te\-nu dans $\pi$,
est conjugué à $(\BJ,\bl)$ par un $g\in\G$ normalisant $\BJ$.
Comme le~carac\-tè\-re~sim\-ple qui lui est attaché est $\t^{-1}\circ\tau=\t$,
on a $g\in \BJI_{\t}$.
Notons $b$ l'indice de $\BJ$ dans $\BJI_{\t}$,
et fixons un $\w\in \BJI_{\t}\cap\B^\times$ tel que
$\BJI_{\t}$ soit engendré par $\BJ^0$
et $\w$ comme 
dans le paragraphe \ref{couragelesgars}.
Par conséquent, $\BJ$ est engendré par $\w^b$ et $\BJ^0$.
On peut donc supposer que $g=\w^i$ pour un $i\in\{0,\dots,b-1\}$
tel que $\w^i\tau(\w^i)\in\BJ$,
\ie que $b$ divise $2i$.
Si $b$ est impair, alors $i=0$ et $(\BJ,\bl)$ est $\tau$-autodual.

Supposons maintenant que $b$ soit pair et posons $\w'=\w^{b/2}$.
Fixons une représentation $\bk$ de~$\BJ$ prolongeant $\n$ et notons
$\bt$ la représentation de $\BJ$ triviale sur $\BJ^1$ telle que
$\bl\simeq\bk\otimes\bt$.
Quoique ce ne soit pas indispensable dans ce qui suit,
nous supposerons que $\bk$ est $\tau$-autoduale, 
comme nous y autorise la proposition 
\ref{kappatauautoduale}. 
Sup\-po\-sons que $(\BJ,\bl)$ ne soit pas
$\tau$-autodual.
Alors~la~repré\-sen\-tation
$\bl^{\vee\tau}$ est isomorphe à $\bl^{\w'}$ et,
comme $\w'$ normalise $\bk$,
on en déduit que $\bt^{\vee\tau}$ est isomorphe à $\bt^{\w'}$.
D'autre part, le fait que $b$ soit pair implique que $c$ l'est aussi.
On est donc dans le cas~où $c=c_0/2$ et $m=2m_0$,
donc $E/E_0$ est ramifiée et $m$ est pair,
avec les notations des paragraphes \ref{Biotabstait} et \ref{cestlong}.

Identifions~main\-te\-nant 
$\BJ^0/\BJ^1\simeq\bb^\times/\BU^1(\bb)$ et $\GL_m(\ee)$
de sorte que $\tau$ agisse sur $\GL_m(\ee)$ de la façon décrite
à la proposition \ref{indiceduntype},
\ie par conjugaison par un élément $\v\in\GL_m(\ee)$ 
tel que $\v^2\in\ee^{\times}$ et $\v^2\notin\ee^{\times2}$.
Soit~$\rho$~la~re\-présentation cuspidale de
$\GL_m(\ee)$ définie par $\bt$.
La conjugaison par $\w$ agit sur $\GL_m(\ee)$ 
comme un générateur $\g$ de $\Gal(\ee/\kk_E)$
dépendant de l'invariant de Hasse de $C$.
L'élément $\w'$ agit~donc comme l'élément $\g'=\g^{b/2}$.
Compte tenu de la définition de $b$ à la remarque \ref{Halo},
la~repré\-sen\-ta\-tion $\rho^{\g'}$~est~isomor\-phe à $\overline{\rho}$,
le conjugué de $\rho$ par
l'unique élément de $\Gal(\ee/\kk_E)$ d'ordre $2$.
La relation $\bt^{\tau\vee}\simeq\bt^{\w'}$ entraîne donc
$\rho^{\vee}\simeq \overline{\rho}$.
D'après \cite[Lemma 2.3]{VSANT19}, 
on en déduit que $m$ est impair~:~con\-tra\-dic\-tion.
Ainsi le type $(\BJ,\bl)$ est $\tau$-autodual.
\end{proof}

Il est intéressant de noter que ce résultat est faux en niveau $0$~:
nous verrons dans la section \ref{theriverofnoreturn}
qu'une représentation cuspidale autoduale de niveau $0$ de $G$ peut contenir
un caractère simple $\tau$-autodual sans contenir de type $\tau$-autodual
(voir le paragraphe \ref{findumonde}).

\subsection{}
\label{onnoubliepas2}

La situation est la même qu'au paragraphe \ref{couragelesgars}~:
on a fixé un caractère simple maximal~$\tau$-au\-todual $\t$ et une
strate simple $\tau$-autoduale $[\aa,\b]$ telle que
$\t\in\Cc(\aa,\b)$.
Supposons en outre qu'on soit dans l'un des cas suivants~:
\begin{itemize}
\item 
$E/E_0$ est {non ramifiée} et $mc$ est impair,
\item
$E/E_0$ est ramifiée, $c_0$ est impair, $m$ est pair ou égal à $1$,
et $\t$ est d'indice $i=\lfloor m/2\rfloor$,
\item
$E/E_0$ est ramifiée et $c_0$ est pair. 
\end{itemize}
Notons $\Om=\Om_{\t}$ l'ensemble des paires $(\BJ,\bt)$ tels que~:
\begin{enumerate}
\item
$\F^\times\BJ^{0}\subseteq\BJ\subseteq\BJI_{\t}$
et $\bt$ est une classe d'isomorphisme de
représentation irréductible $\tau$-auto\-duale
de $\BJ$ triviale sur $\BJ^1$, 
\item
le normalisateur de $\bt$ dans $\BJI_{\t}$ est égal à $\BJ$,
\item
la restriction de $\bt$ à $\BJ^0$ est l'inflation d'une représentation 
cuspidale de $\GL_m(\ee)$.
\end{enumerate}
Cela signifie donc que, si $\bk$ est une représentation de $\BJI_\t$
prolongeant $\n$, alors $(\BJ,\bk\otimes\bt)$ est~un type de $G$.

\'Etant donné $(\BJ,\bt)\in\Om$,
notons $b$ l'indice de $\BJ$ dans $\BJI_\t$.
D'après les remarques \ref{malinlechien} et \ref{marecage}, 
c'est un diviseur de $c$ et $c/b$ est premier à $m$.

\begin{lemm}
\label{gersonide}
L'ensemble $\Om$ est fini,
et il existe un sous-ensemble $\Om^+\subseteq\Om$ tel que~: 
\begin{enumerate}
\item 
le cardinal de $\Om$ est le double de celui de $\Om^+$,
\item
pour tout $(\BJ,\bt)\in\Om^+$,
la représentation $\bt$ est $\BJ\cap\G^\tau$-distinguée.
\end{enumerate}
\end{lemm}

\begin{proof}
Pour prouver que $\Om$ est fini,
il suffit de prouver qu'une représentation cus\-pi\-dale $\tau$-autoduale
$\rho$ de $\GL_m(\ee)$ n'admet qu'un nombre fini de prolongements
$\tau$-autoduaux à son normalisateur $\BJ$ dans $\BJI_{\t}$.
Or, si $\bt$ est un tel prolongement,
les autres sont de la forme $\bt\boldsymbol{\chi}$
où $\boldsymbol{\chi}$ est un caractère quadratique de $\BJ$ trivial sur $\BJ^0$.

Soit $(\BJ,\bt)$ une paire de $\Om$
et soit $\rho$ la représentation cuspidale $\tau$-autoduale
de $\GL_m(\ee)$ qu'elle définit. 
D'après \cite{Gow} si $E/E_0$ est non ramifiée,
\cite[Proposition 6.1]{HMIMRN} 
si $E/E_0$~est ramifiée et $c_0$~est impair,
et enfin \cite[Lemme 4.3.11]{Coniglio} 
si $E/E_0$ est ramifiée et $c_0$ est pair,
$\rho$ est $\GL_m(\ee)^\tau$-distinguée et 
$\Hom_{\GL_m(\ee)^\tau}(\rho,\FC)$ est de dimension~$1$.
Il existe donc un unique caractère $\zeta$ de $\BJ\cap\G^\tau$ trivial sur
$\BJ^0\cap\G^\tau$ tel que~:
\begin{equation*}
\Hom_{\GL_m(\ee)^\tau}(\rho,\FC) = \Hom_{\BJ\cap\G^\tau}(\bt,\zeta)
\end{equation*}
et, par unicité de $\zeta$ et comme $\bt$ est $\tau$-autoduale, on a $\zeta^2=1$.
{Le quotient $\BJ/\BJ^0$ étant isomorphe à $\ZZ$,
le groupe $\BJ$ a un unique caractère d'ordre $2$ trivial sur $\BJ^0$,
que l'on note $\boldsymbol{\omega}$.}
Posons ${\bt}^*=\bt\boldsymbol{\omega}$. 
Alors $(\BJ,\bt^*)\in\Om$ et~: 
\begin{equation*}
\Hom_{\BJ\cap\G^\tau}(\bt,\zeta)
= \Hom_{\BJ\cap\G^\tau}(\bt^*,\zeta\boldsymbol{\omega}). 
\end{equation*}
Supposons que $E/E_0$ soit non ramifiée,
ou que $c_0$ soit pair, ou que $m>1$.
Le groupe $\BJI_{\t}\cap\G^\tau$ est en\-gendré par $\BJ^0\cap\G^\tau$ et par un
élément $\w'\in\w\BJ^0$
défini par le lemme \ref{choixw}, 
donc $\BJ\cap\G^\tau$ est engendré par $\BJ^0\cap\G^\tau$ et $\w'^b$.
Le caractère $\boldsymbol{\omega}$ est donc non trivial sur $\BJ\cap\G^\tau$,
car il est~non~tri\-vial~en $\w'^b\in\w^b\BJ^0$.
Par conséquent, exactement une représentation
parmi $\bt$ et $\bt^*$ est distinguée par
$\BJ\cap\G^\tau$.

Supposons enfin que $E/E_0$ soit ramifiée,
que $c_0$ soit impair et que $m=1$.
Le groupe $\BJI_{\t}\cap\G^\tau$ est en\-gendré par
$\w^2$ et $\BJ^0\cap\G^\tau$,
donc $\boldsymbol{\omega}$ est trivial sur $\BJ\cap\G^\tau$.
Soit $\boldsymbol{\upsilon}$ le caractère quadratique non trivial de
$\ee^\times$ prolongé à $\BJ$ en posant $\boldsymbol{\upsilon} (\w^b)=1$.
Alors $\Om$ est constitué des quatre caractères 
$1$, $\boldsymbol{\omega}$, $\boldsymbol{\upsilon}$ et
$\boldsymbol{\upsilon}\boldsymbol{\omega}$.
Ceux qui sont distingués par 
$\BJ\cap\G^\tau=\langle\w^{2b},\BJ^0\cap\G^\tau\rangle$ sont
$1$ et $\boldsymbol{\omega}$.
\end{proof}

\subsection{}
\label{malinlegarcon}

Dans ce paragraphe,
on note $\TT$ l'endo-classe du caractère simple maximal $\tau$-autodual $\t$ 
fixé au~paragraphe \ref{viteA}
et $T/T_0$~l'ex\-ten\-sion quadratique qui lui est associée.  
On suppose qu'il existe une représentation cuspidale~auto\-duale de $G$
d'endo-classe $\TT$ contenant $\t$.
D'après la propo\-si\-tion \ref{indiceduntype},
quitte à remplacer $\t$ par l'un de ses conjugués, 
on peut supposer que $\t$ est d'indice $\lfloor m/2\rfloor$
si $T/T_0$ est ramifiée et $c_0$ est impair.
D'après le corollaire \ref{fromagequipuecor}, on a~:
\begin{enumerate}
\item
si $T/T_0$ est non ramifiée, alors $mc$ est impair,
\item
si $T/T_0$ est ramifiée, alors $m$ est pair ou égal à $1$.
\end{enumerate}
Nous sommes donc dans le champ d'application des résultats des
paragraphes \ref{onnoubliepas} à \ref{onnoubliepas2}.

Soit une paire $(\BJ,\bt)\in\Om$.
Notons $b$ l'indice de $\BJ$ dans $\BJI_\t$. 
D'après les corollaires \ref{exkappadistaddetp} et \ref{exkappadistaddetpnr}, 
il y a une unique représentation $\bk_{*}$ de $\BJ_{b}$ prolongeant $\n$,
étant à la fois $\tau$-autoduale et distinguée par $\BJ_{b}\cap\G^\tau$,
et dont le déterminant soit d'ordre une puissance de $p$.  
On pose~:
\begin{equation*}
\Pi(\BJ,\bt) = \ind^\G_{\BJ} (\bk_{*}\otimes\bt).
\end{equation*}
La paire $(\BJ, \bk_{*}\otimes\bt)$ étant par construction un type
$\tau$-autodual de $G$, 
la représentation $\Pi(\BJ,\bt)$
est une représentation cuspidale autoduale 
de $G$ d'endo-classe $\TT$.
Notons maintenant~:
\begin{equation*}
\AA(G,\TT)
\end{equation*}
l'ensemble des classes d'isomorphisme
de représentations cuspidales auto\-duales d'endo-classe $\TT$ de $G$.

\begin{prop}
\label{galettedepain}
\begin{enumerate}
\item
L'application de $\Om$ dans $\AA(G,\TT)$ définie par~:
\begin{equation*}
\Pi:(\BJ,\bt)\mapsto\Pi(\BJ,\bt)
\end{equation*}
est surjective.
\item
Deux paires $(\BJ,\bt)$ et $(\BJ',\bt')$ de $\Om$
ont la même image par $\Pi$ si et seulement si 
$\BJ'=\BJ$ et les~re\-pré\-sentations
$\bt$, $\bt'$ sont conjuguées sous $\BJI_{\t}$. 
\item 
Si $\bt$ est $\BJ\cap\G^\tau$-distinguée,
alors $\Pi(\BJ,\bt)$ est $G^\tau$-distinguée.
\end{enumerate}
\end{prop}

\begin{proof}
Commençons par prouver la surjectivité.
Soit $\pi$ une représentation cuspidale autoduale de $G$ d'endo-classe $\TT$.
Elle contient le caractère~sim\-ple $\t$ d'après la proposition \ref{entrenous}. 
D'après la proposition \ref{MAIN1prop},
elle contient donc un type $\tau$-autodual $(\BJ,\bl)$, 
et la représentation $\bl$ se décompose sous la forme $\bk_{*}\otimes\bt$.
Comme $\bk_{*}$ est $\tau$-autoduale, $\bt$ l'est aussi.
Par conséquent, la paire $(\BJ,\bt)$ appartient à $\Om$ et $\pi$ est
isomorphe à $\Pi (\BJ,\bt)$.

Soient ensuite $(\BJ,\bt)$ et $(\BJ',\bt')$ deux paires de $\Om$
donnant lieu à la même représentation~cus\-pi\-dale autoduale $\pi$ de $G$.
Alors les types $(\BJ,\bk_{*}\otimes\bt)$ et $(\BJ',\bk_{*}\otimes\bt')$
sont conjugués
par un élément $g\in\G$ normalisant non seulement $\BJ$,
mais aussi le caractère simple $\t$ attaché à ces types.
On en déduit que $g\in\BJI_{\t}$.

Prouvons le point (3).
Si $\bt$ est $\BJ\cap\G^\tau$-distinguée,
alors $\bk_{*}\otimes\bt$ est $\BJ\cap\G^\tau$-distinguée
car $\bk_{*}$ est $\BJ\cap\G^\tau$-distinguée
(voir le lemme \ref{etatau}).
Par application de la formule de Mackey,
son induite compacte $\Pi(\BJ,\bt)$ est $G^\tau$-distinguée.
\end{proof}

\subsection{}
\label{malinlegarconcor}

Dans ce paragraphe,
$\TT$ est une endo-classe autoduale de niveau non nul
de degré divisant $2n$
(qui n'est plus définie par le caractère simple maximal $\tau$-autodual 
$\t$ fixé au~paragraphe \ref{viteA}).
On note $\AA^+(G,\TT)$ 
le sous-ensemble de $\AA(G,\TT)$
formé des représentations distinguées par $\G^\tau$.
Le résultat suivant met un terme à cette section.

\begin{prop}
\label{OlivierdeNoyen} 
\begin{enumerate}
\item 
L'ensemble $\AA(G,\TT)$ est fini. 
\item 
si $\AA^+(G,\TT)$ n'est pas vide,
alors~:
\begin{equation*}
|\AA^+(G,\TT)| \> \frac {1} {2} \cdot |\AA(G,\TT)|. 
\end{equation*}
\item
Pour que $\AA^+(G,\TT)$ soit non vide, 
il faut et suffit qu'une représentation cuspidale auto\-dua\-le de $G$
d'endo-classe $\TT$ contienne un caractère simple $\tau$-autodual. 
\end{enumerate}
\end{prop}

\begin{proof}
Il n'existe qu'un nombre fini de classes d'inertie de représentations 
cuspidales d'endo-classe fixée. 
Pour prouver la finitude de $\AA(G,\TT)$,
il suffit donc de prouver qu'il n'existe qu'un nombre fini de représentations 
cuspidales autoduales de classes d'inertie fixée. 
Or si $\pi$ est une représentation cuspidale autoduale, 
les représentations autoduales qui lui sont iner\-tiel\-le\-ment
équivalentes sont de la forme $\pi\chi$
où $\chi$ est un caractère non ramifié de $\G$ tel que $\pi\chi^2\simeq\pi$,
ce qui assure que $\chi$ est d'ordre fini divisant $4n$
(voir par exemple le début de \cite[2.7]{BHJL3}). 

Si $\AA^+(G,\TT)$ est non vide, 
la proposition \ref{grammage} assure qu'il existe une représentation cuspidale 
auto\-dua\-le de $G$ 
d'endo-classe $\TT$ contenant un caractère simple $\tau$-autodual.
Inversement,
s'il~y~a une représentation cuspidale auto\-dua\-le de $G$ 
d'endo-classe $\TT$ contenant un caractère simple $\t$
$\tau$-auto\-dual,
la proposition \ref{galettedepain} et le lemme \ref{gersonide} montrent
que~: 
\begin{equation*}
\{\Pi(\BJ,\bt)\ |\ (\BJ,\bt)\in\Om^+\} \subseteq \AA^+(\G,\TT).
\end{equation*}
En outre, notant $[\Om]$ l'ensemble des classes de
$\BJI_{\t}$-conju\-gaison de $\Om$
et $[\Om^+]$ l'analogue pour $\Om^+$,
le cardinal du membre de gauche est égal à celui de $[\Om^+]$.
Comme $\AA(G,\TT)$ et $[\Om]$ ont le même car\-dinal,
il ne reste qu'à prou\-ver que le cardinal de $[\Om^+]$ est au moins égal à
la moitié de celui de $[\Om]$, 
ce qui se déduit du lemme \ref{gersonide}(1). 
\end{proof}

\section{Le niveau $0$}
\label{theriverofnoreturn}

Nous traitons dans cette section le cas des représentations de niveau $0$. 
Notons $A$ la $F$-algèbre centrale simple $\Mat_r(D)$
et posons $G=A^\times$.
Comme au paragraphe \ref{cyprien},
on fixe un $\a\in\F^\times$
et un $\dd\in\G$ non central tel que $\dd^2=\a$,
définissant une involu\-tion $\tau$ de $\G$.
Si $\a$ est un carré dans $F^\times$, on se place dans le cadre de la
convention \ref{CONV2}.
Sinon, on pose $K=\F[\dd]$,
qui est une extension quadratique de $F$.

\subsection{}
\label{niveauzero1}

En niveau $0$,
un~carac\-tè\-re simple est le caractère trivial
d'un sous-groupe ouvert compact~de la forme $\BU^1(\aa)$
où $\aa$ est un ordre héréditaire de $A$,
et un tel carac\-tè\-re simple est maximal si~$\aa$ est~maximal.
Pour qu'il y ait un carac\-tè\-re simple maximal $\tau$-autodual,
il faut et suffit donc~qu'il y ait un
ordre maximal stable par $\tau$ dans $A$.

\begin{lemm}
\label{fredericmoreau}
Supposons que $\a$ ne soit pas un carré de $F^\times$. 
Un ordre maximal de $A$ est~sta\-ble par~$\tau$ si et seulement
s'il est normalisé par $\K^\times$.
\end{lemm}

\begin{proof} 
Un ordre de $A$ est stable par~$\tau$ si et seulement
s'il est normalisé par $\dd$.
Par conséquent, tout ordre normalisé par $\K^\times$ 
est stable par~$\tau$.
Inversement, soit $\aa$ un~or\-dre maximal stable par~$\tau$.
Montrons qu'il est normalisé par $\K^\times$. 
Pour cela, il est commode d'identifier $\aa$ à l'ordre maximal standard.
Son normalisateur dans $G$ est engendré par $\aa^\times$ et
une uniformisante $\w$ de $D$ telle que $\w^d\in\F$.
Aussi peut-on écrire $\dd$ sous la forme $\w^{k} u$
pour un unique $k\in\ZZ$ et un unique $u\in\aa^\times$.
Soit maintenant $\l=x+y\dd\in\K^\times$
avec $x,y\in\F$ qu'on peut supposer non nuls. 
Montrons que $\l$ normalise $\aa$.

{Notons
${\rm val}_D$ la valuation sur $\D^\times$
prenant la valeur $1$ en n'importe quelle uniformisante~de $\D$.}

Si ${\rm val}_D(x)\neq {\rm val}_D(y)+k$,
on peut~se~ra\-me\-ner,
en divisant par $x$ ou par $y\dd$, 
au cas où
$\l$~ap\-par\-tient à $1+\w\aa\subseteq\aa^\times$.
Sinon, $k$ est un multiple de $d$
donc $\w^k\in\F^\times$.
On peut alors se ramener au cas où
$k={\rm val}_F(x)={\rm val}_F(y)=0$,
donc $\l=x+yu\in\aa$.
Comme $u^2=\a$ est une unité de $F$
qui n'est pas un carré de $F$,
l'identité $(x+yu)(x-yu)=x^2-\a y^2\in\Oo_F^\times$
assure que $\l\in\aa^\times$.
\end{proof}

\begin{prop}
\label{baff}
Pour qu'il y ait un ordre maximal stable par $\tau$ dans $A$,
il faut et suffit que l'une des conditions suivantes soit satisfaite~:
\begin{enumerate}
\item 
ou bien $\a\in\F^{\times2}$,
\item
ou bien $\a\notin\F^{\times2}$ et $e_{K/F}$ divise $d$.
\end{enumerate}
\end{prop}

\begin{proof}
{Si $\a\in\F^{\times2}$, 
on peut supposer que $\dd$ est diagonal et que ses éléments diagonaux
sont égaux à $1$ ou à $-1$,
auquel cas il nor\-malise l'ordre maximal standard $\Mat_r(\Oo_D)$.}

{Par ailleurs,
pour qu'il y ait un ordre principal de période $e$ de $A$ 
normalisé par $\K^\times$,
il faut~et suffit que $e_{K/F}$ divise $ed$,
et un ordre principal est maximal si et seulement si sa période~vaut~$1$. 
Si~$\a\notin\F^{\times2}$,
on déduit du lemme \ref{fredericmoreau} qu'il y a
un ordre maximal stable par $\tau$ dans $A$ si et~seu\-lement~si
$e_{K/F}$ divise $d$.}
\end{proof}

Une représentation irréductible de $G$ est de niveau $0$ si elle a un
vecteur non nul inva\-riant par un sous-groupe 
de~la forme $\BU^1(\aa)$,
où $\aa$ est un ordre maximal.
Les ordres maximaux de $A$ étant tous conjugués sous $G$,
une représentation de niveau $0$
contient un caractère simple $\tau$-autodual si et seulement
s'il existe un ordre maximal dans $A$ stable par $\tau$.
On en déduit immédiatement le corollaire suivant. 

\begin{coro}
\label{baffcor}
Soit $\pi$ une représentation cuspidale autoduale de niveau $0$ de $G$.
Pour~que~$\pi$ contienne un caractère simple $\tau$-autodual,
il faut et suffit que $\a\in\F^{\times2}$ ou que $e_{K/F}$ divise $d$.
\end{coro}

\begin{rema}
\label{balmasquecor0}
La première partie du corollaire \ref{tolkienelrond}
s'étend aux~re\-présentations cuspidales autoduales de niveau $0$. 
\end{rema}

\subsection{}

Classons maintenant,
à conjugai\-son près par $\G^\tau$, 
les ordres maximaux de $A$ stables par $\tau$.
On suppose qu'il en existe un,
\ie que les conditions de la proposition \ref{baff} sont vérifiées. 

\begin{prop}
\label{indiceduntype0}
\begin{enumerate}
\item 
Le nombre de classes de $\G^\tau$-conjugaison
d'ordres maximaux de $A$ stables par $\tau$ est~:
\begin{enumerate}
\item 
égal à $1$ si $K/F$ est ramifiée,
ou si $K/F$ est non ramifiée et $d$ est impair,
\item
égal à $\lfloor r/2\rfloor+1$ sinon. 
\end{enumerate}
\item
Si $\aa$ est un tel ordre,
il existe un $F$-automorphisme intérieur de $A$ identifiant $\aa$
à l'ordre maximal standard $\Mat_m(\Oo_D)$
et l'action de $\tau$~sur $\aa^\times/\BU^1(\aa)\simeq\GL_r(\kk_D)$ à~:
\begin{enumerate}
\item
l'action de l'élément d'ordre $2$ de $\Gal(\kk_D/\kk_F)$
si $K/F$ est ramifiée, 
\item
l'ac\-tion
par conjugaison d'un $\v\in\GL_r(\kk_D)$ 
tel que $\v^2\in\kk_D^\times$ et $\v^2\notin\kk_D^{\times2}$
si $K/F$ est non ramifiée et $d$ est impair, 
\item 
l'action par conjugaison de~:
\begin{equation}
\label{sigmainiveau0}
\s_i = {\rm diag}(-1,\dots,-1,1,\dots,1) \in \GL_r(\kk_D)
\end{equation}
où $-1$ apparaît avec multiplicité $i$,
pour un unique entier $i\in\{0,\dots, \lfloor r/2\rfloor\}$, 
si $K/F$ est non ramifiée et $d$ est pair,
ou si $\a\in\F^{\times 2}$.
\end{enumerate}
\end{enumerate}
\end{prop}

\begin{proof}
Par hypothèse, $\a$ est un carré de $\F^{\times}$ ou $e_{K/F}$ divise $d$.
Notons $\aa$ l'ordre~ma\-xi\-mal standard $\Mat_{r}(\Oo_D)$.

Supposons d'abord que $K/F$ soit ramifiée.
Comme $d$ est pair dans ce cas, 
on peut~supposer,
grâce au théorème de Skolem-Noether et 
quitte à remplacer $\dd$ par un de ses con\-jugués sous $G$,
ce qui ne change pas le résultat à~prouver, 
que $K\subseteq\D$.
On peut aussi supposer que $\a$ est une uniformisante de $F$,
de sorte que $\dd$ soit une unifor\-mi\-sante de $K$.
Dans ces conditions, l'ordre $\aa$ est~sta\-ble~par $\tau$,
qui induit sur $\aa^\times/\BU^1(\aa)\simeq\GL_r(\kk_D)$ 
l'unique $\kk_F$-automorphisme de~$\kk_D$~d'or\-dre $2$.
Les
autres ordres maximaux sont con\-ju\-gués à $\aa$ sous $G$.
\'Etant donné un $y\in\G$,
l'ordre $\aa^y$ est sta\-ble~par $\tau$ si et seulement si 
$w=\tau(y)y^{-1}\in\aa^\times$.
Raisonnant comme dans la preuve de~la pro\-po\-si\-tion \ref{indiceduntype},~on
prouve qu'il y a un $x\in\aa^\times$ tel que $w=\tau(x)x^{-1}$,
donc que $\aa^y$ est conjugué à $\aa$ sous $\G^\tau$.

Supposons maintenant $K/F$ non ramifiée et $d$ impair.
Comme $r$ est pair dans ce cas, 
on~peut écrire $r=2k$, identifier $A$ à $\Mat_k(\Mat_2(D))$
et supposer que $K\subseteq\Mat_2(F)$.
On peut aussi supposer que $\a\in\bm_F$ et $\dd\in\Mat_2(\Oo_F)$.
Dans ces conditions, $\aa$ est~sta\-ble par $\tau$,
qui induit sur $\GL_r(\kk_D)$
un automorphisme de conjugaison par un
élément $\v\in\GL_r(\kk_D)$ tel que $\v^2=\a\in\kk_D^\times$,
où l'on note encore $\a$ l'image de $\a$ dans $\kk_F^\times$.
En outre,
comme $d$ est impair et $\a\notin\kk_F^{\times 2}$,
on a $\a\notin\kk_D^{\times2}$.
On termine la preuve du cas (2.b) comme celle du cas précédent. 

Supposons $K/F$ non ramifiée et $d$ pair.
Comme dans le cas (2.a),
on suppose que $K$ est inclus dans $\D$.
On peut aussi supposer que $\a\in\bm_F$ et $\dd\in\bm_K$.
Dans ces conditions,
$\aa$ est sta\-ble par~$\tau$,
qui induit sur $\GL_r(\kk_D)$ l'automorphisme trivial.
Les autres ordres maximaux stables par $\tau$ sont de la forme
$\aa^y$ où $y$ est un élément de $G$ tel que 
$w=\tau(y)y^{-1}\in\aa^\times$.
Raisonnant comme dans~la preuve de~la proposition \ref{indiceduntype},~on
prouve qu'il y a un $x\in\aa^\times$ et un entier $i\in\{0,\dots,r\}$
tels que~:
\begin{equation*}
w=\tau(xt_i)(xt_i)^{-1},
\quad
t_i = {\rm diag}(\w,\dots,\w,1,\dots,1)\in\GL_r(D),
\end{equation*}
où $\w$ est une uniformisante de $D$ telle que $\tau(\w)=-\w$,
apparaissant $i$ fois.
Il s'ensuit que $\aa^y$ est conjugué à $\aa^{t_i}$ sous $\G^\tau$,
et on termine comme dans~la preuve de~la proposition 
\ref{indiceduntype}. 

Supposons enfin que $\a\in\F^{\times 2}$.
Dans ce cas, $r$ est pair, qu'on écrit $r=2k$,
et on peut supposer que $\dd={\rm diag}(-1,\dots,-1,1,\dots,1)$. 
Dans ces conditions, l'ordre $\aa$ est sta\-ble par~$\tau$,
qui induit sur $\GL_r(\kk_D)$ l'automorphisme de conjugaison par
l'élément $\s_k$ défini par \eqref{sigmainiveau0}.
On termine la preuve du cas (2.c) comme celle du cas précédent. 
\end{proof}

\subsection{}
\label{enterrement83}

On suppose toujours qu'il existe un ordre maximal stable par $\tau$,
\ie que les condi\-tions de la proposition \ref{baff} sont vérifiées. 

Si $\aa$ est un ordre maximal de $A$, nous noterons~$\BJI_{\aa}$
le normalisateur de $\aa$ dans $G$.
Rappelons (voir la remarque \ref{pifoulechien00})
qu'une paire $(\BJ,\bt)$ est un type de niveau $0$ de $G$ s'il existe un
ordre maximal $\aa$ de $A$ tel que~:
\begin{enumerate}
\item
$\F^\times\aa^\times\subseteq\BJ\subseteq\BJI_{\aa}$
et $\bt$ est une
représentation irréductible de $\BJ$ triviale sur $\BU^1(\aa)$, 
\item
le normalisateur de $\bt$ dans $\BJI_{\aa}$ est égal à $\BJ$,
\item
la restriction de $\bt$ à $\aa^\times$ est l'inflation d'une représentation 
cuspidale de $\aa^\times/\BU^1(\aa)$.  
\end{enumerate}
L'induite compacte à $G$ d'un type de niveau $0$ est une
représentation~irréduc\-tible cus\-pidale de niveau $0$ de $G$,
et toutes s'obtiennent ainsi.
Par cohérence avec le cas du niveau non nul, nous noterons
$\BJ^0=\aa^\times$ et $\BJ^{1}=\BU^1(\aa)$.
De \cite[Proposition 5.20]{HMIMRP}, on tire le résultat suivant. 

\begin{lemm}
\label{HMwarning}
Soit $\pi$ une représentation cuspidale de niveau $0$ de $G$. 
Pour que $\pi$ soit dis\-tin\-guée par $G^\tau$,
il faut et suffit
qu'elle contienne un type $(\BJ,\bt)$ de niveau $0$ tel que 
$\BJ$ soit stable~par~$\tau$ et $\bt$ soit 
$\BJ\cap\G^\tau$-dis\-tin\-guée.
\end{lemm}

On en déduit la condition nécessaire suivante de distinction.
Soit $(\BJ,\bt)$ un type de niveau $0$ de $G$ tel que $\BJ$ soit stable par 
$\tau$. 
Soit $\aa$ l'ordre maximal de $A$ tel que $\BJ\subseteq\BJI_\aa$. 
Il est stable par $\tau$.
Quitte~à conjuguer par un élément de $\G$,
on peut supposer qu'on est dans l'un des cas décrits par la~pro\-position
\ref{indiceduntype0}.
En particulier, 
si $K/F$ est non ramifiée et $d$ est pair, ou si $\a\in\F^{\times 2}$, 
il y a un entier $i\in\{0,\dots, \lfloor r/2\rfloor\}$
défini par la pro\-position \ref{indiceduntype0}(2.c).

\begin{prop}
\label{HMwarningprop}
\begin{enumerate}
\item 
Soit $(\BJ,\bt)$ un type de niveau $0$ de $G$ tel que 
$\BJ$ soit stable par $\tau$ et la représentation $\bt$ soit 
dis\-tin\-guée par $\BJ\cap\G^\tau$.
Alors~:
\begin{enumerate}
\item 
la représentation $\bt$ est $\tau$-autoduale,
\item
si $K/F$ est non ramifiée et $d$ est pair, ou si $\a\in\F^{\times 2}$,
l'entier $i$ est égal à $\lfloor r/2\rfloor$.
\end{enumerate}
\item 
Soit $(\BJ,\bt)$ un type $\tau$-autodual de niveau $0$ de $G$. 
Alors~: 
\begin{enumerate}
\item 
si $K/F$ est ramifiée,
l'entier $r$ est impair,
\item
si $K/F$ est non ramifiée et $d$ est impair,
l'entier $r$ est pair,
\item
si $K/F$ est non ramifiée et $d$ est pair, ou si $\a\in\F^{\times 2}$,
l'entier $r$ est pair ou égal à $1$.
\end{enumerate}
\end{enumerate}
\end{prop}

\begin{proof}
On identifie $\aa$ à l'ordre maximal standard et 
$\aa^\times/\BU^1(\aa)$ au groupe $\GL_r(\kk_D)$, noté $\GG$.
On note $\rho$ la repré\-sen\-ta\-tion cuspidale de
$\GL_r(\kk_D)$ définie par $\bt$.
Dans chacun des cas (a), (b) et (c) de l'item~2,
l'action de $\tau$ sur $\GL_r(\kk_D)$ est décrite par la proposition
\ref{indiceduntype0}. 

Dans le cas (2.a),
l'entier $d$ est pair et $\GG^\tau$ est de la forme
$\GL_r(\kk)$ où $\kk$~est l'unique sous-corps de $\kk_D$ sur lequel 
celui-ci soit de degré $2$.
D'après \cite{Gow}, 
la~re\-pré\-sentation $\rho$ est distinguée par $\GG^\tau$
si et seulement si $\rho^\vee$ est isomorphe au
conjugué $\overline{\rho}$
de $\rho$ par l'élément non trivial de $\Gal(\kk_D/\kk)$. 
D'après par exemple \cite[Lem\-ma~2.3]{VSANT19},
si $\rho^\vee$ et $\overline{\rho}$ sont isomorphes,
alors $r$ est impair.

Dans le cas (2.b),
l'entier $r$ est pair et $\GG^\tau$ est de la forme
$\GL_{r/2}(\kk)$ où $\kk$ est 
une extension~qua\-dratique de $\kk_D$
dans $\Mat_r(\kk_D)$. 
Selon \cite[Lemme 4.3.11]{Coniglio},
la représentation $\rho$ est distinguée par $\GG^\tau$
si et seulement si elle est autoduale. 

Dans le cas (2.c),
le sous-groupe $\GG^\tau$ est égal à 
$\GL_i(\kk_D)\times\GL_{r-i}(\kk_D)$. 
D'après par exemple~\cite[Proposition 2.14, Lemma 2.19]{VSANT19}, 
la repré\-sen\-tation $\rho$ est distinguée par $\GG^\tau$
si et seulement si $i$ est égal à $\lfloor r/2\rfloor$
et $\rho$ est autoduale.
D'après~par~exem\-ple~\cite[Lem\-ma~2.17]{VSANT19},
si $\rho$ est autoduale,
alors $r$ est pair ou égal à $1$.
Ceci prouve (2) et (1.b).

Supposons maintenant que $\bt$ soit 
dis\-tin\-guée par $\BJ\cap\G^\tau$.
Dans tous les~cas,
d'après ce qui~pré\-cède,
la représentation $\rho$ est $\tau$-autoduale,
donc les restrictions à $\BJ^0$ de $\bt^{\tau\vee}$ et $\bt$
sont isomorphes.
Il y a donc un caractère non ramifié $\boldsymbol{\chi}$ de $\BJ$ tel
que $\bt^{\tau\vee}\simeq\bt\boldsymbol{\chi}$.
Notons $\pi$ la représentation~cus\-pi\-dale $\G^\tau$-distinguée
obtenue par induction compacte de $(\BJ,\bt)$ à $\G$.
Elle est au\-to\-duale d'après le théorème \ref{GuoBM}.
Elle contient donc à la fois $\bt$ et $\bt^{\tau\vee}\simeq\bt\boldsymbol{\chi}$, 
ce dont on déduit que $\boldsymbol{\chi}$ est trivial.
\end{proof}

\subsection{}
\label{malinlegarcon00}

Fixons un ordre maximal $\aa$ de $A$ stable par $\tau$,
et notons $\Om_{\aa}$ l'ensemble des types $\tau$-auto\-duaux
$(\BJ,\bt)$ de niveau $0$ de $G$ tels que $\BJ$ normalise $\aa$. 

\begin{lemm}
\label{gersonide0}
L'ensemble $\Om_{\aa}$ est fini,
et il existe un sous-ensemble
$\Om_{\aa}^+\subseteq\Om_{\aa}^{\phantom{'}}$ tel que~: 
\begin{enumerate}
\item 
le cardinal de $\Om_{\aa}^{\phantom{'}}$ est le double de celui
de $\Om_{\aa}^+$,
\item
pour tout $(\BJ,\bt)\in\Om_{\aa}^+$,
la représentation $\bt$ est $\BJ\cap\G^\tau$-distinguée.
\end{enumerate}
\end{lemm}

\begin{proof}
Même preuve qu'en niveau non nul (voir le lemme \ref{gersonide}). 
\end{proof}

Pour tout $(\BJ,\bt)\in\Om_{\aa}$, on pose~:
\begin{equation*}
\Pi(\BJ,\bt) = \ind^\G_{\BJ} (\bt).
\end{equation*}
La représentation $\Pi(\BJ,\bt)$
est une représentation cuspidale autoduale de niveau $0$
de $G$.
Notons $\AA(G,\0)$ l'ensemble des classes d'isomorphisme
de représentations cuspidales auto\-duales de $G$ de niveau $0$.

\begin{prop}
\label{galettedepain0}
\begin{enumerate}
\item
L'application de $\Om_\aa$ dans $\AA(G,\0)$ définie par~:
\begin{equation*}
\Pi:(\BJ,\bt)\mapsto\Pi(\BJ,\bt)
\end{equation*}
est surjective.
\item
Deux paires $(\BJ,\bt)$ et $(\BJ',\bt')$ de $\Om_\aa$
ont la même image par $\Pi$ si et seulement si 
$\BJ'=\BJ$ et les~re\-pré\-sentations
$\bt$, $\bt'$ sont conjuguées sous $\BJI_{\aa}$. 
\item 
Si $\bt$ est $\BJ\cap\G^\tau$-distinguée,
alors $\Pi(\BJ,\bt)$ est $G^\tau$-distinguée.
\end{enumerate}
\end{prop}

\begin{proof}
Même preuve qu'en niveau non nul
(voir la proposition \ref{galettedepain}),
en remplaçant la proposition \ref{MAIN1prop}
par le lemme \ref{HMwarning}.
\end{proof}

Enfin, nous avons le résultat suivant. 
Notons $\AA^+(G,\0)$ 
le sous-ensemble de $\AA(G,\0)$
formé des représentations distinguées par $\G^\tau$.

\begin{prop}
\label{OlivierdeNoyen0}
\begin{enumerate}
\item 
L'ensemble $\AA(G,\0)$ est fini. 
\item 
Si $\AA^+(G,\0)$ n'est pas vide,
alors~:
\begin{equation*}
|\AA^+(G,\0)| \> \frac {1} {2} \cdot |\AA(G,\0)|. 
\end{equation*} 
\item
Pour que $\AA^{+}(G,\0)$ soit non vide, 
il faut et suffit qu'il existe à la fois une représentation cuspidale
auto\-dua\-le de $G$ de niveau $0$
et un ordre maximal dans $A$ stable par $\tau$.
\end{enumerate}
\end{prop}

\begin{proof}
Même preuve qu'en niveau non nul (voir la proposition \ref{OlivierdeNoyen}), 
en remplaçant la proposition \ref{grammage} par le lemme \ref{HMwarning},
la proposition \ref{galettedepain} par la proposition \ref{galettedepain0} et
le lemme  \ref{gersonide} par le lemme \ref{gersonide0}.
\end{proof}

\begin{rema}
Dans le cas où $G$ est déployé et $\a$ n'est pas un carré de $F^\times$, 
Chommaux~et Matringe \cite{ChommauxMatringe} donnent 
directement une condition nécessaire et suffisante pour qu'une représen\-ta\-tion 
cuspidale de niveau $0$ de $G$ soit distinguée par $G^\tau$,
en termes d'objets
(appelés~paires~ad\-mis\-si\-bles modérées)
paramétrant les types de niveau $0$ de $G$.
Quand $G$ est une forme intérieure quelconque,
la preuve de \cite[Theorem 2.1]{ChommauxMatringe} 
reste valable pour les représentations cus\-pida\-les~dont le transfert à
$\GL_{2n}(F)$ est cuspidal. 
Pour une représentation cuspida\-le
quelconque, la~para\-mé\-tri\-sation d'un type de niveau $0$ par une paire
ad\-mis\-si\-ble modérée (voir \cite{SZ2}) est moins transparente et ne
permet pas une adaptation immédiate de \cite{ChommauxMatringe}.
Nous avons préféré procéder ici com\-me en niveau non nul, 
de façon à avoir une approche uniforme.
\end{rema}

\subsection{}
\label{findumonde}

Dans ce paragraphe,
nous montrons que, contrairement au cas de niveau non nul
(propo\-sition \ref{MAIN1prop})
et à celui de l'involution galoi\-sien\-ne traitée dans \cite{AKMSS},
l'existence d'un caractère~sim\-ple~$\tau$-autodual n'entraîne pas celle
d'un type $\tau$-autodual.

Supposons par exemple que $r=1$ et que $K/F$ soit non ramifiée. 
Désignons par $q$ le cardinal du corps résiduel de $F$,
et fixons un caractère non quadratique
$\rho$ de $\kk_D^\times$ tel que
$\rho^{q^n}=\rho^{-1}$.~On note
$\BJ$ son~nor\-ma\-li\-sateur dans $\D^\times$ et $b$ l'indice de $\BJ$
dans $D^\times$.
Si $\w$ est une uniformisante de $D$ telle que $\w^d$ soit une uniformisante
de $F$,
alors $\BJ$ est engendré par $\w^b$ et le groupe des unités~de $D$.
Soit enfin $\bt$ un caractère de $\BJ$ prolongeant $\rho$
et tel que $\bt(\w^b)\in\{-1,1\}$.
Alors $(\BJ,\bt)$ est~un type de niveau $0$,
son induite compacte $\pi$ à $D^\times$
est irréductible (de dimension $b$) et autoduale,~et
pourtant $\bt$ ne peut pas être $\tau$-autodual
car le caractère $\rho$ n'est pas quadratique.

\section{Calcul du facteur epsilon}
\label{promenade}

Dans toute cette section, 
on suppose que l'élément $\a$ de la section \ref{SEC4}
n'est pas un carré de $F^\times$,
et on pose $K=F[\dd]$.
On fixe une clô\-tu\-re sépa\-ra\-ble $\overline{F}$ de $F$
et un plongement de $K$ dans $\overline{F}$.
On note $\Ww_F$ le groupe~de Weil de $\overline{F}$ sur $F$. 
On rappelle que $\bo_{K/F}$ désigne le caractère quadratique
de $\F^\times$ de noyau $\N_{K/F}(\K^\times)$,
et que $\psi$ est un caractère additif 
de $F$, trivial sur $\p_F$ mais pas sur $\Oo_F$. 
On pose $\psi_K=\psi\circ\tr_{K/F}$.

Le résultat principal de cette section est le
théorème \ref{IsabelledeFrejusNS} et la proposition \ref{alculEPniveau0},
qui calculent le facteur epsilon du~pa\-ragraphe \ref{intro1}
pour une représentation cuspidale autoduale de $G$.

\subsection{}
\label{PAR71}

Dans un premier temps, 
on considère des représentations cus\-pi\-dales de~$\GL_{2n}(F)$.
Si $\pi$ est une telle représentation, on lui associe~:
\begin{itemize}
\item 
son caractère central $c_\pi$, qui est un caractère de $\F^\times$,
\item
son paramètre de~Lang\-lands $\phi$, 
qui est une représentation irréductible de $\Ww_F$ de dimension $2n$, 
\end{itemize}
et on note $\phi_K$ la restriction de $\phi$ à $\Ww_K$.
Rappelons (voir \cite{Tate}) qu'il correspond à $\phi_K$ un 
facteur epsilon, que l'on note $\e(s,\phi_K,\psi_K)$.
On pose~:
\begin{equation}
\label{epKpi}
\ep_K(\pi) = \ep_K(\phi) = \e\left(\frac 1 2,\phi_K,\psi_K\right),
\end{equation}
qui ne dépend pas du choix de $\psi$. 
Il sera commode de poser $\ep_F(V)=\e(1/2,V,\psi)$
pour toute représentation $V$
de~di\-mension finie de $\Ww_F$.
\'Ecrivons (\cite[(30.4.2)]{BHGL2})~:
\begin{equation}
\label{ballon}
\ep_F(\Ind_{K/F} (\phi_K))
= \boldsymbol{\lambda}_{K/F}^{2n} \cdot
\e\left(\frac 1 2,\phi_K,\psi_K\right)
\end{equation}
où $\Ind_{K/F}$ désigne l'induction de $\Ww_K$ à $\Ww_F$
et $\boldsymbol{\lambda}_{K/F}$ la constante de Langlands
(\cite[(30.4.1)]{BHGL2}).
Le carré de~cette constante étant égal à $\bo_{K/F}(-1)$
(voir \cite[(30.4.3)]{BHGL2}),
la quantité
\eqref{ballon} devient $\bo_{K/F}(-1)^n\cdot \ep_K(\pi)$.
D'un autre côté, on a~:
\begin{equation}
\label{splitphiK}
\ep_F(\Ind_{K/F} (\phi_K))
= \ep_F(\phi\otimes\Ind_{K/F}(1))
= \ep_F(\phi)\cdot\ep_F(\phi\bo_{K/F})
\end{equation}
où $\bo_{K/F}$ est considéré comme un caractère de $\Ww_F$
\textit{via} le morphisme de réciprocité d'Artin de~la
théorie du corps de classes.
D'après les propriétés de la correspondance de Langlands locale,~le
facteur $\ep_F(\phi)$ est égal au facteur epsilon de
Godement-Jacquet $\e(1/2,\pi,\psi)$, que l'on note $\ep_F(\pi)$.
De façon analogue,
$\ep_F(\phi\bo_{K/F})$ est égal à $\ep_F(\pi\bo_{K/F})$.
On trouve donc~:
\begin{equation}
\label{formulegeneraleepK}
\ep_K(\pi) = \bo_{K/F}(-1)^n \cdot \ep_F(\pi) \cdot \ep_F(\pi\bo_{K/F}).
\end{equation}
Nous allons pousser ce calcul plus loin,
en supposant que $\pi$ est autoduale. 

\subsection{}
\label{Charity}

Dans ce paragraphe, 
$\pi$ est autoduale de niveau non nul. 
Le résultat suivant exprime le~si\-gne $\ep_K(\pi)$~en fonc\-tion
de~don\-nées relevant de la description de $\pi$ en termes de types
grâce à \cite[Théo\-rème~4.1]{BHCJM01}.

\begin{prop}
\label{FormulePourEKPi}
Soit $\pi$ une représentation cus\-pi\-dale autoduale de~$\GL_{2n}(F)$, 
de niveau non nul. 
Soit $[\aa,\b]$ une strate simple maximale de $\Mat_{2n}(F)$ telle que $\pi$
contienne un caractère simple de $\Cc(\aa,\b)$.
Alors on a~: 
\begin{equation}
\label{FORMULEEPK0}
\ep_K(\pi) =
\cc_{\pi}(-1) \cdot
\bo_{K/F}\left((-1)^n \det\b\right).
\end{equation}
\end{prop}

\begin{proof}
Supposons d'abord que $\pi$ est une représentation cuspidale de $\GL_{2n}(F)$
de~ni\-veau non nul, pas forcément autoduale. 
Soit $[\aa,\b]$ une strate simple maximale de $\Mat_{2n}(F)$ telle que $\pi$
contienne un caractère simple de $\Cc(\aa,\b)$.
Invoquant \cite[Théorème 4.1]{BHCJM01}, on a~:
\begin{equation}
\label{extase}
\ep_F(\pi\chi) = \chi(\det \b)^{-1} \cdot \ep_F(\pi)
\end{equation}
pour tout caractère modérément ramifié $\chi$ de $\F^\times$.
Appliquant ce résultat au caractère quadra\-tique $\bo_{K/F}$, 
on trouve que $\ep_F(\pi\bo_{K/F})$ est égal à 
$\bo_{K/F}(\det \b)\cdot\ep_F(\pi)$.
D'après \cite{BHBLMS99},
on a l'iden\-ti\-té~$\ep_F(\pi) \cdot \ep_F(\pi^\vee)= c_{\pi}(-1)$. 
Supposant maintenant que $\pi$ est autoduale,
on obtient le résultat an\-noncé en appliquant \eqref{formulegeneraleepK}. 
\end{proof}

Nous allons pousser plus loin le calcul de $\ep_K(\pi)$
en choisissant pour $[\aa,\b]$ une strate particu\-lière.  

\begin{prop}
\label{hugonnet}
Soit $\pi$ une représentation cus\-pi\-dale autoduale de~$\GL_{2n}(F)$ de 
niveau non nul, d'endo-classe notée $\TT$.
On note $T/T_0$ l'extension quadratique qui lui est associée au
paragraphe \ref{invTT}.
On a~:
\begin{equation}
\ep_K(\pi) = c_{\pi}(-1) \cdot \bo_{T/T_0}(\a)^{2n/[T:F]}.
\end{equation}
\end{prop}

\begin{proof}
Comme au paragraphe \ref{P12},
fixons un élément $\s\in\GL_{2n}(F)$ tel que $\ss^2=1$,~de
polynôme ca\-rac\-téristique $(X^2-1)^{n}$,
et notons encore $\s$
l'au\-to\-morphisme de~con\-ju\-gaison par~$\s$.
D'après la~remarque \ref{lll},
il y a une strate sim\-ple maximale $\s$-autoduale
$[\aa,\b]$ dans $\Mat_{2n}(F)$ telle que $\pi$~con\-tienne 
un caractère simple $\s$-autodual de $\Cc(\aa,\b)$.
Posons $E=\F[\b]$ et $E_0=\F[\b^2]$,~et
identifons $T$ à la sous-extension modérément ramifiée maximale de $E$
sur $F$ et $T_0$ à $T\cap E_0$.
Le degré de $\TT$ est égal à $[E:F]$.
Posant $m=2n/[E:F]$, on a~:
\begin{equation*}
\label{jounellis}
\det \b = \N_{E/F}(\b)^m
= \N_{E_0/F}(-\b^2)^m=(-1)^{n}\cdot\N_{E_0/F}(\b^2)^m
\end{equation*}
car $\N_{E/E_0}(\b)=-\b^2$ et $m[E_0:F]=n$.
Appliquant la proposition \ref{FormulePourEKPi}, on trouve~:
\begin{equation*}
c_{\pi}(-1) \cdot \ep_K(\pi) =
\bo_{K/F}(\N_{E_0/F}(\b^2))^m = (\a,\N_{E_0/F}(\b^2))_{\F}^m 
\end{equation*}
où $(\cdot,\cdot)_F$ désigne le symbole de Hilbert sur $F$.

Posons maintenant $\g=\N_{E/T}(\b)$.
Comme $[E:T]$ est une puissance de $p$, 
qui est impair, on~a $\s(\g)=-\g$.
On en déduit d'une part que $\T$ est engendré par $\g$ sur $T_0$,
et d'autre part que~:
\begin{equation*}
\N_{E_0/F}(\b^2) = (-1)^{[E_0:F]} \cdot \N_{E/F}(\b) 
=  (-1)^{[E_0:F]} \cdot \N_{T/F}(\g) 
= \N_{T_0/F}(\g^2).
\end{equation*}
Par conséquent, 
on a $(\a,\N_{E_0/F}(\b^2))_{\F}=(\a,\N_{T_0/F}(\g^2))_{\F}$.
Les proprié\-tés du symbole de Hilbert vis-à-vis du 
changement de base 
(voir par exemple \cite[Proposition IV.5.1(8)]{FesenkoVostokov})
assurent que~:
\begin{equation*}
(\a,\N_{T_0/F}(\g^2))_{\F} = (\a,\g^2)_{\T_0}
\end{equation*}
et le membre de droite est égal à $\bo_{T/T_0}(\a)$. 
On conclut en observant que $2n/[T:F]=m[E:T]$
où $[E:T]$ est impair. 
\end{proof}

\subsection{}

On suppose maintenant que $\pi$ est autoduale de niveau $0$. 

\begin{theo}
\label{ChommauxMatringeTH}
Soit $\pi$ une représentation cuspidale autoduale de~ni\-veau $0$ 
de $\GL_{2n}(F)$.  
\begin{enumerate}
\item 
On a~:
\begin{equation*}
\ep_K(\pi) = (-1)^{f_{K/F}}.
\end{equation*} 
\item
Le caractère central de $\pi$ est quadratique et non ramifié. 
\end{enumerate} 
\end{theo}

\begin{proof}
Dans le cas où le caractère central de $\pi$ est trivial, 
le résultat est donné par \cite[Theorem 6.1]{ChommauxMatringe}.
On prendra garde au fait que les auteurs de 
\cite{ChommauxMatringe} uti\-lisent une normalisation 
différen\-te~et~cal\-culent~:
\begin{equation*}
\ep_F(\Ind_{K/F} (\phi_K)) = \bo_{K/F}(-1)^n\cdot \ep_K(\pi).
\end{equation*}
Dans le cas général, 
$\pi$ est paramétré, comme expliqué dans \cite{ChommauxMatringe},
par un caractère modérément ramifié $\chi$ de $L^\times$,
où $L$ est une extension non ramifiée de degré $2n$ de $F$,
avec les propriétés~sui\-vantes~:
\begin{itemize}
\item 
le caractère $\chi$ est distinct de tous ses conjugués sous $\Gal(L/F)$, 
\item
le fait que $\pi$ soit autoduale se traduit par le fait que $\chi$
est trivial sur $\N_{L/L_0}(L^\times)$,
où $L_0$ est la sous-extension de degré $n$ de $F$ dans $L$,
\item 
le caractère central de $\pi$ est égal à la restriction de $\chi$ à 
$F^\times$.
\end{itemize}
En particulier,
le caractère $\chi$ est trivial sur $\Oo^\times_{L_0}$,
ce qui prouve (2),
le fait que~$c_\pi$ soit quadratique provenant de ce que $\pi$ est autoduale. 
Supposons que le caractère central~de~$\pi$ soit~non trivial.
Fixons un caractère non ramifié $\varphi$ de $F^\times$ tel que $\varphi^{2n}$
soit d'ordre $2$.
Alors $\pi\varphi$~est~une~représen\-ta\-tion cus\-pi\-dale
autoduale de niveau $0$
et de caractère central trivial.
On a $\ep_K(\pi\varphi) = (-1)^{f_{K/F}}$.
Par ailleurs, on a~:
\begin{equation*}
\ep_K(\pi\varphi) = \bo_{K/F}(-1)^n \cdot \ep_F(\pi\varphi)
\cdot \ep_F(\pi\varphi\bo_{K/F})
\end{equation*}
d'après \eqref{formulegeneraleepK}.
Le résultat suit de ce que $\ep_F(\pi\varphi)=\ep_F(\pi)$
et $\ep_F(\pi\varphi\bo_{K/F})=\ep_F(\pi\bo_{K/F})$
car $\varphi$ est un caractère non ramifié et ni $\pi$
ni $\pi\bo_{K/F}$ ne sont des caractères non ramifiés de $F^\times$
(voir Tate \cite[(3.4.6)]{Tate}). 
\end{proof}

\subsection{}
\label{passageaGprime}
\label{Banaan}

On considère maintenant des représentations cus\-pi\-dales autoduales
d'une forme intérieu\-re~$G$ de $\GL_{2n}(F)$. 
Si $\pi$ est une telle représentation, on note $\TT$ son endo-classe et
$\d$ son de\-gré~pa\-ra\-mé\-trique.
On lui associe~:
\begin{itemize}
\item 
son caractère central $c_\pi$, qui est un caractère de $\F^\times$,
\item
son transfert de Jacquet-Langlands à $\GL_{2n}(F)$,
noté $\pi'$, 
\item
le paramètre de~Lang\-lands $\phi$ de $\pi'$, 
qui est une représentation irréductible de dimension $2n$
du~grou\-pe~de Weil-Deligne
$\Ww_F\times{\rm SL}_2(\CC)$.
\end{itemize}
Comme au paragraphe \ref{regimbard},
on écrit $\pi'$ sous la~for\-me $L(\pi'_{0},s)$,
où $s$ est égal à $2n/\d$ selon~\eqref{scoprimem}
et où $\pi'_0$ est une~repré\-sentation~cus\-pidale autoduale
de $\GL_{\d}(F)$,
d'endo-classe $\TT$ selon le théorème \ref{conjecjl}. 
Comme dans le cas déployé, on pose~:
\begin{equation}
\label{epKpins}
\ep_K(\pi) = \ep_K(\pi') = \ep_K(\phi) = \e\left(\frac 1 2,\phi_K,\psi_K\right),
\end{equation}
qui ne dépend pas du choix de $\psi$.

\begin{lemm}
\label{paritedelta}
\begin{enumerate}
\item
Le degré paramétrique $\d$ est soit pair, soit égal à $1$.
\item
Si $\d=1$, alors $r=1$ et $\TT$ est nulle. 
\end{enumerate}
\end{lemm}

\begin{proof}
Comme en \eqref{pardeg},
écrivons $\d$ sous la forme $mb[E:F]$.
Si $\TT$ est non nulle,~on déduit (1) de ce que $[E:F]$ est pair. 
Si $\TT$ est nulle,
le lemme \ref{fromagequipue} appliqué à $\pi'_0$ assure que $\d$,
qui est égal à $mb$, est pair ou égal à $1$.

Ensuite,
rappelons que $r=\d/(d,\d)$ d'après \eqref{rsprime}.
Si $\d=1$, alors d'une part $r=1$,
et d'autre part $\pi'_0$ est un~caractère~qua\-dratique
(donc~mo\-dé\-rément ramifié car $p\neq2$) de $F^\times$,
donc $\TT$ est nulle. 
\end{proof}

\begin{prop}
\label{hugonnetsneq2n}
Soit $\pi$ une représentation cus\-pi\-dale autoduale de~$\G$.
Si $\d\neq1$, on a~:
\begin{equation}
\ep_K(\pi) = \ep_K(\pi'_0)^{s}.
\end{equation}
\end{prop}

\begin{proof}
Comme au paragraphe \ref{PAR71}
(voir \eqref{ballon} et \eqref{splitphiK}), on a~:
\begin{equation*}
\ep_K(\pi) = \bo_{K/F}(-1)^{n} \cdot \ep_F(\pi') \cdot \ep_F(\pi'\bo_{K/F}),
\end{equation*}
où $\ep_F(\pi')$ désigne le facteur de Godement-Jacquet
$\e(1/2,\pi',\psi)$.
Comme $\d\neq1$,
la représenta\-tion cuspidale $\pi'_0$ n'est pas un caractère non ramifié
de $F^\times$.
On a donc $\ep_F(\pi')=\ep_F(\pi'_0)^s$,
et on a un résultat similaire pour $\pi'\bo_{K/F}$.
Comme $s$ divise $n$ d'après le lemme \ref{paritedelta},
on trouve~: 
\begin{equation*}
\ep_K(\pi) = \bo_{K/F}(-1)^{n} \cdot \ep_F(\pi') \cdot \ep_F(\pi'\bo_{K/F}) 
= \left(\bo_{K/F}(-1)^{n/s} \cdot
\ep_F(\pi'_0) \cdot \ep_F(\pi'_0\bo_{K/F})\right)^{s} 
\end{equation*}
ce qui donne le résultat voulu.
\end{proof}

On en déduit le théorème suivant. 

\begin{theo}
\label{IsabelledeFrejusNS}
Soit $\pi$ une représentation cuspidale autoduale de $G$
de niveau non nul.~On a alors~:
\begin{equation*}
\ep_K(\pi) = c_{\pi}(-1) \cdot \bo_{T/T_0}(\a)^{2n/[T:F]}
\end{equation*}
où $T/T_0$ est l'extension quadratique associée à l'endo-classe de $\pi$.
\end{theo}

\begin{proof}
La représentation $\pi$ étant de niveau non nul,
$\pi'_0$ est niveau non nul. 
On peut donc lui appliquer la proposition \ref{hugonnet},
ce qui donne~:
\begin{equation*}
\ep_K(\pi'_0) = c_{\pi'_0}(-1) \cdot \bo_{T/T_0}(\a)^{\d/[T:F]}.
\end{equation*}
On en déduit également,
d'après le lemme \ref{paritedelta}, que $\d\neq1$.
On peut donc appliquer la proposition \ref{hugonnetsneq2n},
qui donne~: 
\begin{equation*}
\ep_K(\pi) 
= c_{\pi'_0}(-1)^{s} \cdot \bo_{T/T_0}(\a)^{s\d/[T:F]}
= c_{\pi'}(-1) \cdot \bo_{T/T_0}(\a)^{2n/[T:F]}.
\end{equation*}
La correspondance de Jacquet-Langlands préservant le caractère central,
on obtient le résultat voulu.
\end{proof}

Dans le cas de niveau $0$, on a le résultat suivant. 

\begin{prop}
\label{alculEPniveau0}
Soit $\pi$ une représentation cuspidale autoduale de $G$
de niveau $0$.
\begin{enumerate}
\item
Si $\d=1$, \ie si $r=1$ et $\pi=\chi\circ{\rm Nrd}_{A/F}$
pour un~ca\-rac\-tère quadratique $\chi$ de $\F^\times$, on a~: 
\begin{equation*}
\ep_K(\pi) = \left\{ 
\begin{array}{rl}
-1 & \text{si $\chi$ est trivial sur $\N_{K/F}(\K^\times)$}, \\
1 & \text{sinon}.
\end{array}\right.
\end{equation*}
\item 
Sinon, on a $\ep_K(\pi) =(-1)^{sf_{K/F}}$.
\end{enumerate}
\end{prop}

\begin{proof}
Dans le premier cas,
\ie si $\d=1$,
le transfert de $\pi$ à $\GL_{2n}(F)$ est~la représentation de Steinberg
tordue par le caractère $\chi\circ\det$
et le calcul a été~fait dans \cite[Section 4]{Chommaux}.
Dans le second cas,
le résultat suit de la proposition \ref{hugonnetsneq2n}
et du théorème \ref{ChommauxMatringeTH}. 
\end{proof}

Pour toute représentation cuspidale autoduale $\pi$ de $G$,
on pose~:
\begin{equation*}
\EP_K(\pi) = c_{\pi}(-1) \cdot \ep_K(\pi).
\end{equation*}
On observe que,
compte tenu du théorème \ref{ChommauxMatringeTH}(2),
on a $\EP_K(\pi) = \ep_K(\pi)$ quand $\pi$ est de niveau~$0$. 
On déduit de ce qui précède les résultats suivants
(voir les paragraphes
\ref{malinlegarcon}, \ref{malinlegarconcor} et \ref{malinlegarcon00}
pour~les no\-ta\-tions).

\begin{coro}
\label{IsabelledeFrejusNScoro-1}
Soit $\pi$ une représentation cuspidale autoduale de $G$
de niveau non nul.  
La représentation $\pi$ contient un caractère simple $\tau$-autodual
si et seulement si $\EP_K(\pi)=(-1)^r$. 
\end{coro}

\begin{proof}
Il suffit d'appliquer les théorèmes
\ref{IsabelledeFrejusNS} et \ref{THMEXISTENCECARACSIMPLE}.
\end{proof}

\begin{coro}
\label{IsabelledeFrejusNScoro}
\label{IsabelledeFrejusNSlevel0}
Soit $\pi$ une représentation cuspidale autoduale de $G$
d'endo-classe $\TT$. 
\begin{enumerate} 
\item 
Si $\pi$ est distinguée par $\G^\tau$, alors $\EP_K(\pi)=(-1)^r$.
\item
Si $\EP_K(\pi)=(-1)^r$, alors $\AA^+(G,\TT)$ est non vide et~:
\begin{equation}
\label{cardinaux}
|\AA^+(G,\TT)| \> \frac {1}{2} \cdot |\AA(G,\TT)|.
\end{equation}
\end{enumerate}
\end{coro}

\begin{proof}
Dans le cas où $\pi$ est de niveau non nul,
l'assertion (1) suit de la proposition \ref{grammage} et du corollaire 
\ref{IsabelledeFrejusNScoro-1},
et l'assertion (2) suit du corollaire \ref{IsabelledeFrejusNScoro-1}
et de la proposition \ref{OlivierdeNoyen}.

On suppose dans toute la suite de la preuve que $\pi$ est de niveau $0$.
Prouvons l'assertion~(1).
Supposons d'abord que~$\d=1$,
\ie que $r=1$ et $\pi=\chi\circ{\rm Nrd}_{A/F}$
pour un caractère~qua\-dratique $\chi$ de $\F^\times$. 
Le fait que $\pi$ soit $G^\tau$-distingué signifie que $\pi$ est trivial
sur $\G^\tau$,
\ie que~le caractère $\chi\circ\N_{K/F}$ est trivial sur $K^\times$.
Appliquant la pro\-po\-si\-tion \ref{alculEPniveau0},
on en déduit que $\ep_K(\pi)=-1=(-1)^r$.
Supposons maintenant que $\d\neq1$.
D'après la proposition \ref{alculEPniveau0},
il s'agit de prouver que~:
\begin{equation}
\label{signes0}
(-1)^{sf_{K/F}}=(-1)^r.
\end{equation}
D'après
{le lemme \ref{HMwarning} et la proposition \ref{HMwarningprop}(1),}
la représentation cuspidale distinguée $\pi$ contient un~type $\tau$-autodual
$(\BJ,\bt)$ de niveau $0$, tel que $\bt$ soit $\BJ\cap\G^\tau$-distingué.
Considérons tour à tour les trois cas de la proposition \ref{HMwarningprop}(2).
Dans le cas (2.c), 
l'entier $r$ est pair ou égal à $1$.
Le cas~où $r=1$,
qui entraîne que $\bt$ est un caractère quadratique, 
n'est possible que si $s=2n$,~ce qui a été exclu.
Donc $r$ est pair et,
$f_{K/F}$ étant pair, l'identité \eqref{signes0} est vérifiée. 
Dans le cas (2.b), $r$ et $f_{K/F}$ sont 
pairs,~ce qui donne encore \eqref{signes0}. 
Dans le cas (2.a) enfin,
$f_{K/F}$ et $r$ sont impairs.
Il nous reste donc à prouver que $s$ est impair dans ce cas. 

Supposons donc que $K/F$ soit ramifiée.
Notons $\rho$ la représentation cuspidale de $\GL_r(\kk_D)$~dé\-fi\-nie par $\bt$, 
et notons $\xi$ le caractère $\kk_D$-régulier de $\tt^\times$
paramétrant $\rho$ comme dans la preuve du lemme \ref{hamptons}. 
Posant $b=d/s$,
et notant $q$ le cardinal de $\kk_F$ et $N$ l'ordre de $\xi$, 
on sait que l'ordre de $q$ dans $(\ZZ/N\ZZ)^\times$ est égal à $rb$.
L'ordre de $q^n$ vaut donc~:
\begin{equation*}
{\frac {rb} {(rb,n)}
= \frac {2rb} {(2rb,rd)}
= \frac {2} {(2,s)}.}
\end{equation*}
Or d'après la proposition \ref{HMwarningprop},
la représentation $\rho^\vee$ est isomorphe au conjugué de $\rho$
par l'élément d'ordre $2$ de $\Gal(\kk_D/\kk_F)$.
D'après par exemple \cite[Lemma 2.3]{VSANT19},
cela implique que $\xi^{q^n}=\xi^{-1}$,
\ie que l'ordre de $q^n$ dans $(\ZZ/N\ZZ)^\times$ est égal à $2$. 
On en déduit que $s$ est impair.
Ceci termine la preuve de l'assertion~(1).

Supposons enfin que $\AA^+(G,\0)$ soit vide.
D'après les propositions \ref{OlivierdeNoyen0} et \ref{baff}, 
l'extension $K/F$ est ramifiée et $d$ est impair.
Dans ce cas, $r$ est pair et $s$ (qui divise $d$) est impair,
et la proposition \ref{alculEPniveau0} assure que
$\ep_K(\pi)\neq(-1)^r$.
L'inégalité \eqref{cardinaux} suit alors de la proposition
\ref{OlivierdeNoyen0}.
\end{proof}

\begin{rema}
\label{Parlement}
La preuve du corollaire \ref{IsabelledeFrejusNSlevel0} montre plus
précisément que,
quand $K/F$~est ramifiée, 
$s$ est impair pour toute représentation $\pi\in\AA(G,\0)$.
On en déduit que la valeur $\ep_K(\pi)$ est la même pour toutes
les $\pi\in\AA(G,\0)$
telles que $\d\neq1$.~Lorsque $\d=1$ en revanche
(ce qui ne peut se produire que si $r=1$), 
considérons les caractères auto\-duaux
$\pi(\chi)=\chi\circ{\rm Nrd}_{A/F}$
où $\chi$ est un caractère quadratique de $F^\times$.
Ils sont tous les quatre~dans $\AA(G,\0)$,
mais $\ep_K(\pi(\chi))=-1$
si et seu\-le\-ment si $\chi$ est trivial sur $\N_{K/F}(K^\times)$.
Voir le lemme \ref{comptageAsp}.
\end{rema}

\section{La conjecture de Prasad et Takloo-Bighash}
\label{SECLAPREUVE}

Soit $A$ une $F$-algèbre centrale simple de degré réduit $2n$,
et posons $G=A^\times$.
C'est une forme intérieure du groupe $G'=\GL_{2n}(F)$. 

\subsection{}
\label{banquets}

On fixe une endo-classe~auto\-duale $\TT$ de degré divisant $2n$.
On rappelle que $\AA(\G,\TT)$ est l'en\-sem\-ble des
(classes d'isomorphisme de)
représentations cuspi\-dales auto\-duales de $\G$ 
d'endo-classe $\TT$.
\'Etant donné une telle représentation,
il lui correspond son para\-mètre de Lang\-lands $\phi$,~qui 
est une représentation irréductible autoduale de dimension $2n$
de $\Ww_F\times{\rm SL}_2(\CC)$.
Selon~le lemme~de Schur, 
le $\CC$-espace vectoriel des homomorphismes de $\phi$ vers $\phi^\vee$
est~de di\-mension~$1$.
Soit $f:\phi\to\phi^\vee$~un isomorphisme.
Identifiant canoniquement $\phi^{\vee\vee}$ à $\phi$, 
l'isomor\-phisme contragrédient $f^\vee$ est donc égal à
$\chi(\phi)f$ pour un unique signe $\chi(\phi)\in\{-1,1\}$,
indé\-pen\-dant du choix de $f$.

\begin{defi}
On dit que~$\phi$~est~de~pa\-ri\-té~\textit{orthogonale} si $\chi(\phi)=1$, 
et de parité \textit{symplectique} si $\chi(\phi)=-1$.
\end{defi}

Notons
$\AA^{{\rm sp}}(\G,\TT)$ et $\AA^{{\rm or}}(\G,\TT)$~les~en\-sembles
de représentations $\pi$ dans $\AA(\G,\TT)$
dont le para\-mètre de Lang\-lands
soit de parité respectivement symplectique et orthogonale.
Ainsi l'ensemble $\AA(\G,\TT)$ est la~réu\-nion dis\-jointe de 
$\AA^{{\rm sp}}(\G,\TT)$ et $\AA^{{\rm or}}(\G,\TT)$.

\begin{lemm}
\label{lemmedeparitegeneralise} 
Les ensembles $\AA^{{\rm sp}}(\GL_{2n}(F),\TT)$ et
$\AA^{{\rm or}}(\GL_{2n}(F),\TT)$
ont le même cardinal. 
\end{lemm}

\begin{proof}
Si $\TT$ est non nulle et $m=2n/\deg(\TT)$ ne vérifie pas les conditions du
lemme \ref{lemmedeparite}, 
les ensembles $\AA(\GL_{2n}(F),\TT)$, $\AA^{{\rm sp}}(\GL_{2n}(F),\TT)$ 
et $\AA^{{\rm or}}(\GL_{2n}(F),\TT)$ sont tous les trois vides.
Nous supposerons donc par la suite que ces conditions sont vérifiées.
(Si $\TT$ est nulle, la condition du lemme \ref{lemmedepariteniveauzero}
est toujours vérifiée.)

La preuve commence comme celle du lemme \ref{lemmedeparite}.
En particulier,
on a~une représentation $\bk$ de $\BJI$ prolon\-geant $\n$ telle que
$\bk^{\vee\s}$ soit isomorphe à $\bk$.  
Selon le paragraphe \ref{regimbard},
on a une bijection \eqref{bijthetacusp} entre~:
\begin{enumerate}
\item 
classes d'isomorphisme de représentations irréductibles $\bt$ de $\BJ$
triviale sur $\BJ^1$~dont~la~res\-triction à $\BJ^0$ soit l'inflation d'une
représentation cuspidale de $\GL_m(\ee)$, 
\item 
et classes~d'iso\-mor\-phisme de
représentations cuspidales de $G$ d'endo-classe $\TT$.
\end{enumerate}
Comme $\bk^{\vee\s}$ est isomorphe à $\bk$,
elle induit une bijection entre 
représentations $\bt$ telles que~$\bt^{\vee\s}$ soit isomorphe à $\bt$
et l'ensemble $\AA(\G,\TT)$.
À partir de là, 
on procède exactement comme dans la preuve de \cite[Lemma 7.2]{BHS},
en faisant agir sur $\AA(\G,\TT)$ le groupe $X$ des caractères
modérément ramifiés de $\T^\times$,
d'ordre divisant $2m$,
dont la restriction à $\Oo_T^\times$ soit quadratique.
On montre que chaque $\pi\in\AA(\G,\TT)$ a une orbite sous $X$~de
cardinal $4$,
contenant exactement deux éléments de $\AA^{{\rm sp}}(\G,\TT)$ et
deux de $\AA^{{\rm or}}(\G,\TT)$.
\end{proof}

\begin{rema}
On comparera ce lemme à \cite[Lemma 7.2]{BHS},
qui s'appuie sur une définition d'endo-classe autoduale différente de la 
nôtre, quoique probablement équivalente. 
\end{rema}

\begin{lemm}
\label{lemmedeparitegeneraliseinnerform} 
\label{comptageAsp} 
On a~:
\begin{equation}
\label{incl2}
|\AA^{{\rm sp}}(G,\TT)| = \frac 1 2 \cdot |\AA(G,\TT)| + 
\left\{
\begin{array}{rl}
2 & \text{si $\TT$ est nulle et $r=1$,} \\
0 & \text{sinon.}
\end{array}\right.
\end{equation} 
\end{lemm}

\begin{proof} 
D'après le paragraphe \ref{degueulis}, 
la correspondance de Jac\-quet-Lang\-lands et la~clas\-si\-fication des 
représentations essentiellement de carré intégrable de $\GL_{2n}(F)$
induisent une cor\-res\-pondance bijective~:
\begin{equation*}
\AA(G,\TT)
\quad\leftrightarrow\quad
\coprod\limits_{\d} \ \AA(\GL_{\d}(F),\TT)
\end{equation*}
où l'union disjointe dans le membre de droite porte
sur l'ensemble $\SD$ des entiers de la forme $\d(\pi)$ pour 
$\pi\in\AA(G,\TT)$~: voir le paragraphe~\ref{regimbard}
et le lemme \ref{fromagequipue} qui décrit précisément
l'ensemble~$\SD$.
\'Etant donné une représentation cuspidale $\pi\in\AA(G,\TT)$,
son transfert à $\GL_{2n}(F)$ s'écrit $L(\pi'_0,s)$ avec $s=2n/\d(\pi)$.
Son paramètre de Langlands $\phi$ est égal à $\phi_0\otimes{\rm Sym}^{s-1}$,
où $\phi_0$ est le~para\-mè\-tre de Langlands de $\pi'_0$
et ${\rm Sym}^{s-1}$ est la représentation irréductible de dimension $s$ de
${\rm SL}_2(\CC)$,
chacune des deux étant autoduale. 
Selon \cite[Lemma 3.2]{GanGrossPrasad}, on a 
$\chi(\phi)=\chi(\phi_0)\cdot\chi({\rm Sym}^{s-1})$.
Puis, selon \cite[Lemma 3.2.15]{GoodmanWallach},
on a $\chi({\rm Sym}^{s-1})=(-1)^{s-1}$.
Il s'ensuit que $\phi$ est de parité~symplecti\-que~si et seulement si~:
\begin{itemize}
\item 
ou bien $s$ est impair et $\pi'_0$ est de parité symplectique,
\item
ou bien $s$ est pair et $\pi'_0$ est de parité orthogonale, 
\end{itemize}
sachant que, 
si $\d=1$, tout 
caractère autodual (\ie quadratique) 
de $F^\times$ est de parité~or\-thogonale. 
On en déduit une bijection~: 
\begin{equation*}
\AA^{{\rm sp}}(G,\TT)
\quad\leftrightarrow\quad
\coprod\limits_{\text{$2n/\d$ impair}} \AA^{{\rm sp}}(\GL_{\d}(F),\TT)
\cup
\coprod\limits_{\text{$2n/\d$ pair}} \AA^{{\rm or}}(\GL_{\d}(F),\TT).
\end{equation*} 
D'après le lemme \ref{lemmedeparitegeneralise}, 
pour tout $\d\neq1$,
les ensembles 
$\AA^{{\rm sp}}(\GL_{\d}(F),\TT)$ et $\AA^{{\rm or}}(\GL_{\d}(F),\TT)$ 
ont~le même cardi\-nal, 
ils sont disjoints et leur réunion est $\AA(\GL_{\d}(F),\TT)$.
Compte tenu du lemme \ref{paritedelta},
si $\TT$ est non nulle ou si $r\neq1$, on en déduit que~:
\begin{eqnarray*}
|\AA^{{\rm sp}}(G,\TT)| &=& 
\sum\limits_{\text{$2n/\d$ impair}} |\AA^{{\rm sp}}(\GL_{\d}(F),\TT)| +
\sum\limits_{\text{$2n/\d$ pair}} |\AA^{{\rm or}}(\GL_{\d}(F),\TT)| \\
&=& 
\sum\limits_{\text{$2n/\d$ impair}} \frac {1}{2} \cdot |\AA(\GL_{\d}(F),\TT)| +
\sum\limits_{\text{$2n/\d$ pair}} \frac {1}{2} \cdot |\AA(\GL_{\d}(F),\TT)| \\
&=& \frac {1}{2} \cdot \sum\limits_{\d} |\AA(\GL_{\d}(F),\TT)|
\end{eqnarray*}
ce qui est égal à $|\AA(G,\TT)|/2$ comme voulu.
Si $\TT$ est nulle et $r=1$, l'ensemble~:
\begin{equation*}
\AA^{{\rm or}}(\GL_{1}(F),\0) = \AA(\GL_{1}(F),\0) 
\end{equation*}
est de cardinal $4$.
On trouve~:
\begin{eqnarray*}
|\AA^{{\rm sp}}(D^\times,\0)| 
&=& \sum\limits_{\text{$2n/\d$ impair}} \frac {1}{2} \cdot |\AA(\GL_{\d}(F),\0)| +
\sum\limits_{\text{$2n/\d$ pair, $\d\neq 1$}} \frac {1}{2} \cdot |\AA(\GL_{\d}(F),\0)| +4\\
&=& \frac {1}{2} \cdot \sum\limits_{\d} |\AA(\GL_{\d}(F),\TT)|
- \frac {1}{2} \cdot |\AA(F^\times,\0)| + 4
\end{eqnarray*} 
qui est égal à la quantité voulue.
\end{proof}

\subsection{}

On fixe une extension quadratique $K$ de $F$ incluse dans~$\overline{F}$
et un générateur de $\dd$ de $K$~sur $F$ tel que $\a=\dd^2\in\F^\times$.
On fixe un plongement de $K$ dans $A$
et on note $\tau$ l'invo\-lu\-tion $\Ad(\dd)$ de $G$. 
La notation $\ep_K$ est définie au paragraphe \ref{passageaGprime}.

\begin{theo}
\label{PTB}
Supposons que $\AA^+(G,\TT)$ soit inclus dans $\AA^{\rm sp}(G,\TT)$. 
Alors on a~:
\begin{equation*}
\pi\in\AA^{+}(G,\TT)
\quad\Leftrightarrow\quad
\text{$\pi\in\AA^{\rm sp}(G,\TT)$ et $\ep_K(\pi)=(-1)^r$.}
\end{equation*}
\end{theo}

\begin{proof} 
Notons 
$\AA^{{\rm ptb}}(G,\TT)$ le sous-ensemble de $\AA^{{\rm sp}}(G,\TT)$
formé des représentations $\pi$ telles que $\ep_K(\pi)=(-1)^r$. 
L'ensemble $\AA(G,\TT)$ est fini et les ensembles
$\AA^{{\rm ptb}}(G,\TT)$, $\AA^{{+}}(G,\TT)$ sont tous les deux inclus dans
$\AA^{{\rm sp}}(G,\TT)$. 
Par ailleurs,
le caractère central d'une~re\-présentation $\pi\in\AA^{{\rm sp}}(\G,\TT)$,
correspond \textit{via} l'homomorphisme de réciprocité de la théorie du
corps de classes local à $\det \phi$,
qui est trivial car ${\rm Sp}_{2n}(\CC)\subseteq{\rm SL}_{2n}(\CC)$. 
D'après le corollaire \ref{IsabelledeFrejusNScoro}(1), 
on en déduit~: 
\begin{equation}
\label{incl3}
\AA^{{+}} (G,\TT) \subseteq \AA^{{\rm ptb}}(G,\TT).
\end{equation}
On en déduit le résultat voulu dans le cas où $\AA^{{\rm ptb}}(G,\TT)$
est vide.
Supposons jusqu'à la fin de la preuve que $\AA^{{\rm ptb}}(G,\TT)$
n'est pas vide.
D'après le corollaire \ref{IsabelledeFrejusNScoro}(2),
l'ensemble $\AA^{{+}}(G,\TT) $ n'est pas vide et on a~:
\begin{equation}
\label{incl8}
\frac {1}{2} \cdot |\AA(G,\TT)| \< |\AA^{{+}}(G,\TT)| \< |\AA^{{\rm sp}}(G,\TT)|.
\end{equation}
Si $\TT$ est non nulle ou si $r\neq1$,
on déduit le résultat voulu du lemme \ref{comptageAsp}.
Si $\TT$~est nulle et $r=1$,
la proposition \ref{alculEPniveau0} et la remarque \ref{Parlement}
montrent que $\AA^{{\rm ptb}}(D^\times,\0)$ 
est égal à $\AA^{{\rm sp}}(D^\times,\0)$ privé des deux~ca\-rac\-tères
$\chi\circ\Nrd_{D/F}$ où $\chi$ n'est pas trivial sur $\N_{K/F}(\K^\times)$,
\ie que~:
\begin{equation*}
|\AA^{{\rm ptb}}(D^\times,\0)|=|\AA^{{\rm sp}}(D^\times,\0)|-2.
\end{equation*}
On déduit le résultat voulu du lemme \ref{comptageAsp}. 
\end{proof}

\begin{coro}
\label{PTBcoro}
On suppose que $F$ est de caractéristique nulle.
\begin{enumerate}
\item 
La conjecture \ref{CONJPTB}(2)~est vraie pour toute représentation
cuspidale de $G$. 
\item
Pour toute endo-classe autoduale $\TT$ de degré divisant $2n$,
on a~:
\begin{equation}
\label{idsauvable}
|\AA^{{+}}(G,\TT)|
= \frac 1 2 \cdot |\AA(G,\TT)|.
\end{equation}
\end{enumerate}
\end{coro}

\begin{proof}
Compte tenu du théorème \ref{PTB},
il suffit de prouver que,
si $F$ est de caracté\-ristique nulle, 
toute représentation~cuspi\-da\-le $\G^\tau$-distinguée de $G$
a un paramètre de Langlands autodual symplectique,
ce qui découle de \cite[Theorem 1.1(1)]{Xue}.
L'identité \eqref{idsauvable} suit de la preuve du théorème \ref{PTB}. 
\end{proof}

\begin{rema}
Dans le cas où $G$ est déployé et où $\pi$ est de niveau $0$,
voir \cite[Corollary~6.1]{ChommauxMatringe}
sous l'hypothèse que $F$ est de caractéristique résiduelle impaire. 
Dans le cas où le
transfert de $\pi$ à $\GL_{2n}(F)$ est cuspidal,
voir \cite[Theorems 4.1, 7.3]{Xue}.
\end{rema}

\subsection{}
\label{generic}

En guise de corollaire au théorème précédent,
donnons le théorème suivant,
qui est un~ana\-lo\-gue autodual de \cite[Theorem 10.3]{VSANT19}.
On suppose ici que $F$ est de caractéristique nulle. 
Par analogie avec \cite[Definition 1.5]{VSANT19},
on introduit la définition suivante. 

\begin{defi}
Un type $\tau$-autodual $(\BJ,\bl)$ auquel est attaché un caractère simple 
$\tau$-auto\-dual $\t\in\Cc(\aa,\b)$
sera dit \textit{générique} si l'on est dans l'un des cas suivants~:
\begin{enumerate}
\item 
$T/T_0$ est {non ramifiée}, ou $T/T_0$ est ramifiée et $c_0$ est pair,
\item
$T/T_0$ est ramifiée, $c_0$ est impair
et $\t$ est d'indice $\lfloor m/2\rfloor$.
\end{enumerate}
\end{defi}

On observera que, 
selon la proposition \ref{indiceduntype},
si une représentation cuspidale contient un~type $\tau$-autodual,
elle contient un type $\tau$-autodual générique,
et il est unique à $\G^\tau$-conjugaison près.

\begin{rema}
Cette terminologie introduite dans \cite{VSANT19} 
est justifiée par \cite[Proposition 5.5]{AKMSS},
faisant le lien entre \cite[Definition 1.5]{VSANT19} 
et la compatibilité de certains types à certaines données de Whittaker.
Je ne sais pas s'il y a un analogue de \cite[Proposition 5.5]{AKMSS} 
pour les types $\tau$-autoduaux de $G$.
\end{rema}

Le résultat suivant détermine l'image de $\Om^+$ par l'application
$\Pi$ définie au paragraphe \ref{malinlegarcon}.

\begin{theo}
\label{MAIN2}
On suppose que $F$ est de caractéristique nulle.
Une représentation cuspidale de niveau non nul 
de $\G$ est $G^\tau$-distinguée~si et~seu\-le\-ment si 
elle contient un type $\tau$-autodual générique $(\BJ,\bl)$
tel que $\bl$ soit~dis\-tingué par $\BJ\cap\G^\tau$. 
\end{theo}

\begin{proof}
Si $\pi$ contient un type $\tau$-autodual $\BJ\cap\G^\tau$-distingué, 
alors $\pi$ est autoduale et $G^\tau$-distinguée.
Inversement, supposons que $\pi$ soit une représentation cuspidale
$G^\tau$-distinguée~de $\G$. 
D'après le théorème \ref{GuoBM}, elle est autoduale.
D'après la proposition \ref{grammage},
elle contient un~ca\-rac\-tère simple $\tau$-autodual $\t$,
d'endo-classe notée $\TT$.
D'après la proposition \ref{MAIN1prop},
elle contient un type $\tau$-autodual,
que l'on peut supposer générique. 
D'après la proposition \ref{galettedepain},
$\pi$ est donc de la forme $\Pi(\BJ,\bt)$ pour un $(\BJ,\bt)\in\Om$, 
et le type $(\BJ,\bk_{*}\otimes\bt)$ est $\tau$-autodual et générique. 
Nous allons montrer que $(\BJ,\bt)\in\Om^+$.
Cela provient de ce que~:
\begin{equation*}
\frac 1 2 \cdot |[\Om]|
= |[\Om^+]|
\< |\AA^{{+}}(G,\TT)|
= \frac 1 2 \cdot |\AA(G,\TT)|
= \frac 1 2 \cdot |[\Om]|
\end{equation*}
(l'égalité centrale provenant du corollaire \ref{PTBcoro})
donc l'image de $\Om^+$ par $\Pi$ est exactement égale~à $\AA^{{+}}(G,\TT)$. 
\end{proof}

\subsection{}

Supposons maintenant que $r$ soit pair,
qu'on écrit $r=2k$,
et que $\tau$ soit l'automorphisme de conjugaison par un élément $\s\in\G$ 
tel que $\s^2=1$,
de polynôme caractéristique réduit $(X^2-1)^n$.
De façon analogue au théorème \ref{PTB},
on a le résultat suivant. 

\begin{prop}
\label{PTBsplitlevi}
Soit $\TT$ une endo-classe autoduale, de degré divisant $2n$.
Supposons que $\AA^+(G,\TT)\subseteq\AA^{\rm sp}(G,\TT)$. 
Alors $\AA^{+}(G,\TT)=\AA^{\rm sp}(G,\TT)$. 
\end{prop}

\begin{proof} 
D'après le corollaire \ref{tolkienelrond}(1) et la proposition 
\ref{OlivierdeNoyen}, on a~:
\begin{equation}
\label{incl81}
\frac {1}{2} \cdot |\AA(G,\TT)| \< |\AA^{{+}}(G,\TT)| \< |\AA^{{\rm sp}}(G,\TT)|.
\end{equation}
Comme $r\neq1$,
on déduit le résultat voulu du lemme \ref{comptageAsp}. 
\end{proof}

\subsection{}
\label{LAPREUVE11}

Dans ce paragraphe, $G$ est égal à $\GL_{2n}(F)$.
La proposition \ref{PTBsplitlevi} prend la forme plus précise suivante. 

\begin{theo}[\cite{JNQ} Theorem 1.1]
\label{THMJNQ}
On suppose que $F$ est de caractéristique nulle.
Une représentation cuspidale autoduale
de $\GL_{2n}(F)$ est de parité symplectique si et seulement si
elle est distinguée par $\GL_n(F)\times\GL_n(F)$.
\end{theo}

Soit $\pi$ une représentation cuspidale autoduale de $\GL_{2n}(F)$
de niveau non nul.
Il y a une unique représentation cuspidale autoduale de $\GL_{2n}(F)$
inertiellement équivalente mais non isomorphe à $\pi$~;
notons-là $\pi^*$.
Soit $\TT$ l'endo-classe de $\pi$,
soit $T/T_0$ l'extension quadratique qui lui est associée et posons
$m=2n/\deg(\TT)$.
On définit $\s$ comme au paragraphe \ref{P12}.

{Selon \cite[Proposition~6.6]{BHS},
les représentations $\pi$, $\pi^*$ ont même parité
si et seulement si $T/T_0$ est ramifiée et $m=1$,
auquel cas \cite[6.8]{BHS} montre comment déterminer
cette parité en termes de types.}
Dans les autres cas, \ie
{(d'après le lemme \ref{lemmedeparite})}
si $T/T_0$ est non ramifiée, 
ou si $T/T_0$ est ramifiée et $m$ est pair,
les représentations~$\pi$ et $\pi^*$ ont~des parités différentes,
et il s'agit de déterminer laquelle des deux est
de parité sym\-plec\-tique en~ter\-mes de types. 
Nous donnons ci-dessous une réponse à cette question. 

D'après le corollaire \ref{MAIN1}, 
la représentation $\pi$ contient un type $\s$-autodual 
$(\BJ,\bl)$, 
qu'on peut supposer générique,
auquel cas il est unique à $\G^\s$-conjugaison près.
Notons $\t$ le caractère simple qui lui est attaché et~$\n$ la représentation de
Heisenberg de $\t$.
D'après les corollaires \ref{exkappadistaddetp} et \ref{exkappadistaddetpnr}, 
il existe une~unique~re\-présentation
$\bk_{*}$ de $\BJ$ prolongeant $\n$,
qui soit à la fois $\s$-autoduale et $\BJ\cap\G^\s$-distinguée
et dont le dé\-terminant soit
d'ordre une puissance de $p$.
Soit $\bt$ l'unique représentation de $\BJ$ triviale sur $\BJ^1$
telle que $\bl$ soit isomorphe à $\bk_{*}\otimes\bt$.

Notons $\EuScript{Z}$ le centre de $\BJ/\BJ^1$.
Si $[\aa,\b]$ est une strate simple $\s$-autoduale dans $\Mat_{2n}(F)$
telle~que $\t\in\Cc(\aa,\b)$, il est égal à $E^\times\BJ^1/\BJ^1$,
avec $E=F[\b]$.

\begin{prop}
\label{critsymplec}
On suppose que $F$ est de caractéristique nulle.
Supposons que $T/T_0$ soit ramifiée. 
La représentation cuspidale autoduale de niveau non nul $\pi$ est de
parité symplectique si et seulement si~: 
\begin{enumerate}
\item 
$\bt$ est un caractère non ramifié si $m=1$,
\item
la restriction de $\bt$ à $\EuScript{Z}$ est non triviale si $m$ est pair.
\end{enumerate}
\end{prop}

\begin{rema}
\begin{enumerate}
\item
Dans le cas où $m=1$,
on retrouve le résultat de \cite[6.8]{BHS},~\ie
que $\pi$ est de parité symplectique si et seulement si
elle contient la restriction de la re\-présentation $\bk_{*}$
au sous-groupe compact $\BJ^0$.
\item
Dans le cas où $m$ est pair,
le caractère central de $\bt$ sur $\EuScript{Z}$ est toujours un caractère
non ramifié d'ordre au plus $2$.
\end{enumerate}
\end{rema}

\begin{proof}
En vertu du théorème \ref{THMJNQ}, 
la représentation $\pi$ est de parité symplectique si et seulement si elle
est distinguée par $\GL_n(F)\times\GL_n(F)$,
qui est conjugué à $\G^\s$ dans 
$G$.~Raison\-nant comme dans la preuve du théorème \ref{MAIN2}, 
on voit que c'est le cas si et seulement si le type géné\-ri\-que
$\bk_{*}\otimes\bt$ est distingué par $\BJ\cap\G^\s$,
\ie~si~et seulement si $\bt$ est $\BJ\cap\G^\s$-distinguée,
car $\bk_{*}$ est $\BJ\cap\G^\s$-distinguée.
La repré\-sen\-ta\-tion $\bt$ étant $\s$-autoduale,
elle est $\BJ^0\cap\G^\s$-dis\-tinguée. 
Le résultat suit de \cite[Proposition 6.3]{HMIMRN} 
(voir également \cite[Lem\-ma 8.2]{VSANT19}).
\end{proof}

Dans le cas où $T/T_0$ est non ramifiée,
posons $E_0=\F[\b^2]$ et $\EuScript{Z}_0=E_0^\times\BJ^1$.
C'est un sous-groupe du centre $\EuScript{Z}$ qui ne dépend
pas du choix de la strate $[\aa,\b]$.
En effet, si $[\aa,\b']$ en est une autre,
on a $E'^\times\BJ^1=E^\times\BJ^1$,
donc il y a un $x\in\BJ^1$ tel que $\b'=\b x$.
De façon similaire à la proposition \ref{critsymplec},
en remplaçant \cite[Pro\-po\-sition 6.3]{HMIMRN}
par \cite[Lemma 9.11]{VSANT19}, on obtient~:

\begin{prop}
\label{critsymplecnr}
On suppose que $F$ est de caractéristique nulle.
Supposons que $T/T_0$~soit non ramifiée. 
La représentation cuspidale auto\-duale de niveau non nul $\pi$ est de
parité symplecti\-que si et seulement si la restriction 
de $\bt$ à $\EuScript{Z}_0$ est~tri\-viale. 
\end{prop}

\section{Involutions galoisiennes sur les formes intérieures 
  de $\GL_{n}(F)$}
\label{tarteauxpommes}

Nous considérons maintenant 
une situation différente de celle introduite à la section~\ref{SEC4} 
et~indi\-quons comment les méthodes développées dans les
sections \ref{SEC4} et \ref{SEC5} peuvent s'y appliquer.

\subsection{}
\label{fablab}

Fixons une extension quadratique $F/F_0$,
une $F_0$-algèbre centrale simple $A_0$ de degré~ré\-duit $n$
et posons $A=A_0\otimes_{F_0}F$.
C'est une $F$-algèbre centrale simple de degré réduit $n$,
qu'on munit~de l'action naturelle
de $\Gal(F/F_0)$ dont on note $\s$ le générateur.
Posons $G=A^\times$. 
Nous allons~mon\-trer que l'étude des représentations cuspidales
de $G$ distinguées par $\G^\s$ au moyen de la théorie des types offre
beaucoup de similarités avec celle traitée dans cet article.
Observons que le cas où $A_0$ est égale à $\Mat_n(F_0)$,
\ie le cas des représentations~cus\-pi\-dales de $\GL_n(F)$
distinguées par $\GL_n(F_0)$,
a été traité dans \cite{AKMSS,VSANT19}.

\subsection{}

Fixons un $\a\in\F_0^\times$ et un $\dd\in\G$ tel que $\dd\s(\dd)=\a$.
Ceci définit l'involu\-tion $\tau=\Ad(\dd)\circ\s$ de $G$.
La $F$-algèbre centrale simple $A$ munie de l'involution $\tau$
entre dans le cadre du paragraphe \ref{Biotabstait}
et notamment du lemme \ref{bilbon1}.
Par conséquent, $A^\tau$ est une $F_0$-algèbre centrale simple~et
$A$ s'identifie à $A^\tau\otimes_{F_0}\F$,
\ie que la paire $(A^\tau,A)$ relève du cadre fixé au paragraphe \ref{fablab}.
Nous ne gagnons donc aucune généralité à passer de $\s$ à $\tau$. 

\subsection{}

Une représentation irréductible $\pi$ de $\G$ est dite $\s$-\textit{autoduale} 
si sa contragrédiente $\pi^\vee$~est iso\-morphe à sa conjuguée $\pi^\s$. 
Nous allons voir que,
contraire\-ment à ce qui se passe dans~le~cas autodual,
une représentation cuspidale $\s$-autoduale de $G$
contient \textit{toujours} un caractère simple ma\-xi\-mal
$\s$-autodual.
D'abord, 
c'est vrai dans le cas où $A_0=\Mat_n(F_0)$
grâce à \cite[Theorem~4.1, 4.2]{AKMSS}.
Soit $\pi$ une~représen\-ta\-tion 
cuspidale $\s$-autoduale de $G$
et soit $\TT$ son endoclasse.
{Nous allons sui\-vre 
la preuve du théorème \ref{THMEXISTENCECARACSIMPLE}.}
D'abord, il existe dans $\GL_n(F)$
un caractère sim\-ple ma\-xi\-mal $\s$-autodual
$\t_1\in\Cc(\aa_1,\b)$ d'endo-classe $\TT$,~et
on peut supposer,
d'après \cite[Corollary~4.21]{AKMSS},
que~la strate simple $[\aa_1,\b]$ 
est $\s$-autoduale.
Posons $E=F[\b]$ et $E_0=E^\s$.
D'après \cite[Remark~4.22]{AKMSS},
l'ho\-mo\-morphisme naturel $E_0\otimes_{F_0}F\to\E$
est un isomorphisme.
Pour plonger $E$ dans $A$ de façon que $\s$ induise sur $E$
l'automorphisme non trivial de $E/E_0$, 
il suffit donc de plonger $E_0$ dans $A_0$ 
comme une $F_0$-algèbre, 
ce qui est toujours possible au vu des degrés réduits,
puis d'en déduire un plongement de $F$-algèbres de $E$ dans $A$
par extension de $F_0$ à $F$.
Fixons un tel plongement,
et notons $B$ le centralisateur de $E$ dans $A$.

Pour comprendre l'analogie entre les cas autodual et $\s$-autodual,
il faut comprendre que~$B$~se comporte ici exactement comme au
paragraphe \ref{Biotabstait}~:
c'est une $E$-algèbre centrale simple munie de la $E_0$-involution
$\s$.
Par conséquent,
les conditions d'existence d'un ordre maximal de~$B$~stable par $\s$
sont exactement les mêmes que dans le cas autodual.
Pour~adapter au cas $\s$-autodual~le lemme~\ref{fromagequipue},
il suffit de prouver,
en reprenant la preuve du lemme \ref{lemmedeparite},
que pour qu'il existe~une
re\-pré\-sentation irréductible cus\-pidale $\s$-auto\-duale de 
$\GL_n(F)$ 
d'endo-classe $\TT$,
il faut et suffit~que l'entier $m=n/\deg(\TT)$ soit
impair si $\E/\E_0$ est non ramifiée, 
et pair ou égal à $1$ si $\E/\E_0$ est~ra\-mi\-fiée.~(On
observera que, contrairement au cas autodual,
le cas où $\TT$ est nulle n'est pas traité à part.)
Pour finir d'adapter la preuve du théorème \ref{THMEXISTENCECARACSIMPLE},
il ne reste donc qu'à adapter le lemme \ref{lemmedetransfertautodual}
au cas $\s$-autodual.

\subsection{}

La $E$-algèbre $B$ se comportant de la même façon
qu'on soit dans le cas $\s$-autodual ou~au\-to\-dual,
la classification des caractères simples $\s$-autoduaux 
contenus dans une représentation~ir\-ré\-ductible
cuspidale $\s$-autoduale de $G$,
ainsi que la description de l'action résiduelle de $\s$ sur~le groupe
$\BJ^0(\aa,\b)/\BJ^1(\aa,\b)\simeq\GL_m(\ee)$,
est exactement celle donnée par la proposition \ref{indiceduntype}.
Il~s'en\-suit que toute l'analyse faite aux paragraphes
\ref{viteA} à \ref{malinlegarcon} reste valable,
quitte à remplacer le~théo\-rème \ref{GuoBM}
par une version non déployée de \cite[Theorem 4.1]{VSANT19}.
Par conséquent,
si $\pi$ est une~repré\-sentation cuspidale $\s$-autoduale de $G$,
si $\t$ est un caractère simple $\s$-autodual générique contenu dans $\pi$
et si $(\BJ,\bl)$ est un type contenu dans $\pi$ auquel $\t$
est attaché, alors~:
\begin{itemize}
\item 
le type $\bl$ est $\s$-autodual,
\item
il existe une unique représentation $\bk_*$ de $\BJ$ prolongeant
la représentation de Heisenberg de $\t$,
étant à la fois $\s$-autoduale et $\BJ\cap G^\s$-distinguée,
et telle que $\det \bk_*$ soit d'ordre une puissance de $p$,
\item
et $\bl$ s'écrit $\bk_*\otimes\bt$ pour une unique représentation 
irréductible $\s$-autoduale $\bt$ de $\BJ$ triviale sur $\BJ^1$.
\end{itemize}
En outre,
si cette représentation $\bt$ est $\BJ\cap G^\s$-distinguée,
alors $\pi$ est $G^\s$-distinguée. 
Pour prouver que la réciproque est vraie
et obtenir du même coup un théorème~de dichotomie et de
disjonction (\cite{Kable,AKT}), 
il faut analyser la contribution de toutes les
doubles classes dans \eqref{girafeintro}. 
Ceci peut être fait en suivant l'approche de 
\cite[Section 6]{VSANT19}. 

\subsection{}

Beuzart-Plessis a montré (\cite[Theorem 1]{BPINVENT18})
qu'une représentation cuspidale
(et plus~géné\-ra\-le\-ment une représentation essentiellement de carré intégrable)
de $G$ est $\G^\s$-distinguée si et~seu\-lement si son transfert de
Jacquet-Langlands au groupe $\GL_n(F)$ est $\GL_n(F_0)$-distingué. 
Il serait intéressant de savoir si un tel résultat est accessible par les
méthodes présentées ici. 

\providecommand{\bysame}{\leavevmode ---\ }

\end{document}